\documentclass[12pt, a4paper]{amsart}
\usepackage{amscd, amssymb, pb-diagram}
\usepackage{mathtools}
\usepackage{mathrsfs}
\usepackage{enumerate}
\usepackage[shortlabels]{enumitem}
\usepackage{amsmath}
\usepackage[margin=2.9cm]{geometry}
\usepackage{amscd, amssymb, pb-diagram}
\usepackage{tikz-cd}
\usepackage{amsthm} 
\usepackage[nospace,noadjust]{cite}
\usepackage{lmodern}
\usepackage{graphicx}
\graphicspath{{images/}{../images/}}
\usepackage[toc,page]{appendix}
\usepackage[normalem]{ulem}
\usepackage{emptypage}
\usepackage{hyperref}
\usepackage{subfiles}
\usepackage{datetime}

\usepackage{tensor}
\usepackage{wrapfig}
\usepackage{tikz}
\usepackage{multicol}
\usepackage[all]{xy}
\usepackage[autostyle]{csquotes}

\allowdisplaybreaks

\def\Dom{\operatorname{Dom}}

\def\ker{\operatorname{Ker}}

\def\id{\operatorname{id}}

\def\sup{\operatorname{sup}}
\def\max{\operatorname{max}}
\def\min{\operatorname{min}}

\def\Im{\operatorname{Im}}

\def\cl{\operatorname{cl}}
\def\diam{\operatorname{diam}}
\def\Lip{\operatorname{Lip}}

\def\Ha{\operatorname{\mathcal{H}_{\alpha}}}
\def\Hb{\operatorname{\mathcal{H}_{\beta}}}
\def\M{\operatorname{M}}
\def\m{\operatorname{m}}

\def\K{\operatorname{K}}
\def\R{\operatorname{R}}
\def\Ex{\operatorname{E}}
\def\F{\operatorname{F}}
\def\A{\operatorname{A}}

\def\df{\operatorname{d_f}}
\def\dw{\operatorname{d_w}}
\def\ds{\operatorname{d_s}}

\def\Xint#1{\mathchoice
   {\XXint\displaystyle\textstyle{#1}}%
   {\XXint\textstyle\scriptstyle{#1}}%
   {\XXint\scriptstyle\scriptscriptstyle{#1}}%
   {\XXint\scriptscriptstyle\scriptscriptstyle{#1}}%
   \!\int}
\def\XXint#1#2#3{{\setbox0=\hbox{$#1{#2#3}{\int}$}
     \vcenter{\hbox{$#2#3$}}\kern-.5\wd0}}

\def\dashint{\Xint-}

\newcommand{\Hol}{\textup{H\"ol}}

\newcommand{\D}{\mathcal{D}}
\newcommand{\C}{\mathcal{C}}

\newcommand{\Da}{\Delta_{\alpha}}

\newcommand{\cB}{\overline{B}}

\newcommand{\E}{\mathcal{E}_{\alpha}}

\newcommand{\Eba}{\mathcal{E}_{\beta(\alpha)}}
\newcommand{\Hba}{\mathcal{H}_{\beta(\alpha),x_0}}
\newcommand*{\dd}{\mathop{}\!\mathrm{d}}
\newcommand{\Q}{\mathcal{Q}}
\newcommand{\Sp}{\mathcal{S}_p}

\makeatletter
\def\Ddots{\mathinner{\mkern1mu\raise\p@
\vbox{\kern7\p@\hbox{.}}\mkern2mu
\raise4\p@\hbox{.}\mkern2mu\raise7\p@\hbox{.}\mkern1mu}}
\makeatother

\numberwithin{equation}{section}

\tikzstyle{vertex}=[circle]
\tikzstyle{goto}=[->,shorten >=1pt,>=stealth,semithick]

\newtheorem{thm}{Theorem}[section]
\newtheorem{cor}[thm]{Corollary}
\newtheorem{lemma}[thm]{Lemma}
\newtheorem{prop}[thm]{Proposition}
\newtheorem*{thm1}{Theorem A}
\newtheorem*{thm2}{Theorem B}
\newtheorem*{thm3}{Theorem C}
\newtheorem*{thm4}{Theorem D}

\theoremstyle{definition}
\newtheorem{definition}[thm]{Definition}

\theoremstyle{remark}
\newtheorem{remark}[thm]{Remark}

\newtheorem{example}[thm]{Example}
\newtheorem*{Acknowledgements}{Acknowledgements}

\makeatletter
\@namedef{subjclassname@2020}{%
  \textup{2020} Mathematics Subject Classification}
\makeatother
\begin{document}

\begin{abstract}
We develop a novel framework for Monge–Kantorovič metrics using Schatten ideals and commutators of fractional Laplacians on Ahlfors regular spaces. Notably, for those metrics we derive closed formulas in terms of spectra of higher-order fractional Laplacians. For our proofs we develop new techniques in noncommutative geometry, in particular a Weyl law and Schatten-class commutators, yielding refined quantum metrics on the space of Borel probability measures. Lastly, our fractional analysis extends to dynamical systems. We showcase this in the setting of expansive algebraic $\mathbb Z^m$-actions and homoclinic $C^*$-algebras of certain hyperbolic dynamical systems. These findings illustrate the versatility of fractional analysis in fractal geometry, dynamical systems and noncommutative geometry. 
\end{abstract}

\title[Ideal quantum metrics from fractional Laplacians]{Ideal quantum metrics from fractional Laplacians}
\author[D.M. Gerontogiannis]{Dimitris M. Gerontogiannis}
\address{\scriptsize Dimitris Michail Gerontogiannis, IM PAN,  Jana i Jedrzeja \'Sniadeckich 8, 00-656 Warszawa, Poland}
\email{dgerontogiannis@impan.pl}
\author[B. Mesland]{Bram Mesland}
\address{\scriptsize Bram Mesland, Leiden University, Niels Bohrweg 1, 2333 CA Leiden, The Netherlands}
\email{b.mesland@math.leidenuniv.nl}
\keywords{Fractional Laplacian, Monge–Kantorovi\v{c} metrics, Ahlfors regular, dyadic analysis, quantum metric spaces, crossed products, algebraic dynamical systems}
\subjclass[2020]{31C25, 30L99, 46L87, 46L30, 37A55}

\maketitle
\section*{Introduction}

Monge--Kantorovi\v{c} metrics, pivotal in optimal transport \cite{Vil} and quantum state geometry \cite{BZ}, have driven substantial advancements in the understanding of classical and quantum spaces. The space of probability measures on a metric space comes equipped with the weak $*$-topology, and the problem we address in this paper is that of constructing computable Monge–Kantorovi\v{c} metrics on the space of probability measures that are compatible with dynamics.

Specifically, we present a novel, uniform construction of Monge--Kantorovi\v{c} metrics for compact metric spaces $(X,d)$ with an Ahlfors regular measure $\mu$. At the heart of the method lies the fractional Dirichlet Laplacian $\Delta_{\alpha}$, the generator of the $\alpha$-stable L\'evy process on $X$. The operator $\Delta_{\alpha}$ is an intrinsic, unbounded self-adjoint operator on $L^{2}(X,\mu)$. The key property is that for $0<2\alpha<\beta \leq 1$, the operator $\Delta_{\alpha}$ commutes with the representation of $\beta$-H\"older continuous functions,
\[\mathrm{m}:\Hol_{\beta}(X,d)\to B(L^{2}(X,\mu)),\quad \mathrm{m}_{f}\varphi:=f\varphi,\] modulo the Schatten $p$-ideal $\Sp(L^2(X,\mu))$, for every $p>p(\alpha,\beta)>0$. Using dyadic analysis (as developed by Hyt\"onen--Kairema in \cite{HK}), NWO-sequences for Ahlfors regular metric-measure spaces, fractional Sobolev embeddings, and analysis of the spectrum and domain of the fractional Dirichlet Laplacian, we prove the cornerstone result of the present paper.
\begin{thm1}[{Theorem \ref{thm:SchattenCQMS}}] Let $0<2\alpha<\beta \leq 1$. Then, for every $p>\max \{p(\alpha,\beta),1\}$, the weak $*$-topology on the space of Borel probability measures on X is generated by the metric $$\rho_{\alpha,\beta,p}(\varphi,\psi):=\sup \{|\varphi(f)-\psi(f)|: f\in \Hol_{\beta}(X,d),\,\, \|\cl([\Da,\m_f])\|_{\Sp}\leq 1\}.$$ In particular, there are uniform constants $0<c(\alpha,\beta,p)<C(\alpha,\beta,p)$, so that for any $x,y\in X$ the distance of the Dirac-masses $\delta_x$ and $\delta_y$ satisfies $$c(\alpha,\beta,p)d(x,y)^{\beta}\leq \rho_{\alpha,\beta,p}(\delta_x, \delta_y)\leq C(\alpha,\beta,p)d(x,y)^{2\alpha}.$$ The constant $c(\alpha,\beta,p)$ is bounded with respect to $p$, whereas $C(\alpha,\beta,p)\to \infty$ as $p\to \infty$.
\end{thm1}

The method of constructing Monge--Kantorovi\v{c} metrics from commutators with an unbounded self-adjoint operator is widespread in the literature of Compact Quantum Metric Spaces (CQMS), see Section \ref{sec:ISMS}. It was pioneered by Connes \cite{Connes2} and Rieffel \cite{Rie,Rie2, Rie3} in the 1990's for Dirac operators on spin manifolds. These constructions always employ the operator norm of the commutator, as the commutator with the Dirac operator is bounded but never compact (unless the manifold consists of finitely many points).

The novelty of Theorem A is twofold: it applies to any Ahlfors regular metric-measure space and it uses the Schatten $p$-norm of the commutator. The latter gives rise to the more refined notion of an \emph{Ideal Spectral Metric Space} (see Definition \ref{def:p-spectral}), where the commutators define operators in a symmetrically normed ideal, allowing for a finer analysis. In this paper we exploit this new way of generating CQMS and show that it allows for
\begin{itemize}
\item[i.] the derivation of closed analytic formulae for Monge--Kantorovi\v{c} metrics;
\item[ii.] the construction of dynamical Monge--Kantorovi\v{c} metrics.
\end{itemize}

To derive our explicit formula, we specialise to Ahlfors spaces having walk dimension $\dw$ greater than the fractal dimension $\df$, which is often a manifestation of the existence of recurrent diffusion processes. This is still a fairly large class of Ahlfors regular metric-measure spaces. For those, the aforementioned commutators become Hilbert--Schmidt. Namely, for $0<2\alpha<\min\{(\dw-\df)/2,1\}$, let $\beta(\alpha)=\df/2+2\alpha$, and write $\mathcal{E}_{\beta(\alpha)}$ for the Dirichlet form defining $\Delta_{\beta(\alpha)}$. Under the imposed conditions, $\Dom \Eba$ is a dense $*$-subalgebra of $C(X)$; the space of continuous functions on $X$ equipped with the supremum norm. This property allows us to refine the metrics of Theorem A to a Monge--Kantorovi\v{c} metric 
\[\rho_{\alpha}(\varphi,\psi):=\sup \{|\varphi(f)-\psi(f)|: f\in \Dom \Eba,\,\, \|\cl([\Da,\m_f])\|_{\mathcal{S}_2}\leq 1\},\]
generating the weak $*$-topology. It is this metric that admits an explicit, closed formula.
\begin{thm2}[{Theorem \ref{thm:Connes_distance}}]
$$\rho_{\alpha}(\varphi,\psi)=\left(\sum_{n\in \mathbb N} \lambda^{-1}_{\beta(\alpha),n} |(\varphi-\psi)(h_{\beta(\alpha),n})|^2\right)^{1/2}.$$
\end{thm2}
Here $h_{\beta(\alpha),n}$ are the eigenfunctions of $\Delta_{\beta(\alpha)}$ associated to the eigenvalues $\lambda_{\beta(\alpha),n}>0$, counting multiplicity. To our knowledge, this formula is new.

Further, we show how our fractional analysis extends to dynamical systems. This requires the language of crossed product $C^*$-algebras. Recently, a general recipe for constructing CQMS structures on crossed product $C^*$-algebras was developed by Austad--Kaad--Kyed \cite{AKK}. The gist is to find several new examples where this can be applied. Our work does exactly that. We illustrate the strength of our fractional analysis in the large class of expansive algebraic $\mathbb Z^m$-actions and the hyperbolic dynamical systems called Smale spaces. These are chaotic dynamical systems for which no CQMS structure had been constructed before, let alone ideal spectral metric space structure. The precise statements of our dynamical results require a bit more terminology, and can be found in Theorems C and D of this Introduction. We now proceed with a more detailed discussion of our results.

\subsection*{Metrics from Schatten-class commutators}\label{sec:Intro1}
In this paper we study compact metric spaces $(X,d)$ with a finite Ahlfors $\df$-regular measure $\mu$, that is, the $\mu$-volume of every ball in $X$ of radius $r$ is of order $r^{\df}$. Here $\df$ is the fractal dimension of $(X,d)$. These include compact Riemannian manifolds, several fractals \cite{Fal}, self-similar Smale spaces \cite{Ger1}, and several limit sets of hyperbolic isometry groups and Gromov boundaries \cite{Cor}, among other examples. Although our study includes the aforementioned manifolds, considering local operators as in \cite{Connes2} is not a viable option, since Ahlfors regular metric-measure spaces encompass several totally disconnected spaces, e.g. shift spaces, Gromov boundaries of free groups. Instead, we need to consider non-local operators with interesting spectral properties.

For this reason we focus on the fractional Dirichlet Laplacian, which is the positive, self-adjoint generator $\Da$ of the Dirichlet form
\begin{equation*}
\E(f,g):=\frac{1}{2}\int_X\int_X \frac{(f(x)-f(y))(g(x)-g(y))}{d(x,y)^{\df+2\alpha}}\dd\mu(y)\dd\mu(x).
\end{equation*}
The operator $\Da$ is reminiscent of the $\alpha$-power of the Laplacian on a Euclidean space, and is ubiquitous in mathematics and several other scientific areas. The domain $\Dom \E$ is the $\alpha$-fractional Sobolev space, see Section \ref{sec:Lap} for details. Our analysis reveals new properties of $\Da$ that can be used to study Monge--Kantorovi\v{c} metrics and dynamics. Moreover, it is typical that $\Da$ has nicer properties when $0<\alpha<1/2,$ see Subsection \ref{sec:anal_domain}. Nevertheless, for special cases, that we will describe soon, it makes sense to consider $\alpha$ to be well above $1/2$. For more details see Subsection \ref{sec:low_spectral_dim}.

A key observation that kick-starts our work is that, whenever $0<2\alpha<\beta \leq 1$, the commutator of $\Da$ and the multiplication operator $\m_f$ (with $f\in \Hol_{\beta}(X,d)$) is an integral operator in the Schatten $p$-ideal $\Sp(L^2(X,\mu))$ for $p>\max \{p(\alpha,\beta),1\}$. Here, $\Hol_{\beta}(X,d)$ is the $*$-algebra of $\beta$-H\"older continuous functions and the constant $p(\alpha,\beta)>0$ tends to $\infty$ as $2\alpha$ and $\beta$ get closer together, see Proposition \ref{prop:Schatten_com}. Already this property shows that $\Da$ on a compact (spin) Riemannian manifold behaves quite differently from the Dirac operator in \cite{Connes2}, whose commutator with non-trivial functions can never be compact.

It is only natural then to consider the finer Schatten norm of the closure $\cl([\Da,\m_f])$ instead of the operator norm for defining a metric on the state space $S(C(X))$ (equivalently, space of Borel probability measures). Here $C(X)$ is the $C^*$-algebra of continuous functions on $X$. More importantly, those Schatten norms ought to reflect the geometry on $(X,d,\mu)$ imposed by $\Da$. This is indeed the case, as exemplified by Theorem A.

As far as we know, Theorem A is completely new. The first reason for implementing $\Da$ in the study of Monge--Kantorovi\v{c} metrics is its wide applicability, since as a non-local diffusion operator it makes sense even in totally disconnected situations. The second reason is that up until now, there was no example of an intrinsically defined operator on a level as general as that of Ahlfors regular metric-measure spaces, which encodes the metric structure (see Rieffel's \cite[Problem 4.1]{Rie3}). For instance, Triebel's singular value estimates \cite{T} for compact embeddings of Besov spaces on fractals in Euclidean space and Assouad's Embedding Theorem \cite{A}, immediately lead to a Weyl law for $\Da$, see Corollary \ref{cor:Weyl_law}.

Also, in noncommutative geometry \cite{Connes} the triple $(\Hol_{\beta}(X,d), L^2(X,\mu),\Da)$ forms a spectral triple, see Definition \ref{def:spectral_triple}. Here we point out that it is remarkable that there is a geometrically defined spectral triple with compact, let alone Schatten-class, commutators. Spectral triples with compact commutators can be constructed through functional calculus \cite{GMR}, or brute-force \cite{BJ}, but such constructions do not encode a meaningful geometric structure. 

The third and more important reason is that the refined Monge--Kantorovi\v{c} metric $\rho_{\alpha,\beta,p}$ is sensitively dependent on the spectrum of each $\cl([\Da,\m_f])$. So, a fine understanding of the latter leads to an understanding of $\rho_{\alpha,\beta,p}$. This is the point where the heavy machinery of dyadic analysis and maximal functions comes into play. Dyadic analysis provides powerful tools for geometric analysis on metric-measure spaces in the absence of Fourier theory or algebro-geometric structure, see Subsection \ref{sec:Dyadiccubes}.

The proof of Theorem A relies on establishing the bounds 
\begin{equation}\label{eq:intro0}
C(\alpha,\beta,p)^{-1}\Hol_{2\alpha}(f)\leq \|\cl([\Da,\m_f])\|_{\Sp} \leq c(\alpha,\beta,p)^{-1} \Hol_{\beta}(f),
\end{equation}
in terms of H\"older seminorms. The most involved is the lower bound which allows to derive equicontinuity for the family $f\in \Hol_{\beta}(X,d)$ with $\|\cl([\Da,\m_f])\|_{\Sp}\leq 1$. This is estimated by starting with a fine analysis via NWO-sequences (see Rochberg--Semmes \cite{RS}). Surprisingly, this technique has been used in the literature only in the Euclidean setting, although it can easily be made intrinsic to Ahlfors regular metric-measure spaces, see Subsection \ref{sec:RS-tech}. Then, those estimates together with the fractional Sobolev embedding \cite{GS} into H\"older spaces lead to \eqref{eq:intro0}. This means $\rho_{\alpha,\beta,p}$ is not only an extended metric, but an actual metric generating the weak $*$-topology. For all these we refer to Subsection \ref{sec:Schatten_com}. At this point we should highlight that our results and treatment of $\Da$ are intrinsic to $(X,d,\mu)$, and only the constants in inequalities such as those in Theorem A might depend on the dyadic analysis.

\begin{remark}
An interesting question is if the constant $C(\alpha,\beta,p)$ of Theorem A can be optimised to the point it becomes bounded with respect to $p$. Then, it would be tempting to consider the metric $\rho_{\alpha,\beta,\infty}:=\lim_{p\to \infty} \rho_{\alpha,\beta,p}$ which would also generate the weak $*$-topology.
\end{remark}

Summarising the discussion so far, we can say that metrics as that of Theorem A on the state space $S(A)$ of a $C^*$-algebra $A$, equip $A$ with the structure of an $\Sp$\textit{-spectral metric space}. Roughly, such metrics are obtained from pairs $(\mathcal{A},L)$ where $\mathcal{A}$ is a dense $*$-subalgebra of $A$ and $L:\mathcal{A}\to \mathbb [0,\infty)$ is a seminorm defined as the Schatten $p$-norm of the closure of the commutator of an unbounded self-adjoint operator such that $\ker L=\mathbb C 1$, and the associated \textit{Monge--Kantorovi\v{c} metric} 
\begin{equation}\label{eq:intro1}
\rho_L(\varphi,\psi):=\sup \{|\varphi(a)-\psi(a)|: a\in \mathcal{A},\,\,L(a)\leq 1\}
\end{equation}
generates the weak $*$-topology of $S(A)$. One can go even further to define $\mathcal{I}$\textit{-spectral metric spaces} where $\mathcal{I}$ is a symmetrically normed ideal of compact operators on a separable Hilbert space. For the formal definition we refer to Definition \ref{def:p-spectral}. 

\begin{remark}
Recently a generalisation of CQMS to the context of $L^p$-operator algebras was put forward in \cite{DFP}. As far as we can see, this generalisation and our concept of ideal spectral metric spaces are unrelated. For instance, for $p=2$ in \cite{DFP} one recovers ordinary spectral triples (i.e. the commutators only need to be bounded), whereas our $\mathcal{S}_2$-spectral metric space is a new notion.
\end{remark}

\subsection*{Walk dimension greater than the fractal dimension}\label{sec:Intro2}

The case of $\mathcal{S}_2$-spectral metric spaces is particularly nice. For Theorem A, this would imply that the commutators $\cl([\Da,\m_f])$ are Hilbert--Schmidt. However, as we see from Proposition \ref{prop:Schatten_com} this can happen only if the fractal dimension $\df$ is very small. The issue for larger $\df$ comes from the fact that we consider commutators of $\Da$ with every $\m_f$, where $f\in \Hol_{\beta}(X,d)$ and $2\alpha<\beta \leq 1$. Instead, one has to consider potentially smaller dense $*$-subalgebras of $C(X)$.

Such $*$-algebras can be found if we impose a constraint on the walk dimension $\dw$ of $(X,d)$. The walk dimension can be viewed as an exponent governing the mean exit time of diffusion processes, see Definition \ref{def: walkdim}. In Proposition \ref{prop:domain_module_Sob} we show that whenever $\dw >\df$ then, for $$0<2\alpha<\min\{(\dw-\df)/2,1\},\qquad \beta(\alpha)=\df/2+2\alpha,$$ the domain $\Dom \Eba$ is in fact a dense $*$-subalgebra of $C(X)$ and $\|\cl([\Da,\m_f])\|_{\mathcal{S}_2}<\infty$ for every $f\in \Dom \Eba$. Moreover, we obtain a metric $\rho_{\alpha}$ on $S(C(X))$ given by $$\rho_{\alpha}(\varphi,\psi):=\sup \{|\varphi(f)-\psi(f)|: f\in \Dom \Eba,\,\, \|\cl([\Da,\m_f])\|_{\mathcal{S}_2}\leq 1\}$$ generating the weak $*$-topology. Then, since $\Dom \Eba$ is a Hilbert space, and thus self-dual, we manage to obtain a closed formula for $\rho_{\alpha}$. Specifically, consider the generator $\Delta_{\beta(\alpha)}$ of $\Eba$, with sequence of (positive) eigenvalues $(\lambda_{\beta(\alpha),n})_{n\geq 0}$ in increasing order, counting multiplicities, and an associated orthonormal eigenbasis $(h_{\beta(\alpha),n})_{n\geq 0}$. Note that $\lambda_{\beta(\alpha),0}=0$ is a simple eigenvalue and $h_{\beta(\alpha),0}=\mu(X)^{-1/2}\cdot 1$, where $1\in L^{\infty}(X,\mu)$ is the constant function. For the metric $\rho_{\alpha}$, Theorem B gives us the formula
$$\rho_{\alpha}(\varphi,\psi)=\left(\sum_{n\in \mathbb N} \lambda^{-1}_{\beta(\alpha),n} |(\varphi-\psi)(h_{\beta(\alpha),n})|^2\right)^{1/2}.$$

To our knowledge this is the only such closed formula for Monge--Kantorovi\v{c} metrics we could find in the literature. 

An interesting aspect of Theorem B is its potential utility in applications. In practice one only needs to compute a finite sum, as the error ought to be small, and governed by a Weyl law for $\Delta_{\beta(\alpha)}$. In Example \ref{exm:SFT} we illustrate the formula for the Cantor space $\{1,\ldots, N\}^{\mathbb N}$, equipped with the usual ultrametric and the Bernoulli measure.

The condition $\dw>\df$ is more ubiquitous than it first seems. Equivalently, it can be stated as $\ds<2$, where $\ds:=2\df/\dw$ is the spectral dimension of $(X,d)$. In Subsection \ref{sec:low_spectral_dim}, we provide background on $\ds$ but is probably important to mention also here that $\ds<2$ for a plethora of fractals of arbitrary topological dimension, and fractal structures found in self-similar group theory, quantum gravity, resistance networks, karst networks, and biology. 

\begin{remark}
It would be interesting to generalise Theorem B for the case $\df \geq \dw$, thus obtaining closed formulas for Monge--Kantorovi\v{c} metrics on the state space of any Ahlfors regular metric-measure space. In principle, Theorem B can even have analogues with respect to symmetrically normed ideals other than Schatten ideals, which might lead to Monge--Kantorovi\v{c} metrics with fascinating properties.
\end{remark}

\subsection*{Extension to dynamical systems}\label{sec:Intro3}

Finally, we discuss our study in the setting of dynamical systems using the language of crossed product $C^*$-algebras. For details on those $C^*$-algebras we refer to Subsection \ref{sec:Crossed_Schatten}. Namely, given a $C^*$-algebra $A$ and an action of an amenable countable discrete group $\Gamma$ on $A$ one considers the reduced crossed product $A\rtimes_r \Gamma$, which encodes the orbit space of the action. The goal now is to assemble any CQMS structures on $A$ and the reduced group $C^*$-algebra $C_r^*(\Gamma):=\mathbb C \rtimes_r \Gamma$ to a CQMS structure on $A\rtimes_r \Gamma$. The recent result \cite[Theorem 3.6]{AKK} states that this can be done whenever Conditions \eqref{eq:condition1}, \eqref{eq:elementarytensor} and \eqref{eq:condition2} are satisfied. 

The CQMS structures on $C_r^*(\Gamma)$ that we are mainly interested in are derived from proper translation $p$-summable functions on $\Gamma$ (see Definition \ref{ptranslationbounded}), for $1<p<\infty$. Because of the more refined properties of Schatten norms, as opposed to the operator norm, one cannot verify the sufficient conditions of \cite[Theorem 3.6]{AKK} in the same way as in the proof of \cite[Theorem A]{AKK}. 

In Subsection \ref{sec:Crossed_Schatten} we develop a new technique encompassing proper translation $p$-summable functions. Our method involves constructing a map $$w:\mathcal{S}_{p}(\ell^2(\Gamma))\otimes_{\pi} B(\ell^2(\Gamma))\to \mathcal{S}_{p}(\ell^2(\Gamma))\otimes_{\varepsilon} B(\ell^2(\Gamma)),$$ 
where $B(\ell^2(\Gamma))$ is the $C^*$-algebra of bounded operators on $\ell^2(\Gamma)$, and $\otimes_{\varepsilon},\otimes_{\pi}$ are the injective, projective tensor products of Banach spaces. The main technical result is then Lemma \ref{lem:wisacontraction} where we show that $w$ is a contraction using Grothendieck's Theorem on Schur multipliers \cite[Theorem 5.11]{Pisier}, \cite[Theorem 1.7]{LafSal}. This provides the main technical tool to verify the sufficient conditions of \cite[Theorem 3.6]{AKK}. Our result is the following.

\begin{thm3}[{Theorem \ref{thm:CQMS_crossed}}]
Let $Z$ be a compact metric space and $\Gamma$ an amenable, countable discrete group acting on $C(Z)$, equipped with a proper translation $p$-summable function $\ell$, for $1\leq p\leq \infty$. If $(\mathcal A, L)$ is a CQMS structure on $C(Z)$, and $(C_c(\Gamma),L_{\ell,p})$ is the $\Sp$-spectral metric space on $C^*_r(\Gamma)$ associated to $\ell$, and $\mathcal{A}$ is $\Gamma$-invariant, then the seminorm $\mathcal{L}_p: C_c(\Gamma, \mathcal{A})\to [0,\infty)$ given by $$\mathcal{L}_p(f):=\max \{L_{V,p}(f),L_H(f),L_H(f^*)\}$$ provides $C(Z)\rtimes_{r} \Gamma$ with a CQMS structure. Here $C_c(\Gamma)$ is the $*$-algebra of finitely supported functions on $\Gamma$. Also, $L_H$ and $L_{V,p}$ are the \textit{horizontal} and \textit{vertical} seminorms on $C(Z)\rtimes_{r} \Gamma$ induced by $L$ and $L_{\ell,p}$, respectively. 
\end{thm3}

Then, we move on to illustrate Theorem C for homoclinic group actions associated to the expansive algebraic $\mathbb Z^m$-actions of Lind--Schmidt \cite{LS}, see Subsection \ref{sec:Pon_dual}. The strength of our fractional analytic approach is that it offers an intrinsic way to study those dynamical systems from the standpoint of CQMS and defining length functions on homoclinic groups, independently of the fact that these are typically infinitely generated. Here we shall only informally present the statement and details can be found in Corollary \ref{cor:Z^m}.

\begin{thm4}
Assume that the metric-measure space $(X,d,\mu)$ is given by a compact abelian group $X$, an invariant metric $d$ and a normalised Ahlfors regular Haar measure $\mu$. Further, assume that $(X,d,\mu)$ admits an expansive and strongly mixing $\mathbb Z^m$-action by continuous group automorphisms, where $m\geq 1$. Then, for large enough $p>1$, 
\begin{enumerate}[(1)]
\item $C(X)$ admits an $\Sp$-spectral metric space structure by fractional Dirichlet Laplacians as per Theorem A; 
\item the structure in (1) gets transferred via Fourier transform to an $\Sp$-spectral metric space structure on $C^*_r(X^h(1_X))$, where $X^h(1_X)\subset X$ is the homoclinic (discrete) group of the $\mathbb Z^m$-action that is isomorphic to the Pontryagin dual $\widehat{X}$;
\item the structures in (1) and (2) assemble to a CQMS structure on the crossed product $C^*$-algebra $C(X)\rtimes_r X^h(1_X)$ as per Theorem C, where $X^h(1_X)$ acts on $X$ by translations. 
\end{enumerate}
\end{thm4}

In particular, this implies that Ruelle's homoclinic $C^*$-algebra \cite{Put,Ruelle_algebras} of a mixing Smale space with abelian group structure admits a CQMS structure by fractional Dirichlet Laplacians. For details we refer to Example \ref{exm:Smale}.

\begin{remark}
It would be interesting to extend the above result to homoclinic $C^*$-algebras of general Smale spaces, which might not have a group structure. This would require to extend our work to the setting of groupoid $C^*$-algebras \cite{Renault_Book}. One other approach would be to seek generalisations of Theorem C to the setting of quantum group actions, as this would facilitate extensions of Theorem D to non-abelian groups.
\end{remark}

\section{Preliminaries}\label{sec_Prelim}

\subsubsection*{Notation} If $F$ and $G$ are real valued functions on some parameter space $Z$, we write $F \lesssim G$ whenever there is a constant $L>0$ such that for all $z\in Z$ we have $F(z) \leq L G(z)$. Similarly, we define the symbol $\gtrsim$ and write $F\simeq G$ if $F \lesssim G$ and $F \gtrsim G$. Finally, given a metric-measure space $(X,d,\mu)$ and a measurable subset $Y\subset X$ as well as a measurable function $f:Y\to \mathbb{R}$ we denote by
\[\dashint_{Y}f(y)\dd\mu(y):=\frac{1}{\mu(Y)}\int_{Y} f(y)\dd\mu(y),\]
the average of $f$ over $Y$.

\subsection{Ahlfors regular metric-measure spaces}\label{sec:Ahlfors}
In this paper we focus on metric-measure spaces $(X,d,\mu)$ where $(X,d)$ is a compact metric space and $\mu$ is a finite Borel measure. \begin{definition}[{\cite{MT}}] Let $\df >0$. The metric-measure space $(X,d,\mu)$ is \textit{Ahlfors} $\df$\textit{-regular} if there is a constant $C\geq 1$, so that for every $x\in X$ and $0\leq r<\diam X$ it holds that 
\begin{equation*}
C^{-1}r^{\df}\leq \mu(\cB(x,r))\leq Cr^{\df}.
\end{equation*}
\end{definition}

For such metric-measure space, $\df$ equals the Hausdorff dimension of $(X,d)$ and in fact several other fractal dimensions of $(X,d)$, like the box-counting and Assouad dimensions, coincide with $\df$. Therefore, we shall refer to $\df$ as the \emph{fractal dimension} of $(X,d)$. Moreover, the measure $\mu$ is comparable to the $\df$-dimensional Hausdorff measure. Finally, such $\mu$ does not have atoms and hence $(X,d)$ has no isolated points. As a result, the measure $\mu$ is strictly positive.

From now on $(X,d,\mu)$ is always an Ahlfors $\df$-regular metric-measure space. Also, we will use the following estimates.
\begin{lemma}[{\cite{GSV}}]\label{lem:Ahlfors_estimates}
Let $0<r\leq \diam(X,d)$ and $s>0$. Then, 
\begin{enumerate}[(1)]
\item for every $x\in X$ it holds that, $$\int_{B(x,r)} \frac{1}{d(x,y)^{\df -s}} \dd\mu(y)\lesssim r^s;$$ 
\item for every $x\in X$ it holds that, $$\int_{X\setminus B(x,r)} \frac{1}{d(x,y)^{\df+s}}\dd\mu(y)\lesssim r^{-s}.$$
\end{enumerate}
\end{lemma}

Further, we will often consider the space $\Hol_{\beta}(X,d)$ of H\"older continuous ($\mathbb C$ or $\mathbb R$-valued) functions of exponent $0<\beta \leq 1$ on $(X,d)$. The H\"older constant $\Hol_{\beta}(f)$ of $f\in \Hol_{\beta}(X,d)$ is the smallest positive real number such that, for every $x,y \in X$, 
\begin{equation*}\label{eq:Holdersem}
|f(x)-f(y)|\leq \Hol_{\beta}(f) d(x,y)^{\beta}.
\end{equation*}
Since $\mu(X)<\infty$, we have that $\Hol_{\beta}(X,d)$ is a dense subspace of $L^2(X,\mu)$. Further, $\Hol_{\beta}(X,d)$ becomes a Banach space with respect to the norm $\|f\|_{\Hol_{\beta}}:=\|f\|_{\infty} + \Hol_{\beta}(f).$ For consistency, contrary to the notation in the introduction, from now on we will denote the space of Lipschitz continuous functions $\Lip(X,d)$ simply by $\Hol_1(X,d)$. 

Finally, we will use the Hardy--Littlewood Maximal Function Theorem \cite[Theorem 2.2]{Hei} which in particular asserts that the maximal function of $f\in L^2(X,\mu)$ given by
\begin{equation}\label{eq:HL}
\M f(x):=\sup_{r>0}\dashint_{B(x,r)}|f(y)|\dd\mu(y)
\end{equation}
satisfies $\|\M f\|_{L^2}\lesssim \|f\|_{L^2}.$

\subsection{Dyadic cubes and Haar wavelets}\label{sec:Dyadiccubes}
An indispensable tool in the study of Ahlfors regular metric-measure spaces (and more generally of spaces of homogeneous type) is the notion of dyadic cube decomposition, which allows to extend harmonic analysis to a setting that might not be Euclidean or algebraic, and where the Fourier transform is missing. There is a vast literature on this subject, with a seminal result being that of M. Christ \cite{C}. Here we will follow the one from \cite{KLPW}, which includes several references.

Specifically, from the Ahlfors regularity of $(X,d,\mu)$, there are constants $0<\theta<1$, $0<c_1\leq C_1,M>0$ and a countable family, called \textit{system of dyadic cubes},
\begin{equation}\label{eq:dyadic}
\D=\bigcup_{n\geq 0}\D_n,
\end{equation}
where each $\D_n$ is a finite non-empty collection of disjoint Borel subsets $D_{n,k}$ of $X$ that
\begin{enumerate}[(i)]
\item $\D_0 =\{X\}$ and for every $n\in \mathbb N$, $$X=\bigsqcup_{k=0}^{\# \D_n-1} D_{n,k};$$\label{eq:dyadic1}
\item if $n\geq m$ then either $D_{n,k}\subset D_{m,\ell}$ or $D_{n,k}\cap D_{m,\ell}=\varnothing;$\label{eq:dyadic2}
\item for every $D_{n,k}$ and $m\leq n$ there is a unique $0\leq \ell \leq \# \D_m-1$ such that $D_{n,k}\subset D_{m,\ell}$; \label{eq:dyadic3}
\item the \textit{children} $\C(D_{n,k}):=\{D_{n+1,\ell}\in \D_{n+1}: D_{n+1,\ell} \subset D_{n,k}\}$ of each $D_{n,k}$ are at least $2$ and at most $M$ (we shall call $D_{n,k}$ the \textit{parent} of dyadic cubes in $\C(D_{n,k})$), and $$D_{n,k}=\bigsqcup_{D\in \C(D_{n,k})}D;$$\label{eq:dyadic4}
\item for every $D_{n,k}$ there are points $x_{n,k}\in D_{n,k}$ such that $$B(x_{n,k},c_1\theta^n)\subset D_{n,k}\subset B(x_{n,k},C_1\theta^n);$$\label{eq:dyadic5}
\item $\# \D_n \lesssim \theta^{-\df n}$.\label{eq:dyadic6}
\end{enumerate}

Property (vi) follows from a volume argument, see \cite[Subsection 1.4.4]{MT}. A useful tool in our study will be the expectation operators $\Ex_n:L^2(X,\mu)\to L^2(X,\mu)$ given by
\begin{equation}\label{eq:expectation}
\Ex_n f(x)=\sum_{D\in \D_n} \left( \dashint_D f(y)\dd \mu(y)\right) \chi_D(x),
\end{equation}
where $\chi_D$ is the characteristic function of $D$. It is straightforward to see that $\Ex_n$ is a finite rank projection with rank at most $\# \D_n$. Another important aspect of dyadic cubes is that they offer a substitute for the one-third trick in Euclidean spaces.

\begin{thm}[{\cite[Theorem 4.1]{HK}}]\label{thm:adj_dyadic_cubes}
There exist constants $0<\theta<1$, $0<c_1\leq C_1,M>0$ and a finite collection $\{\D^t\}_{t=1}^{K}$ of systems of dyadic cubes, each satisfying properties (i)-(vi), and collectively, for every ball $B(x,r)\subset X$ with $\theta^{n+3}<r\leq \theta^{n+2}$ and $n\geq 0$, there exist $t\in \{1,\ldots K\}$ and $D^t\in \D^t_n$ such that $$B(x,r)\subset D^t \subset B(x,C_2r).$$ The constant $C_2$ depends only on $\theta$. 
\end{thm}

The final ingredient that we require from this setting is a Haar basis of $L^p(X,\mu)$, for $1<p<\infty$, associated to a system of dyadic cubes. We will only present the basic properties that are most relevant to our study and will not delve into specifics of the construction.

\begin{thm}[{\cite[Theorem 4.2]{KLPW}}]\label{thm:Haar_wav}
Given a system of dyadic cubes $\D$ as in \eqref{eq:dyadic}, there is a family of \textit{Haar wavelets} $h_{D,u}$, where $D\in \D$ and $u=1, \ldots, \# \C(D)-1$, satisfying 
\begin{enumerate}[(1)]
\item $h_{D,u}$ is a simple Borel-measurable $\mathbb R$-valued function on $X$;
\item $h_{D,u}$ is supported on $D$;
\item $h_{D,u}$ is constant on each $C\in \C(D)$;
\item $\int_X h_{D,u}(x) \dd \mu(x) =0$;
\item $\langle h_{D,u}, h_{D,u'}\rangle =0$ if $u\neq u'$;
\item the collection $$\{\mu(D)^{-\frac{1}{2}}\chi_D\}\cup \{h_{D,u}: u=1,\ldots,\# \C(D)-1\}$$ is an orthogonal basis for the vector space of all functions on $D$ that are constant on each $C\in \C(D)$;
\item $\|h_{D,u}\|_{L^{2}}=1$, and in general $\|h_{D,u}\|_{L^{p}}\simeq \mu(D)^{\frac{1}{p} -\frac{1}{2}}$ if $1\leq p\leq \infty.$ 
\end{enumerate}
\end{thm}

\subsection{Dirichlet forms}
We note that all Hilbert spaces are assumed to be separable and in Sections \ref{sec_Prelim} and \ref{sec:Lap} are also assumed to be real, as this is customary in the setting of Dirichlet forms, unless stated otherwise. For the theory of Dirichlet forms we refer the reader to \cite{FOT}.

Let $H$ be a Hilbert space and $\Dom\Q \subset H$ be a linear subspace that is the domain of a bilinear form $\Q:\Dom\Q \times\Dom\Q \to \mathbb{R}$. Then, $\Dom\Q$ carries the following inner product and induced norm
\begin{equation}
\label{eq: form-inner-product}
\langle f,g\rangle_{\Q}:=\langle f,g\rangle_{H}+\Q(f,g), \quad \|f\|_{\Q}^{2}:=\langle f,f\rangle_{\Q}.
\end{equation}
The form $\Q$ is \emph{densely defined} if $\Dom\Q\subset H$ is dense and \emph{closed} if $\Dom\Q$ is a Hilbert space with respect to the inner product \eqref{eq: form-inner-product}. 

\begin{thm}
\label{thm: form-operator-equivalence}
There is a one-to-one correspondence between densely defined closed bilinear forms $\Q$ on $H$ and positive self-adjoint operators $\A:\Dom\A\to H$ with $\Dom\A^{1/2}=\Dom \Q$ and $\Q(f,g)=\langle \A^{1/2}f,\A^{1/2}g\rangle_{H}$. 
\end{thm}

Here we are interested in the Hilbert space $H=L^{2}(X,\mu)$. For functions $f,g\in L^{2}(X,\mu)$ we say that $g$ is a \emph{normal contraction} of $f$, written $g\prec f$, if for $\mu$-almost all $x,y\in X$ we have
\begin{equation*}
|g(x)|\leq |f(x)|,\quad |g(x)-g(y)|\leq |f(x)-f(y)|.
\end{equation*}
A densely defined, closed bilinear form $\Q$ on $L^{2}(X,\mu)$ is a \emph{Dirichlet form} if 
\begin{equation}
\label{eq: normalcontractionproperty}
f\in\Dom \Q,\,\, g\prec f \Rightarrow g\in \Dom \Q,\,\, \Q(g,g)\leq \Q(f,f).
\end{equation}

Further, the form is \emph{regular} if $C(\Q):=\Dom \Q \cap C(X)$ is dense in $\Dom\Q$ and $C(X)$ with the $\|\cdot \|_{\Q}$-norm and the $\|\cdot \|_{\infty}$-norm, respectively. Regular Dirichlet forms have a rich theory as they give rise to special Markov processes on $(X,d,\mu)$, known as Hunt processes. 

\subsection{Schatten classes and norm estimates}\label{sec:RS-tech}
A basic reference about singular values of compact operators on separable Hilbert spaces is \cite[Chapter II]{GK}. Given a compact operator $T:H_1\to H_2$ denote by $(s_n(T))_{n\in \mathbb N}$ the sequence of its singular values in decreasing order, counting multiplicities. It holds that for every $n\in \mathbb N$,
\begin{equation}\label{eq:singular_values_1}
s_n(T)=\inf_{F\in \R_{n-1}}\|T-F\|,
\end{equation}
where $\R_{n-1}$ is the set of operators from $H_1$ to $H_2$ of rank at most $n-1$ and $\|\cdot \|$ is the operator norm. Further, for every bounded operator $S:H_0\to H_1$ one has 
\begin{equation}\label{eq:singular_values_2}
s_n(TS)\leq s_n(T)\|S\|.
\end{equation}
Now, given a separable Hilbert space $H$, we will denote by $\mathcal{K}(H)$ the closed two-sided ideal of compact operators in $\mathcal{B}(H)$. For $p\geq 1$, the \textit{Schatten} $p$\textit{-ideal} $\Sp(H)$ given by $$\Sp(H):=\{T\in \mathcal{K}(H): (s_n(T))_{n\in \mathbb N} \in \ell^p(\mathbb N)\},$$ is a two-sided ideal of $\mathcal{B}(H)$ and is a Banach $*$-algebra with respect to the norm $$\|T\|_{\Sp}:=\|(s_n(T))_{n\in \mathbb N}\|_{\ell^p}.$$ 
In the sequel we will be interested to obtain estimates for Schatten norms. For this it is important to understand for what sequences $(e_i)_{i\in \mathbb N},(f_i)_{i\in \mathbb N}$ in $H$ it holds that, whenever $(\lambda_n)_{n\in \mathbb N}\in \ell^p(\mathbb N)$, the operator 
\begin{equation}\label{eq:Schmidt_dec}
T=\sum_{i\in \mathbb N} \lambda_n \langle \cdot, e_i\rangle f_i
\end{equation}
lies in $\Sp(H)$ and in fact $\|T\|_{\Sp}\lesssim \|(\lambda_n)_{n\in \mathbb N}\|_{\ell^p}.$ This is clearly true when $(e_i)_{i\in \mathbb N},(f_i)_{i\in \mathbb N}$ are orthonormal sets, or more generally when the matrices $(\langle e_i,e_j\rangle)_{i,j}$, $(\langle f_i,f_j\rangle)_{i,j}$ define bounded operators on $\ell^2(\mathbb N)$. In practice though it can happen that $(e_i)_{i\in \mathbb N},(f_i)_{i\in \mathbb N}$ are not that tractable. However, if $H=L^2(\mathbb R^N)$, to the rescue comes the dyadic technique of Rochberg--Semmes \cite{RS} which uses \textit{nearly weakly orthonormal} (NWO) sequences. Although it clearly can be adjusted to $L^2(X,\mu)$, surprisingly, in the literature we were able to find this only for Euclidean settings. We begin with the version of NWO-sequences that we require for this paper. 

\begin{definition}\label{defn:NWO}
Let $\D$ be a system of dyadic cubes for $(X,d,\mu)$. A sequence $e=\{e_D\}_{D\in \D}$ in $L^2(X,\mu)$ will be called \textit{NWO($\D$)} if for every $f\in L^2(X,\mu)$, the maximal function $$\M_e f(x):=\sup_{x\in D} \frac{|\langle f, e_D \rangle |}{\mu(D)^{1/2}}$$ satisfies $\|\M_e f\|_{L^2}\lesssim \|f\|_{L^2}.$
\end{definition}

A prototype example of NWO($\D$)-sequence is one that satisfies $|e_D|\lesssim \mu(D)^{-1/2} \chi_D$. This follows immediately from the boundedness of the $\D$-maximal function on $L^2(X,\mu)$, see \cite{ABI}. Now, the next result follows by using similar techniques as in \cite[p. 241]{RS}, with the main exception being the Carleson measure norm of a sequence is replaced in all parts of the proof by the equivalent norm described in (1.7) of \cite{RS}. For convenience of the reader we include the proof. 

\begin{prop}\label{prop:upper_Schatten_est}
Assume $p>1$. Let $\D$ be a system of dyadic cubes for $(X,d,\mu)$ as in \eqref{eq:dyadic} and $e=\{e_D\}_{D\in \D}$, $f=\{f_D\}_{D\in \D}$ be NWO($\D$)-sequences. Then, for every $\lambda=(\lambda_D)_{D\in \D}\in \ell^p(\D)$ the operator $$T=\sum_{D\in \D} \lambda_D\langle \cdot , e_D \rangle f_D$$ satisfies $\|T\|_{\Sp}\lesssim \|\lambda\|_{\ell^p}.$
\end{prop}

\begin{proof}
We first partially order the index set $\{(n,k): n\geq 0, 0\leq k\leq \#\D_n-1\}$ of $\D$ by setting $(n,k)\leq (m,\ell)$ when $D_{m,\ell}\subset D_{n,k}$. Further, for every $(n,k)$ and $m\geq n$ we define $$L_m(n,k):=\{\ell: 0\leq \ell \leq \# \D_m-1,\,\, (m,\ell) \geq (n,k)\}.$$  Now, we consider the sequence $M(\lambda)$ indexed over $\D$ defined as $$M(\lambda)_{D_{n,k}}:=\sum_{(m,\ell)\geq (n,k)} |\lambda_{D_{m,\ell}}|\mu (D_{m,\ell}) \mu (D_{n,k})^{-1}.$$ The proof of the Proposition follows from Claims 1 and 3, and Claim 2 is used to prove Claim 3. Specifically, we have the following.

\vspace{0.5cm}

\noindent \textbf{Claim 1:} \textit{It holds that} $$\|M(\lambda)\|_{\ell^p}\lesssim \|\lambda\|_{\ell^p}.$$ \textit{In particular, we have that} $\|M(\lambda)\|_{\infty} < \infty$. 

\vspace{0.5cm}

To this end, let $q=p/(p-1)>1$ be the conjugate of $p$ and choose some $1<\eta <q$. Also, let $c(\eta):=1-\theta^{\df (p-1) \eta}$. Then,
\begin{align*}
\|M(\lambda)\|_{\ell^p}^p&=\sum_{(n,k)}\left(\sum_{(m,\ell)\geq (n,k)} |\lambda_{D_{m,\ell}}|\mu (D_{m,\ell}) \mu (D_{n,k})^{-1}\right)^p\\
&=\sum_{(n,k)}\left(\sum_{m\geq n} \sum_{\ell \in L_m(n,k)}|\lambda_{D_{m,\ell}}|\mu (D_{m,\ell}) \mu (D_{n,k})^{-1}\right)^p\\
{}^{(\ast)}&\lesssim \sum_{(n,k)}\left(\sum_{m\geq n} \theta^{\df (m-n)} \sum_{\ell \in L_m(n,k)}|\lambda_{D_{m,\ell}}| \right)^p\\
&\lesssim \sum_{(n,k)}\left(\sum_{m\geq n} c(\eta)\theta^{\df (m-n)(p-1)\eta} \left(\theta^{\df (m-n) (1-(p-1)\eta)}\sum_{\ell \in L_m(n,k)}|\lambda_{D_{m,\ell}}|\right) \right)^p\\
{}^{(\ast \ast)}&\lesssim \sum_{(n,k)} \left(\sum_{m\geq n} c(\eta)\theta^{\df (m-n)(q-\eta)}\right)^{p-1} \left(\sum_{m\geq n}c(\eta)\theta^{\df (m-n)(p-1)\eta} \left( \sum_{\ell \in L_m(n,k)}|\lambda_{D_{m,\ell}}|\right)^p\right)\\
{}^{(\ast \ast \ast)}&\lesssim \sum_{(n,k)} \sum_{m\geq n} c(\eta)\theta^{\df (m-n)(p-1)\eta}\theta^{-\df (m-n)(p-1)}\sum_{\ell \in L_m(n,k)}|\lambda_{D_{m,\ell}}|^p\\
&= \sum_{(n,k)}\sum_{(m,\ell)\geq (n,k)} c(\eta)\theta^{\df (m-n)(p-1)(\eta-1)}|\lambda_{D_{m,\ell}}|^p\\
{}^{(\ast \ast \ast \ast)}&=\sum_{(m,\ell)} |\lambda_{D_{m,\ell}}|^p\sum_{(n,k)\leq (m,\ell)} c(\eta)\theta^{\df (m-n)(p-1)(\eta-1)}\\
&\lesssim \|\lambda\|_{\ell^p}^p.
\end{align*}
Inequality ($\ast$) follows from the Ahlfors $\df$-regularity of $\mu$ and property (v) of dyadic cubes. Inequality ($\ast \ast$) follows from a H\"older inequality with respect to the probability measure on the set $\{m\geq n\}$ (for fixed $n$) associated to the density function $c(\eta)\theta^{\df (m-n)(p-1)\eta}$. Further, inequality ($\ast \ast \ast$) follows from a H\"older inequality with respect to the counting measure on $L_m(n,k)$. Finally, the sum $$\sum_{(n,k)\leq (m,\ell)} c(\eta)\theta^{\df (m-n)(p-1)(\eta-1)}$$ in ($\ast \ast \ast \, \ast$) is bounded independently of $(m,\ell)$ since for every $(m,\ell)$ and $n\leq m$ there is a unique $0\leq k\leq \# \D_n-1$ such that $(n,k)\leq (m,\ell)$, see property (iii) of dyadic cubes.

\vspace{0.5cm}

\noindent \textbf{Claim 2:} \textit{It holds that} $\|T\|\lesssim \|M(\lambda)\|_{\infty}.$

\vspace{0.5cm}

First, for every $D\in \D$ denote by $\widehat{D}$ its parent. If $D=\{X\}$ then we define $\widehat{D}:=\{X\}$. Recall that each $D$ and $\widehat{D}$ have comparable volumes due to property (v) of dyadic cubes. Now, using that $e,f$ are NWO($\D$)-sequences, for every $g,h \in L^2(X,\mu)$ we obtain that
\begin{align*}
|\langle Tg,h\rangle_{L^2} |&= \left \lvert  \sum_{D\in \D} \lambda_D\langle g,e_D\rangle_{L^2}\langle h,f_D\rangle_{L^2}\right \rvert\\
&\lesssim \sum_{D\in \D} |\lambda_D|\mu(D)\mu(\widehat{D})^{-1}\mu(D)\mu(D)^{-1}|\langle g,e_D\rangle_{L^2}||\langle h,f_D\rangle_{L^2}|\\
&= \sum_{D\in \D} |\lambda_D|\mu(D)\mu(\widehat{D})^{-1}\int_X\chi_D(x)\frac{|\langle g,e_D\rangle_{L^2}|}{\mu(D)^{\frac{1}{2}}}\frac{|\langle h,f_D\rangle_{L^2}|}{\mu(D)^{\frac{1}{2}}}\dd \mu (x)\\
&\leq \sum_{D\in \D} |\lambda_D|\mu(D)\mu(\widehat{D})^{-1}\int_X \sup_{x\in R}\frac{|\langle g,e_R\rangle_{L^2}|}{\mu(R)^{\frac{1}{2}}}\frac{|\langle h,f_R\rangle_{L^2}|}{\mu(R)^{\frac{1}{2}}}\dd \mu (x)\\
&= \sum_{D\in \D} |\lambda_D|\mu(D)\mu(\widehat{D})^{-1}\int_X \M_e g(x) \M_f h(x) \dd \mu (x)\\
&\lesssim M(\lambda)_{\widehat{D}}\|g\|_{L^2}\|h\|_{L^2}\\
&\leq \|M(\lambda)\|_{\infty} \|g\|_{L^2}\|h\|_{L^2},
\end{align*}
thus proving Claim 2. 

\vspace{0.5cm}

Finally, let us denote by $M_*(\lambda)$ the non-increasing re-arrangement over $\mathbb N$ of $M(\lambda)$. Also, we order the dyadic cubes in $\D$ as $\{D_1,D_2,\ldots\}$ so that $M_*(\lambda)_n=M(\lambda)_{D_n}$, for all $n\in \mathbb N$.

\vspace{0.5cm}

\noindent \textbf{Claim 3:} The singular values of $T$ satisfy $$s_n(T)\lesssim M_*(\lambda)_n.$$ 

\vspace{0.5cm}

The case $n=1$ follows immediately from Claim 2. Specifically, we have that $$s_1(T)=\|T\|\lesssim \|M(\lambda)\|_{\infty}=M_*(\lambda)_1.$$ For $n\geq 2$ define the sequence $\lambda^n$ over $\D$ by $\lambda^n(D_j)=0$ for $j=1,\ldots, n-1$ and $\lambda^n(D_j)=\lambda_{D_j}$ for $j\geq n$. Also, define the operator $$T_n:=\sum_{j=1}^{n-1} \lambda_{D_j}\langle \cdot , e_{D_j} \rangle f_{D_j},$$ which has rank at most $n-1$. Then, $$s_n(T)\leq \|T-T_n\|\lesssim \|M(\lambda^n)\|_{\infty} \lesssim M_*(\lambda)_n,$$ where the second inequality follows from Claim 2 and the third is proved as in \cite[p. 241]{RS}. 
\end{proof}

Then, from Proposition \ref{prop:upper_Schatten_est} and duality of $\ell^p$-spaces one immediately obtains the following. 

\begin{prop}[{\cite[p. 261]{RS}}]\label{prop:NWO}
Assume $p>1$. Let $\D$ be a system of dyadic cubes for $(X,d,\mu)$ and $\{e_D\}_{D\in \D}$, $\{f_D\}_{D\in \D}$ be NWO($\D$)-sequences. Then, for every $T\in \Sp(L^2(X,\mu))$ it holds that $$\left(\sum_{D\in \D} |\langle Te_D,f_D\rangle|^p\right)^{1/p}\lesssim \|T\|_{\Sp}.$$
\end{prop}

\section{Fractional Dirichlet Laplacian}\label{sec:Lap}

Let $0<\alpha <1$ and consider the real vector space
\label{eq: Sobolevspace}
\begin{equation}
\Ha(X,d,\mu):=\left\{f\in L^2(X,\mu):\int_X\int_X\frac{|f(x)-f(y)|^2}{d(x,y)^{\df+2\alpha}}\dd\mu(y)\dd\mu(x)<\infty\right\}.
\end{equation}
We will often simply write $\Ha$, if no confusion arises. Further, consider the bilinear form $\E:\Ha \times \Ha \to \mathbb R$ given by
\begin{equation}
\label{eq: formdef}
\E(f,g):=\frac{1}{2}\int_X\int_X \frac{(f(x)-f(y))(g(x)-g(y))}{d(x,y)^{\df+2\alpha}}\dd\mu(y)\dd\mu(x).
\end{equation}
The integral defining $\E$ is to be interpreted as an integral over $X\times X\setminus D$, with $D$ being the diagonal defined as $$D:=\{(x,x):x\in X\}\subset X\times X,\quad (\mu\times \mu)(D)=0.$$ It is a well-known fact that $\E$ is densely defined. We include the proof so that the reader understands the use of Ahlfors regularity.

\begin{prop}
For every $0<\alpha <\beta \leq 1$, $\Hol_{\beta}(X,d)$ embeds continuously in $\Ha$. Consequently, $\E$ is densely defined on $L^2(X,\mu)$.
\end{prop}

\begin{proof}
Let $f\in \Hol_{\beta}(X,d)$. Then,
\begin{align*}
\int_X\int_X\frac{|f(x)-f(y)|^2}{d(x,y)^{\df+2\alpha}}\dd\mu(y)\dd\mu(x)&\leq \Hol_{\beta}(f)^2 \int_X\int_X\frac{1}{d(x,y)^{\df+2(\alpha-\beta)}}\dd\mu(y)\dd\mu(x).
\end{align*}
Using Lemma \ref{lem:Ahlfors_estimates} and the fact that $\mu(X)<\infty$, the proof follows.
\end{proof}

Further, $\E$ is closed (see \cite[Example 1.2.4]{FOT}), satisfies \eqref{eq: normalcontractionproperty} and hence is a Dirichlet form. Equipping $\Ha$ with the inner product \eqref{eq: form-inner-product} it then becomes a Hilbert space known as the $\alpha$\emph{-fractional Sobolev space} on $(X,d,\mu)$. 

\begin{remark}
It is clear that as $\alpha$ becomes larger the space $\Ha$ shrinks, as if $0<\alpha \leq \beta <1$, then $\Hb$ embeds continuously to $\Ha$. The reason for restricting the values of $\alpha$ though is that in general it can happen that $\Ha$ is not dense in $L^2(X,\mu)$, if $a\geq 1$. Nevertheless, in Subsection \ref{sec:low_spectral_dim} we will examine cases where it makes sense to consider $\alpha \geq 1$.
\end{remark}

The main object of study in this paper is the following.

\begin{definition}\label{def:Lap}
Let $0<\alpha<1$. The $\alpha$\emph{-fractional Dirichlet Laplacian} is the positive, self-adjoint operator $\Da$ associated to the Dirichlet form $\E$ via Theorem \ref{thm: form-operator-equivalence}.
\end{definition}

Fractional Sobolev spaces have been studied extensively, see for instance \cite{GS, NPV} for general properties and \cite{GHL, CKW} for H\"older heat kernel estimates. However, fine properties of fractional Dirichlet Laplacians on arbitrary Ahlfors regular metric-measure spaces which can be used in noncommutative geometry have not been studied yet. This is a challenging endeavour which very much depends on how robust $(X,d,\mu)$ is. For instance, in \cite{GSV,Nah} the authors examine the case where $d$ is instead a quasi-metric constructed from a Coifman approximate identity. This (non-canonical) choice of quasi-metric allows them to use Calder\'on--Zygmund theory and Littlewood--Paley decompositions to study the analogue $\Da'$ of $\Da$ on those quasi-metric-measure spaces. This change to a quasi-metric can also be done in our framework, however the associated $\Da'$ will typically have no close relation to $\Da$, except for the fact that their fractional Sobolev spaces will be equivalent, similarly as in Lemma \ref{lem:bi_lip_inv}. In the present paper we aim to study $\Da$ explicitly, without changing the underlying metric-measure space.

\subsection{Weyl law}
In our study of $\Da$ we first aim to obtain a Weyl law for its eigenvalues. To this end, we start with an interesting scaling property of fractional Sobolev spaces, whose proof is immediate. We note that for every $0<\varepsilon <1$, the \textit{snowflaked} metric-measure space $(X,d^{\varepsilon},\mu)$ is Ahlfors $\df/\varepsilon$-regular. 

\begin{lemma}\label{lem:scale_inv}
Let $0<\alpha <1$. For every $0<\beta,\varepsilon <1$ such that $\alpha=\beta \varepsilon$, it holds that $\Ha (X,d,\mu)= \Hb (X,d^{\varepsilon},\mu)$. In particular, one can choose $\beta = (1+\alpha)/2$ and $\varepsilon=2\alpha/(1+\alpha)$.
\end{lemma}

Moreover, fractional Sobolev spaces are essentially invariant under bi-Lipschitz maps. The proof follows from a straightforward change of variables. 

\begin{lemma}\label{lem:bi_lip_inv}
Let $(Z,\rho)$ be a metric space and $F:(X,d)\to (Z,\rho)$ be a bi-Lipschitz map. Then, the push-forward measure $\mu_F$ of $\mu$ on $(Z,\rho)$ is Ahlfors $\df$-regular and the map $U_F:L^2(X,\mu)\to L^2(Z,\mu_F)$ given by $U_Ff=f\circ F^{-1}$ is a Hilbert space isomorphism such that for every $f\in \Ha(X,d,\mu)$ we have $$\E^{(X,d,\mu)}(f,f)\simeq \E^{(Z,\rho,\mu_F)}(U_F(f),U_F(f)).$$ In particular, $U_F$ descends to a bounded bijection $U_F:\Ha(X,d,\mu)\to \Ha(Z,\rho,\mu_F)$ so that the following diagram commutes

$$\begin{tikzcd}
\Ha(X,d,\mu) \arrow[hookrightarrow]{r} \arrow{d}{U_F} & L^2(X,\mu) \arrow{d}{U_F}\\
\Ha(Z,\rho,\mu_F) \arrow[hookrightarrow]{r} & L^2(Z,\mu_F)
\end{tikzcd}.
$$
\end{lemma}

Lemmas \ref{lem:scale_inv} and \ref{lem:bi_lip_inv} combined with Assouad's Embedding Theorem \cite[Proposition 2.6]{A} assert that $\Ha$ essentially consists of functions on Ahlfors regular subsets of Euclidean spaces. For convenience of the reader, we note that Assouad's Theorem holds in the more general setting of doubling metric spaces, but in our case it implies that, for every $0<\varepsilon <1$ there is a bi-Lipschitz embedding $F_{\varepsilon}:(X,d^{\varepsilon})\to (\mathbb R^N,|\cdot|)$ where $|\cdot|$ is the Euclidean norm. The dimension $N$ depends on $\varepsilon$ and the Ahlfors regularity constants $C, \df$ of the measure $\mu$. 

Before presenting the Weyl law, we mention an interesting fact that is not required for obtaining the law. This might be already known, but we could not find it in the literature. Since the Dirichlet form $\E$, in the case $X\subset \mathbb R^N$ for some $N\in \mathbb N$ and $d$ is the restricted Euclidean metric on $X$, is regular (follows immediately from a trace theorem for fractional Sobolev spaces, see discussion \cite[p. 32]{CKW}), a corollary is the following.

\begin{prop}\label{eq: regularform}
Let $0<\alpha<1$. Then, the space $C(\E)=\Ha \cap C(X)$ is $\|\cdot \|_{\E}$-dense in $\Ha$ and hence the Dirichlet form $\E$ is regular.
\end{prop}

The main result regarding the spectrum of fractional Dirichlet Laplacians is the following.

\begin{thm}\label{thm:Weyl_law}
For every $0<\alpha<1$, the operator $\Da$ has compact resolvent. In particular, $$s_n((1+\Da)^{-1})\simeq n^{-\frac{2\alpha}{\df}}.$$ 
\end{thm}

\begin{proof}
Let $0<\alpha <1$ and denote the embedding $\Ha(X,d,\mu)\hookrightarrow L^2(X,\mu)$ by $J$. We claim that $J$ is compact and that its singular values 
\begin{equation}\label{eq:Weyl_law_0}
s_n(J)\simeq n^{-\frac{\alpha}{\df}}.
\end{equation}
This immediately implies that the resolvent $(1+\Da)^{-1}$ is compact with 
\begin{equation*}\label{eq:Weyl_law_1}
s_n((1+\Da)^{-1})\simeq n^{-\frac{2\alpha}{\df}}.
\end{equation*}

We now prove \eqref{eq:Weyl_law_0}. Let $0<\beta, \varepsilon <1$ such that $\alpha=\beta \varepsilon$. Then, from Lemma \ref{lem:scale_inv} we have 
\begin{equation*}\label{eq:Weyl_law_2}
\Ha(X,d,\mu)=\Hb(X,d^{\varepsilon},\mu).
\end{equation*}
Consider a bi-Lipschitz embedding $F_{\varepsilon}:(X,d^{\varepsilon})\to (\mathbb R^N,|\cdot|)$, which exists from Assouad's Theorem. Then, we obtain the Ahlfors $\df/\varepsilon$-regular subset $(F_{\varepsilon}(X,d^{\varepsilon}), |\cdot|, \mu_{F_{\varepsilon}})$ of $(\mathbb R^N,|\cdot |)$ for which we denote the embedding $\Hb(F_{\varepsilon}(X,d^{\varepsilon}), |\cdot|, \mu_{F_{\varepsilon}})\hookrightarrow L^2(F_{\varepsilon}(X,d^{\varepsilon}), \mu_{F_{\varepsilon}})$ by $J_{F_{\varepsilon}}$. Then, from Lemma \ref{lem:bi_lip_inv} it holds that 
\begin{equation*}\label{eq:Weyl_law_3}
s_n(J)\simeq_{F_{\varepsilon}} s_n(J_{F_{\varepsilon}}).
\end{equation*}
So, we only need to obtain compactness of $J_{F_{\varepsilon}}$ and estimate its singular values.

To this end, without loss of generality we can assume that $N>\df/\varepsilon$ and then observe that $\Hb(F_{\varepsilon}(X,d^{\varepsilon}), |\cdot|, \mu_{F_{\varepsilon}})$ can be identified with the Besov space $B_{2,2}^{\beta}(F_{\varepsilon}(X,d^{\varepsilon}))$ described in \cite[Section 20]{T} which comes from taking the trace (restriction) onto $F_{\varepsilon}(X,d^{\varepsilon})$ of functions in the classical Besov space $B_{2,2}^{\beta+(N-\df/\varepsilon)/2}(\mathbb R^N),$ see for instance \cite{CKW,JW}. Then, from \cite[Theorem 25.2]{T} we get that $J_{F_{\varepsilon}}$ is compact with singular values $$s_n(J_{F_{\varepsilon}})\simeq n^{-\frac{\beta}{\df / \varepsilon}}=n^{-\frac{\alpha}{\df}}.$$ This completes the proof.
\end{proof}

An immediate corollary of the compactness of the resolvent is that the image of $\Da$ is closed. Specifically, in order to describe the kernel and image of $\Da$, consider the projection $\Ex_0:L^2(X,\mu)\to L^2(X,\mu)$ onto the constant functions given by 
\begin{equation*}
\Ex_0f=\dashint_Xf(x)\dd\mu(x).
\end{equation*} 
Therefore, $L^2(X,\mu)=\mathbb R 1\oplus \ker \Ex_0$, with $1\in L^{\infty}(X,\mu)$ denoting the constant function and $\ker \Ex_0$ being the subspace of square integrable functions with zero integral. The proof of the next result is similar to our proof of \cite[Proposition 3.4]{GM}.

\begin{prop}\label{prop:kernel}
There is an orthogonal decomposition $L^2(X,\mu)=\ker \Da \bigoplus \Im \Da$, where $\ker \Da=\Im \Ex_0$ and $\Im \Da = \ker \Ex_0$. 
\end{prop}

As a result, from Theorem \ref{thm:Weyl_law} and Proposition \ref{prop:kernel} we obtain the following.

\begin{cor}[Weyl law]\label{cor:Weyl_law}
Let $0<\alpha<1$. If $(\lambda_{\alpha,n})_{n\geq 0}$ is the sequence of eigenvalues of $\Da$ arranged in increasing order, while counting multiplicities, then $$\lambda_{\alpha,n}\simeq n^{\frac{2\alpha}{\df}}.$$ Moreover, the eigenvalue $\lambda_{\alpha,0}=0$ is simple. 
\end{cor}

\begin{remark}\label{rem:com_emb}
There are more ways than the one described in Theorem \ref{thm:Weyl_law} for obtaining the compactness of the resolvent $(1+\Da)^{-1}$. However, these do not lead to singular value estimates. Specifically, one is through the embedding of $\Ha$ into a Haj\l{}asz space \cite{GS}, where the latter has a compact embedding into $L^2(X,\mu)$. The second is through the existence of a H\"older heat kernel for $\Da$, see \cite{CKW}. Another one was developed in our previous work while studying the logarithmic Dirichlet Laplacian \cite[Proposition 3.2]{GM}. It can be used to show that the embedding of $\Ha$ into $L^2(X,\mu)$ is compact, without intermediate embeddings. It even allows to estimate the singular values of $(1+\Da)^{-1}$ from above (in fact, this method can be used to obtain such estimates in the case of bilinear forms with other kind of singularities than the one in \eqref{eq: formdef}). Then, in combination with the Littlewood--Paley theory used in \cite{Nah} we can obtain the same Weyl law for $\Da$, though whenever $0<\alpha<\alpha_0<1$, where $\alpha_0$ is obtained in \cite[Proposition 1.2]{Nah}. Finally, obtaining estimates for $\alpha \geq 1$ (see Subsection \ref{sec:low_spectral_dim}) is more involved (see for instance \cite{JW}) and outside of the scope of this paper. Nevertheless, the method using Haj\l{}asz spaces can at least yield the compactness of the embeddings into $L^2(X,\mu)$ for every $\alpha>0$, see \cite [Theorem 4.3]{GS}.
\end{remark}

\subsection{Analysis of domain}\label{sec:anal_domain}
For $h\in L^{\infty}(X,\mu)$ consider the multiplication operator $\m_h:L^2(X,\mu)\to L^2(X,\mu)$ given by $\m_hf=hf$. Note that $\m_{h}^{*}=\m_{h}$. We first show that fractional Sobolev spaces are Banach modules over the spaces of H\"older continuous functions. 

\begin{lemma}\label{lem:module_Hol}
For every $0<\alpha < \beta \leq 1$, $h\in \Hol_{\beta}(X,d)$ and $f\in \Ha$ we have that $\m_{h}f=hf\in \Ha$ and $$\|hf\|_{\E}\lesssim \|h\|_{\Hol_{\beta}}\|f\|_{\E}.$$ 
\end{lemma}

\begin{proof}
First observe that $hf\in L^2(X,\mu)$ since $h$ is bounded. Then, 
\begin{align*}
\E (hf,hf)&=\frac{1}{2}\int_{X}\int_{X}\frac{|h(x)f(x)-h(y)f(y)|^2}{d(x,y)^{\df+2\alpha}}\dd \mu (y) \dd \mu (x)\\
&\leq \int_X \int_X \frac{|f(x)(h(x)-h(y))|^2+|h(y)(f(x)-f(y))|^2}{d(x,y)^{\df + 2\alpha}}\dd \mu (y) \dd \mu (x)\\
&\lesssim \|f\|_{L^2}^2 \Hol_{\beta}(h)^2 + \|h\|_{\infty}^2\E(f,f).
\end{align*}
Since $\|hf\|_{L^2}\leq \|h\|_{\infty}\|f\|_{L^2}$, the proof is complete.
\end{proof}

\begin{lemma}\label{lem:commutator_bounded}
Let $0<2\alpha<\beta\leq 1$. Then, for every $h\in \Hol_{\beta}(X,d)$, the kernel
\[K_{\alpha,h}:X\times X\setminus D\to \mathbb{R},\,\,K_{\alpha,h}(x,y):=\frac{h(x)-h(y)}{d(x,y)^{\df +2\alpha}},\]
defines a bounded operator $\K_{\alpha,h}:L^2(X,\mu)\to L^2(X,\mu)$ given by $$\K_{\alpha,h}f(x)=\int_X\frac{h(x)-h(y)}{d(x,y)^{\df+2\alpha}}f(y)\dd\mu(y),$$ with operator norm $\|\K_{\alpha,h}\|\lesssim \Hol_{\beta}(h).$ 
\end{lemma}

\begin{proof}
Let $x\in X$ and for every integer $k\geq 0$, define $r_k=e^{-k}\diam(X,d)$ and the annulus $B_{x,k}=B(x,r_k)\setminus B(x,r_{k+1}).$ Then, 
\begin{align*}
|\K_{\alpha,h} f(x)|&\leq \sum_{k=0}^{\infty}\int_{B_{x,k}}\frac{|h(x)-h(y)|}{d(x,y)^{\df+2\alpha}}|f(y)|\dd\mu(y)\\
&\lesssim \sum_{k=0}^{\infty} e^{k(2\alpha-\beta)}\Hol_{\beta}(h)\dashint_{B(x,r_k)}|f(y)|\dd\mu(y)\\
&\lesssim \Hol_{\beta}(h)\M(f)(x),
\end{align*}
where $\M(f)$ is the Hardy--Littlewood maximal function of $f\in L^2(X,\mu)$ defined in \eqref{eq:HL}. The proof is then complete since $\|\M(f)\|_{L^2}\lesssim \|f\|_{L^2}.$
\end{proof}

\begin{thm}\label{theorem:domain_module}
Let $0<2\alpha<\beta\leq 1$ and $h\in \Hol_{\beta}(X,d)$. Then, $\m_h(\Dom \Da)\subset \Dom \Da$ and the commutator $[\Da,\m_{h}]:\Dom\Da\to L^{2}(X,\mu)$ extends to the bounded operator $\K_{\alpha,h}$.
\end{thm}

\begin{proof}
For every $f\in \Dom \Da$, Lemma \ref{lem:module_Hol} asserts that $hf\in \Ha$. So by \cite[Proposition 4.2]{GM} we only have to show that there is $C_{h}>0$ such that, for all $g\in \Ha$, it holds  
\begin{equation*}
|\E(hf,g)-\E(f,hg)|\leq C_{h}\| f\|_{L^{2}}\|g\|_{L^2}.
\end{equation*}
Then,
\begin{align}
\nonumber\E(hf,g)-\E(f,hg)&=\frac{1}{2}\int_{X}\int_{X}\frac{(h(y)-h(x))f(x)g(y)+(h(x)-h(y))f(y)g(x)}{d(x,y)^{\df+2\alpha}}\mathrm{d}\mu(y)\mathrm{d}\mu(x)\\ 
\label{form-integral} &=\int_{X}\int_{X}\frac{(h(x)-h(y))f(y)g(x)}{d(x,y)^{\df+2\alpha}}\mathrm{d}\mu(y)\mathrm{d}\mu(x)
\end{align}
where \eqref{form-integral} is finite by Lemma \ref{lem:commutator_bounded} and the change of variables $(x,y)\mapsto (y,x)$.
Therefore, by Lemma \ref{lem:commutator_bounded} we find
\begin{align*}
|\nonumber\E(hf,g)-\E(f,hg)|=|\langle \K_{\alpha,h}f,g\rangle_{L^{2}}|\lesssim \Hol_{\beta}(h)\|f\|_{L^{2}}\|g\|_{L^{2}},
\end{align*}
as desired.
\end{proof}

Working as in \cite[Corollary 4.10]{GM} and \cite[Proposition 4.11]{GM} we obtain the following.

\begin{prop}\label{prop:Hol_in_Dom} 
For every $0<2\alpha<\beta\leq 1$, $\Hol_{\beta}(X,d)$ embeds continuously in $\Dom\Da.$ Moreover, for every $f\in \Hol_{\beta}(X,d), x\in X$ we have the principal value integral representation $$\Da f(x)=\int_X \frac{f(x)-f(y)}{d(x,y)^{\df+2\alpha}}\dd\mu(y),$$
and $\Da f\in L^{\infty}(X,\mu)\subset L^{2}(X,\mu)$.
\end{prop}

\begin{remark}
The integral representation in Proposition \ref{prop:Hol_in_Dom} has appeared in \cite{GSV} as fractional differentiation. Working as in \cite[Theorem 1.2]{GSV}, it can easily be seen that for $f\in \Hol_{\beta}(X,d)$ in fact one has $\Da f\in \Hol_{\beta-2\alpha}(X,d)$. This is in stark contrast to what happens with the logarithmic Dirichlet Laplacian that corresponds to the extremal case $\alpha=0$ and where any H\"older continuous function of any exponent defines a smooth vector, see \cite[Subsection 4.5]{GM}. Further, the integral representation determines $\Da$ only if $\Hol_{\beta}(X,d)$ is a core for $\Da$. We do not know whether this holds in general, but as in the case of the logarithmic Dirichlet Laplacian, this is true for special settings, like that of Riemannian manifolds, ultrametric spaces and topological groups, see \cite[Section 5]{GM}.
\end{remark}

\subsection{Schatten commutators and norm estimates}\label{sec:Schatten_com}

We now aim to examine in more detail the bounded operators $\K_{\alpha,h}$ of Lemma \ref{lem:commutator_bounded} which extend on $L^2(X,\mu)$ the commutators $[\Da,\m_h]$ of Theorem \ref{theorem:domain_module}. Specifically, the following holds.

\begin{lemma}\label{lem:compact_com}
Let $0<2\alpha<\beta\leq 1$ and $h\in \Hol_{\beta}(X,d)$. Then, the bounded operator $\K_{\alpha,h}:L^2(X,\mu)\to L^2(X,\mu)$ is compact.
\end{lemma}

\begin{proof}
The idea is to consider truncations of the kernel $K_{\alpha,h}$ of $\K_{\alpha,h}$ that will lead to compact approximations of $\K_{\alpha,h}$ in the operator norm. To this end, let $0<r\leq \diam(X,d)$ and define 
\begin{equation}\label{eq:compact_com_1}
K_{\alpha,h,r}=
\begin{cases}
(h(x)-h(y))d(x,y)^{-(\df+2\alpha)}, &\text{if } d(x,y)\geq r,\\
0, &\text{otherwise}
\end{cases}.
\end{equation}
Consider now the compact operator $\K_{\alpha,h,r}:L^2(X,\mu)\to L^2(X,\mu)$ given by
\begin{equation}\label{eq:compact_com_2}
\K_{\alpha,h,r}f(x):=\int_XK_{\alpha,h,r}(x,y)f(y)\dd \mu(y).
\end{equation}
In order to estimate $\|\K_{\alpha,h}-\K_{\alpha,h,r}\|$ one simply has to observe that, for every $x\in X$ and integer $k\geq 0$, setting $B_{x,k}=B(x,e^{-k}r)\setminus B(x,e^{-k-1}r)$ it gives
\begin{align*}
\bigl \lvert \int_{B(x,r)}\frac{h(x)-h(y)}{d(x,y)^{\df+2\alpha}} f(y) \dd \mu(y)\bigr\rvert &\leq \sum_{k=0}^{\infty} \int_{B_{x,k}} \frac{|h(x)-h(y)|}{d(x,y)^{\df+2\alpha}} |f(y)| \dd \mu(y)\\
&\lesssim r^{\beta-2\alpha}\M f(x)\Hol_{\beta}(h)\sum _{k=0}^{\infty} e^{k(2\alpha-\beta)}\\
&\lesssim r^{\beta-2\alpha}\M f(x),
\end{align*}
where $\M f$ is the Hardy--Littlewood maximal function \eqref{eq:HL}. Then, since $\||M f\|_{L^2}\lesssim \|f\|_{L^2}$ we obtain that $\|\K_{\alpha,h}-\K_{\alpha,h,r}\|\lesssim r^{\beta-2\alpha}$. Considering $r\to 0$ completes the proof.
\end{proof}

\begin{prop}\label{prop:Schatten_com}
Let $0<2\alpha<\beta\leq 1$, $h\in \Hol_{\beta}(X,d)$ and define $p(\alpha,\beta):=\ell \df(\beta-2\alpha)^{-1}$ with $\ell=\left \lceil{\df \beta^{-1}}\right \rceil +1$. Then, for every $p> p(\alpha,\beta)$ it holds that $\K_{\alpha,h}\in \Sp(L^2(X,\mu))$ with $$\|\K_{\alpha,h}\|_{\Sp}\lesssim \Hol_{\beta}(h).$$ Moreover, if $\df (\beta-2\alpha)^{-1}\geq 2$, we can choose $\ell=1$.
\end{prop}

\begin{proof}
We will approximate $\K_{\alpha,h}$ in the operator norm by finite rank operators and use \eqref{eq:singular_values_1}. To this end we choose a system of dyadic cubes $\D$ as in \eqref{eq:dyadic} with $0<\theta<1$ and $0<C_1$. The constants $c_1,M>0$ are not required for this proof. 

For brevity we shall denote the operator $\K_{\alpha,h}$ by $\K$ and its kernel by $K$. Similarly, for $n\in \mathbb N$ we denote the compact operator $\K_{\alpha,h,4C_1\theta^n}$ defined in \eqref{eq:compact_com_2} (without loss of generality we can assume $4C_1\theta\leq \diam(X,d)$) by $\K_{\theta^n}$ and its kernel by $K_{\theta^n}$. Also, we pick some $\ell\in \mathbb N$ to be determined later and consider the expectation operator $\Ex_{\ell n}$ from \eqref{eq:expectation} as well as the operator 
\begin{equation}\label{eq:Schatten_norm_1}
\F_n:=\Ex_{\ell n} \K_{\theta^n}.
\end{equation}
The operator $\F_n$ is an integral operator with kernel
\begin{equation}\label{eq:Schatten_norm_2}
F_n(x,y)=\sum_{D\in \D_{\ell n}} \left(\dashint_D K_{\theta^n}(z,y) \dd \mu (z)\right) \chi_D(x)
\end{equation}
and has finite rank that is at most $\# \D_{\ell n}\lesssim \theta^{-\df \ell n}.$ 

Our goal is to estimate the rate that $\|\K-\F_n\|$ converges to zero. For this reason it suffices to estimate, for $x\neq y \in X$, the quantity
\begin{equation}\label{eq:Schatten_norm_3}
\left\lvert K(x,y)- F_n(x,y) \right\rvert = \left\lvert \sum_{D\in \D_{\ell n}} \left(\dashint_D K(x,y)- K_{\theta^n}(z,y) \dd \mu (z)\right) \chi_D(x)\right\rvert.
\end{equation}
\vspace{0.5cm}

\noindent \textbf{Claim 1:} \textit{If $d(x,y)<2C_1\theta^n$ then} $$\left\lvert K(x,y)- F_n(x,y) \right\rvert= |K(x,y)|.$$ 

\vspace{0.5cm}

Indeed, let $D\in \D_{\ell n}$ be the unique cube that $x\in D$. Then, for every $z\in D$ we have 
\begin{align*}
d(z,y)&\leq d(z,x)+d(x,y)\\
&<2C_1\theta^{\ell n} +2C_1\theta^n\\
&\leq 4C_1\theta^n,
\end{align*}
and hence $K_{\theta^n}(z,y)=0$. This means $F_n(x,y)=0$, proving Claim 1. 
\vspace{0.5cm}

\noindent \textbf{Claim 2:} \textit{If $2C_1\theta^n\leq d(x,y)<6C_1\theta^n$ then} $$\left\lvert K(x,y)- F_n(x,y) \right\rvert\lesssim |K(x,y)|+\Hol_{\beta}(h)(\theta^{(\ell - \df -2\alpha-1+\beta)n}+\theta^{(\ell \beta-\df -2\alpha)n}).$$ 

\vspace{0.5cm}

For this, let $D\in \D_{\ell n}$ be the unique cube that $x\in D$. Then, $\left\lvert K(x,y)- F_n(x,y) \right\rvert$ consists of one averaged integral over $D$ which we decompose over $D\cap B(y,4C_1\theta^n)$ and $D\setminus B(y,4C_1\theta^n)$. In particular, $$\left\lvert K(x,y)- F_n(x,y) \right\rvert\leq |K(x,y)|+\frac{1}{\mu(D)}\int_{D\setminus B(y,4C_1\theta^n)}\left\lvert K(x,y)- K_{\theta^n}(z,y)\right\rvert \dd \mu(z).$$ Denote the second term by $J_1(x,y)$. Then, 
\begin{align*}
J_1(x,y)&= \frac{1}{\mu(D)} \int_{D\setminus B(y,4C_1\theta^n)}\left\lvert\frac{h(x)-h(y)}{d(x,y)^{\df +2\alpha}}- \frac{h(z)-h(y)}{d(z,y)^{\df +2\alpha}}\right \rvert \dd \mu(z)\\
&\leq \frac{1}{\mu(D)} \int_{D\setminus B(y,4C_1\theta^n)} |h(x)-h(y)| \left\lvert \frac{1}{d(x,y)^{\df +2\alpha}}-\frac{1}{d(z,y)^{\df +2\alpha}}\right \rvert \dd \mu(z)\\
&\qquad + \frac{1}{\mu(D)} \int_{D\setminus B(y,4C_1\theta^n)}\frac{|h(x)-h(z)|}{d(z,y)^{\df +2\alpha}}\dd \mu(z).
\end{align*}
Denote the first term of the last inequality by $J_2(x,y)$ and the second by $J_3(x,y)$. For $J_2(x,y)$ observe that $d(x,z)<2C_1\theta^{\ell n}$ since $x,z\in D$, while $d(z,y)\geq 4C_1\theta^{n}$. Therefore, we have $2d(x,z)<d(z,y)$ and from \cite[Lemma 2.3]{GSV} it holds that $$\left\lvert \frac{1}{d(x,y)^{\df +2\alpha}}-\frac{1}{d(z,y)^{\df +2\alpha}}\right \rvert\lesssim \frac{d(x,z)}{d(z,y)^{\df +2\alpha +1}}.$$ Consequently, one has
\begin{align*}
J_2(x,y)&\lesssim \frac{\Hol_{\beta}(h)}{\mu(D)}\int_{D\setminus B(y,4C_1\theta^n)} \frac{d(x,z)}{d(z,y)^{\df +2\alpha +1-\beta}}\dd \mu(z)\\
&\lesssim \Hol_{\beta}(h)\theta^{-(\df +2\alpha+1-\beta)n}\theta^{\ell n}\\
&= \Hol_{\beta}(h)\theta^{(\ell - \df -2\alpha-1+\beta)n}.
\end{align*}
Moreover, 
\begin{align*}
J_3(x,y)&\leq \frac{\Hol_{\beta}(h)}{\mu(D)}\int_{D\setminus B(y,4C_1\theta^n)} \frac{d(x,z)^{\beta}}{d(z,y)^{\df +2\alpha}}\dd \mu(z)\\
&\lesssim \Hol_{\beta}(h) \theta^{-(\df +2\alpha)n}\theta^{\ell \beta n}\\
&=\Hol_{\beta}(h)\theta^{(\ell \beta-\df -2\alpha)n}.
\end{align*}
This proves Claim 2.
\vspace{0.5cm}

\noindent \textbf{Claim 3:} \textit{If $d(x,y)\geq 6C_1\theta^n$ then} $$\left\lvert K(x,y)- F_n(x,y) \right\rvert\lesssim \Hol_{\beta}(h)(\theta^{(\ell - \df -2\alpha-1+\beta)n}+\theta^{(\ell \beta-\df -2\alpha)n}).$$

\vspace{0.5cm}

Again, let $D\in \D_{\ell n}$ be the unique cube that $x\in D$. Then, clearly $D\cap B(y,4C_1\theta^n)=\varnothing$. As a result, 
\begin{align*}
\left\lvert K(x,y)- F_n(x,y) \right\rvert &= \left\lvert\frac{1}{\mu(D)}\int_{D\setminus B(y,4C_1\theta^n)} K(x,y)- K_{\theta^n}(z,y) \dd \mu(z)\right\rvert\\
&\lesssim \Hol_{\beta}(h)(\theta^{(\ell - \df -2\alpha-1+\beta)n}+\theta^{(\ell \beta-\df -2\alpha)n}),
\end{align*}
working as before. Claim 3 has been proved.

The estimate for \eqref{eq:Schatten_norm_3} can now be used to estimate, for every $f\in L^2(X,\mu)$ and $x\in X$, the quantity
\begin{equation}\label{eq:Schatten_norm_4}
|\K f(x)-\F_nf(x)|=\left\lvert \int_X \left( K(x,y)-F_n(x,y)\right) f(y) \dd \mu(y)\right\rvert .
\end{equation}
To this end, we decompose the integral in \eqref{eq:Schatten_norm_4} over the subsets dictated in Claims $1-3$, namely 
\begin{align*}
X_1&:= \{y\in X: d(x,y)<2C_1\theta^n\},\\
X_2&:= \{y\in X: 2C_1\theta^n\leq d(x,y)<6C_1\theta^n\},\\
X_3&:= \{y\in X: d(x,y)\geq 6C_1\theta^n\}.
\end{align*}
Working as in the proof of Lemma \ref{lem:compact_com} for the integrals over $X_1$ and $X_2$, we obtain 
\begin{equation*}
|\K f(x)-\F_nf(x)|\lesssim \Hol_{\beta}(h)(\theta^{(\beta-2\alpha)n}\M f(x)+\left( \theta^{(\ell - \df -2\alpha-1+\beta)n}+\theta^{(\ell \beta-\df -2\alpha)n} \right) \|f\|_{L^2}),
\end{equation*}
where $\M f$ is the Hardy--Littlewood maximal function \eqref{eq:HL}. Specifically, this means 
\begin{equation}\label{eq:Schatten_norm_5}
\|\K - \F_n\| \lesssim \Hol_{\beta}(h)( \theta^{(\beta-2\alpha)n}+ \theta^{(\ell - \df -2\alpha-1+\beta)n}+\theta^{(\ell \beta-\df -2\alpha)n}).
\end{equation}
At this point the role of $\ell \in \mathbb N$ becomes apparent, as it is used to speed-up the averaging process in \eqref{eq:Schatten_norm_1} so that \eqref{eq:Schatten_norm_5} converges to zero as $n$ goes to infinity. For $$\ell=\left \lceil{\df \beta^{-1}}\right \rceil +1$$ it holds that $$0<\beta -2\alpha \leq \ell \beta -\df -2\alpha \leq \ell -\df - 2\alpha -1 +\beta.$$ Therefore, we have 
\begin{equation*}
\|\K - \F_n\| \lesssim \Hol_{\beta}(h)\theta^{(\beta-2\alpha)n}
\end{equation*}
and by recalling that $\F_n$ has rank at most $\# \D_{\ell n}\lesssim \theta^{-\df \ell n}$, the proof that $\K \in \Sp(L^2(X,\mu))$ for every $p> \df \ell(\beta-2\alpha)^{-1}$ is immediate. 

Finally, assuming $\df(\beta-2\alpha)^{-1}\geq 2$ yields sharper estimates. From \cite[Theorem 2.3]{Gof}, when $p\geq 2$ and $q=p/(p-1)$ is its conjugate, it holds that 
\begin{equation}\label{eq:Schatten_norm_6}
\|K\|_{\Sp}\leq \left(\int_X\left(\int_X |K(x,y)|^q\dd \mu (x)\right)^{\frac{p}{q}}\dd \mu (y) \right)^{\frac{1}{p}}.
\end{equation}
In particular, for every $p>\df(\beta-2\alpha)^{-1}$ the right hand-side of \eqref{eq:Schatten_norm_6} is less than or equal to a constant multiple of $\Hol_{\beta}(h)$.
\end{proof}

\begin{remark}
To our knowledge, it is an open problem whether inequality \eqref{eq:Schatten_norm_6} holds for $1<p<2$. In case the answer to this problem is affirmative, then in Proposition \ref{prop:Schatten_com} we can choose $\ell=1$ even if $\df (\beta-2\alpha)^{-1}< 2$.
\end{remark}

The proof of the next result is inspired by the proof of \cite[Lemma 3.1]{FLLVW}, which is about commutators with Riesz potentials on Euclidean spaces. We note that Riesz potentials correspond to negative powers of the classical Laplacian.

\begin{lemma}\label{lem:Schatten_norm_est_1}
Let $0<2\alpha<\beta\leq 1$ and $h\in \Hol_{\beta}(X,d)$. Also, let $\mathcal{D}$ be a system of dyadic cubes as in \eqref{eq:dyadic} with expectation operators $\{\Ex_n\}_{n\geq 0}$ as in \eqref{eq:expectation}. Then, for every $p>\max \{p(\alpha,\beta),1\}$ with $p(\alpha,\beta)$ as in Proposition \ref{prop:Schatten_com} and $\gamma=2\alpha+\df p^{-1}$, we have 
\begin{equation}\label{eq:Schatten_norm_est_1.1}
\sum_{n=0}^{\infty} \theta^{-\gamma p n}\|\Ex_{n+1} h -\Ex_n h\|_{L^p}^p\lesssim \|\K_{\alpha,h}\|_{\Sp}^p.
\end{equation}
Consequently, we have 
\begin{equation}\label{eq:Schatten_norm_est_1.2}
\sum_{n=0}^{\infty} \theta^{-\gamma p n}\|\Ex_n h -h\|_{L^p}^p\lesssim \|\K_{\alpha,h}\|_{\Sp}^p.
\end{equation}
\end{lemma}

\begin{proof}
The inequalities of the statement are well-defined since $h$ and each $\Ex_n h$ lie in $L^p(X,\mu)$. Also, from Proposition \ref{prop:Schatten_com} we have that $\|\K_{\alpha,h}\|_{\Sp}<\infty$. We first prove inequality \eqref{eq:Schatten_norm_est_1.1}. From property (v) of $\D$ and the Ahlfors $\df$-regularity of $\mu$ we get
\begin{align*}
\sum_{n=0}^{\infty} \theta^{-\gamma p n}\|\Ex_{n+1} h -\Ex_n h\|_{L^p}^p&= \sum_{n=0}^{\infty}\theta^{-\gamma p n}\sum_{D\in \D_n}\int_D \left \lvert\Ex_{n+1}h(x)-\Ex_n h(x) \right \rvert^p\dd \mu(x)\\
&\lesssim \sum_{n=0}^{\infty}\theta^{-2 \alpha p n}\sum_{D\in \D_n}\dashint_D \left \lvert\Ex_{n+1}h(x)-\Ex_n h(x) \right \rvert^p\dd \mu(x).
\end{align*}
Also, for every $n\geq 0$ and $D\in \D_n$ it holds that $$(\Ex_{n+1} h- \Ex_n h)\chi_D=\sum_{u=1}^{\# \C(D)-1}\langle h, h_{D,u}\rangle_{L^2} h_{D,u},$$ where the $h_{D,u}$ are the Haar wavelets supported on $D$, see Theorem \ref{thm:Haar_wav}. Then, from part (7) of Theorem \ref{thm:Haar_wav} it follows that 
\begin{align*}
\|(\Ex_{n+1} h- \Ex_n h)\chi_D\|_{L^p}&\leq \sum_{u=1}^{\# \C(D)-1} |\langle h,h_{D,u}\rangle_{L^2}| \|h_{D,u}\|_{L^p}\\
&\lesssim (M-1) \mu(D)^{\frac{1}{p} -\frac{1}{2}}|\langle h,h_{D}\rangle_{L^2}|,
\end{align*}
where $M$ is the uniform upper bound for the number of children of dyadic cubes in $\D$, and $h_D$ is the Haar wavelet among the $h_{D,u}$ with $|\langle h,h_{D}\rangle_{L^2}|$ being maximal. Then, 
\begin{equation*}
\dashint_D \left \lvert\Ex_{n+1}h(x)-\Ex_n h(x) \right \rvert^p\dd \mu(x)\lesssim \mu(D)^{-\frac{p}{2}}|\langle h,h_{D}\rangle_{L^2}|^p.
\end{equation*}
We shall further estimate the quantity $\mu(D)^{-\frac{p}{2}}|\langle h,h_{D}\rangle_{L^2}|^p$. To this end, consider a median value $m_D(h)\in \mathbb R$ of $h$ over $D$, which is a number satisfying 
\begin{align*}
\mu (\{x\in D: h(x)> m_D(h)\})&\leq \frac{\mu(D)}{2},\\
\mu (\{x\in D: h(x)< m_D(h)\})&\leq \frac{\mu(D)}{2}.
\end{align*}
The median allows us to \enquote{decompose} $D$ into the sets 
\begin{align*}
E_{D,1}&:=\{x\in D: h(x)\leq m_D(h)\},\\
E_{D,2}&:=\{x\in D: h(x)\geq m_D(h)\}
\end{align*}
that clearly satisfy $$\mu(E_{D,1})\simeq \mu(E_{D,2})\simeq \mu(D)\simeq \theta^{\df n}.$$ It will be useful to also define the sets $F_{D,1}:=E_{D,2}$ and $F_{D,2}:=E_{D,1}$, and notice that whenever $x\in E_{D,s}$ and $y\in F_{D,s}$ for $s=1,2$ one has 
\begin{equation}\label{eq:Schatten_norm_est_2}
|h(x)-m_D(h)|\leq |h(x)-h(y)|.
\end{equation}
Now writing $D=\bigsqcup_{D' \in \mathcal{C}(D)} D'$ and using properties (4), (7) of Theorem \ref{thm:Haar_wav} for $h_D$ we have
\begin{align*}
\mu(D)^{-\frac{1}{2}}|\langle h,h_{D}\rangle_{L^2}|&= \mu(D)^{-\frac{1}{2}} \left \lvert \int_D \left((h(x)-m_D(h)\right) h_D(x) \dd \mu(x)\right \rvert\\
&\leq \mu(D)^{-1}\int_D|h(x)-m_D(h)| \dd \mu(x)\\
&= \mu(D)^{-1} \sum_{D' \in \C(D)} \int_{D'}|h(x)-m_D(h)| \dd \mu(x)\\
&\leq \sum_{s=1}^{2}\mu(D)^{-1} \sum_{D' \in \C(D)} \int_{D'\cap E_{D,s}} |h(x)-m_D(h)| \dd \mu(x).
\end{align*}
In turn, denoting the summands over $s=1,2$ by $I_{D,s}$ and using \eqref{eq:Schatten_norm_est_2} we have that 
\begin{align*}
I_{D,s}&\lesssim \mu(D)^{-1} \theta^{2\alpha n} \mu (F_{D,s})\theta^{-(\df +2\alpha)n} \sum_{D' \in \C(D)} \int_{D'\cap E_{D,s}} |h(x)-m_D(h)| \dd \mu(x)\\
&\lesssim \mu(D)^{-1} \theta^{2\alpha n}\sum_{D' \in \C(D)} \int_{D'\cap E_{D,s}} \int_{F_{D,s}} \frac{|h(x)-m_D(h)|}{d(x,y)^{\df +2\alpha}}\dd \mu(y) \dd \mu(x)\\
&\leq \mu(D)^{-1} \theta^{2\alpha n}\sum_{D' \in \C(D)} \int_{D'\cap E_{D,s}} \int_{F_{D,s}} \frac{|h(x)-h(y)|}{d(x,y)^{\df +2\alpha}}\dd \mu(y) \dd \mu(x)\\
&= \mu(D)^{-1} \theta^{2\alpha n}\sum_{D' \in \C(D)} \left \lvert \int_{D'\cap E_{D,s}} \int_{F_{D,s}} \frac{h(x)-h(y)}{d(x,y)^{\df +2\alpha}}\dd \mu(y) \dd \mu(x) \right \rvert\\
&= \mu(D)^{-1} \theta^{2\alpha n} \sum_{D' \in \C(D)} \left \lvert \langle \K_{\alpha,h} \chi_{F_{D,s}}, \chi_{D'\cap E_{D,s}} \rangle \right \rvert\\&= \theta^{2\alpha n} \sum_{D' \in \C(D)} \left \lvert \langle \K_{\alpha,h} \frac{\chi_{F_{D,s}}}{\mu(D)^{\frac{1}{2}}}, \frac{\chi_{D'\cap E_{D,s}}}{\mu(D)^{\frac{1}{2}}} \rangle \right \rvert.
\end{align*}
Consequently, we obtain that 
\begin{align*}
\sum_{n=0}^{\infty} \theta^{-\gamma p n}\|\Ex_{n+1} h -\Ex_n h\|_{L^p}^p&\lesssim \sum_{n=0}^{\infty} \theta^{-2\alpha p n} \sum_{D\in \D_n} \mu(D)^{-\frac{p}{2}}|\langle h,h_{D}\rangle_{L^2}|^p\\
&\lesssim \sum_{s=1}^{2} \sum_{n=0}^{\infty} \theta^{-2\alpha p n} \sum_{D\in \D_n} I_{D,s}^p\\
&\lesssim \sum_{s=1}^{2} \sum_{D\in \D} \left( \sum_{D' \in \C(D)} \left \lvert \langle \K_{\alpha,h} \frac{\chi_{F_{D,s}}}{\mu(D)^{\frac{1}{2}}}, \frac{\chi_{D'\cap E_{D,s}}}{\mu(D)^{\frac{1}{2}}} \rangle \right \rvert \right)^p\\
&\lesssim \sum_{s=1}^{2} \sum_{D\in \D} \left \lvert \langle \K_{\alpha,h} \frac{\chi_{F_{D,s}}}{\mu(D)^{\frac{1}{2}}}, \frac{\chi_{\widetilde{D}\cap E_{D,s}}}{\mu(D)^{\frac{1}{2}}} \rangle \right \rvert^p,
\end{align*}
where $\widetilde{D}\in \C(D)$ is the one with $$\left \lvert \langle \K_{\alpha,h} \frac{\chi_{F_{D,s}}}{\mu(D)^{\frac{1}{2}}}, \frac{\chi_{\widetilde{D}\cap E_{D,s}}}{\mu(D)^{\frac{1}{2}}} \rangle \right \rvert$$ being maximal. Also, recall that $\# \C(D) \leq M$ uniformly in $D$. Then, we note that the sequences $$\{\frac{\chi_{F_{D,s}}}{\mu(D)^{\frac{1}{2}}}\}_{D\in \D},\,\, \{\frac{\chi_{\widetilde{D}\cap E_{D,s}}}{\mu(D)^{\frac{1}{2}}}\}_{D\in \D}$$ are NWO($\D$)-sequences as per Definition \ref{defn:NWO}. Therefore, from Proposition \ref{prop:NWO} we obtain the desired inequality \eqref{eq:Schatten_norm_est_1.1}, that is, $$\sum_{n=0}^{\infty} \theta^{-\gamma p n}\|\Ex_{n+1} h -\Ex_n h\|_{L^p}^p\lesssim \|\K_{\alpha,h}\|_{\Sp}^p.$$ We now focus on proving \eqref{eq:Schatten_norm_est_1.2}. First, from \eqref{eq:Schatten_norm_est_1.1} we obtain that for every $n\geq 0$ it holds 
\begin{equation}\label{eq:Schatten_norm_est_3}
\|\Ex_{n+1} h -\Ex_n h\|_{L^p}\lesssim \theta^{\gamma n} \|\K_{\alpha,h}\|_{\Sp}.
\end{equation}
Moreover, since $h\in \Hol_{\beta}(X,d)$ we have that $\|\Ex_n h -h\|_{\infty} \lesssim \Hol_{\beta}(h)\theta^{\beta n}$ and as $\mu(X)<\infty$, we get that 
\begin{equation}\label{eq:Schatten_norm_est_4}
\lim_{n}\|\Ex_n h-h\|_{L^p}=0.
\end{equation}
As a result, from \eqref{eq:Schatten_norm_est_3} and \eqref{eq:Schatten_norm_est_4} it follows that 
\begin{align*}
\|h-\Ex_n h\|_{L^p}&= \lim_m\|\Ex_m h-\Ex_n h\|_{L^p}\\
&\leq \lim_k \sum_{l=0}^{k} \|\Ex_{n+k-l+1}h - \Ex_{n+k-l}h\|_{L^p}\\
&\lesssim \lim_k \sum_{l=0}^{k} \theta^{\gamma(n+k-l)}\|\K_{\alpha,h}\|_{\Sp}\\
&\lesssim \theta^{\gamma n} \|\K_{\alpha,h}\|_{\Sp}.
\end{align*}
The proof of \eqref{eq:Schatten_norm_est_1.2} then follows from $\|h-\Ex_n h\|_{L^p}\lesssim \theta^{\gamma n} \|\K_{\alpha,h}\|_{\Sp}$ and \eqref{eq:Schatten_norm_est_1.1}, similarly as in \cite[Lemma 3.4]{FLLVW}.
\end{proof}

The next result follows from Lemma \ref{lem:Schatten_norm_est_1}, Theorem \ref{thm:adj_dyadic_cubes} and embeddings of fractional Sobolev $L^p$-spaces into H\"older spaces. Fractional Sobolev $L^p$-spaces are defined analogously as in \eqref{eq: Sobolevspace}, see \cite{GS, NPV}. However, we refrain from delving deeper into their theory.

\begin{prop}\label{prop:Sobolev_embed}
Let $0<2\alpha<\beta\leq 1$ and $h\in \Hol_{\beta}(X,d)$. Also, let $p>\max \{p(\alpha,\beta),1\}$ with $p(\alpha,\beta)$ as in Proposition \ref{prop:Schatten_com}. Then, $$\Hol_{2\alpha}(h)\lesssim \|\K_{\alpha,h}\|_{\Sp}.$$ 
\end{prop}

\begin{proof}
Let $\gamma=2\alpha+\df p^{-1}$ and consider the quantity 
\begin{equation*}
\|h\|_{\gamma,p}=\left( \frac{1}{2} \int_X\int_X \frac{|h(x)-h(y)|^p}{d(x,y)^{\df+\gamma p}}\dd \mu(y) \dd \mu(x) \right)^{\frac{1}{p}},
\end{equation*}
that is, the semi-norm of $h$ in the fractional Sobolev $L^p$-space $\mathcal{W}^{\gamma,p}(X,d,\mu),$ for details see \cite[Section 2]{GS}. From Lemma \ref{lem:Ahlfors_estimates} we obtain that $\|h\|_{\gamma,p}<\infty$. Then, as $h$ is continuous and the measure $\mu$ is strictly positive due to Ahlfors regularity, from the fractional Sobolev embedding \cite[Theorem 3.2]{GS} we get that 
\begin{equation}\label{eq:Sobolev_embed_1}
\Hol_{2\alpha}(h)\lesssim \|h\|_{\gamma,p}. 
\end{equation} 

Therefore, it suffices to prove that 
\begin{equation}\label{eq:Sobolev_embed_2}
\|h\|_{\gamma,p}\lesssim \|\K_{\alpha,h}\|_{\Sp}.
\end{equation}
For this proof we shall use a finite collection $\{\D^t\}_{t=1}^{K}$ of systems of dyadic cubes as per Theorem \ref{thm:adj_dyadic_cubes}. The parameters that we will use are $0<\theta <1$ and $C_1>0$, which are common for all $1\leq t\leq K$. The other parameters will not be used. Also, every dyadic cube $D_{n,k}^t\in \D_n^t$ for $0\leq k\leq \# \D_n^t-1$ will be centred around $x_{n,k}^t\in X$, see property (v) of dyadic cubes. 

Now, without loss of generality let us distinguish the system $\D^1$ and make an observation. Specifically, if $x\in D_{n,k}^1\in \D_n^1$ and $y\in X$ is such that $d(x,y)\leq \theta^{n} \diam(X,d)$, then 
\begin{equation}\label{eq:Sobolev_embed_3}
x\in B(x_{n,k}^1,C_1\theta^n),\,\, y\in B(x_{n,k}^1, (C_1+\diam(X,d))\theta^{n}).
\end{equation}
In order to use Theorem \ref{thm:adj_dyadic_cubes} for including the open balls \eqref{eq:Sobolev_embed_3} in dyadic cubes we proceed as follows. Consider $n_0\geq 2$ and $n_1\geq 3$ such that $\theta^{n_1}<(C_1+\diam(X,d))\theta^{n_0}\leq \theta^2.$ Then, for every $l\geq 0$ we have $$\theta^{l+n_1}<(C_1+\diam(X,d))\theta^{l+n_0}\leq \theta^{l+2}.$$ Now, if $n_1>3$, the number $(C_1+\diam(X,d))\theta^{l+n_0}$ is in one of the intervals, for $3\leq m\leq n_1$, $$(\theta^{l+(n_1-m)+3},\theta^{l+(n_1-m)+2}].$$  Then, by Theorem \ref{thm:adj_dyadic_cubes} there is some $1\leq t\leq K$ and $D^t\in \D_{l+(n_1-m)}^t$ such that $$B(x_{{l+n_0},k}^1, (C_1+\diam(X,d))\theta^{l+n_0})\subset D^t.$$ Further, by property (iii) of dyadic cubes we can assume that $D^t\in \D_l^t.$ Re-indexing, for every $n\geq n_0$ and $0\leq k\leq \# \D_n^1-1$, there is $1\leq t\leq K$ and $D^t\in \D_{n-n_0}^t$ such that 
\begin{equation}\label{eq:Sobolev_embed_4}
B(x_{{n},k}^1, (C_1+\diam(X,d))\theta^{n})\subset D^t.
\end{equation}
For convenience, let us also define $\D_{n-n_0}^t$ for $0\leq n<n_0$ by $$\D_{n-n_0}^t:=\D_0^t=\{X\}.$$ 

We have concluded with the observation and focus on proving \eqref{eq:Sobolev_embed_2}. Specifically, for every $n\geq 0$ we consider the set $B_n\subset X\times X$ defined as $$B_n:=\{(x,y): d(x,y)\leq \theta^n\diam(X,d)\}\setminus \{(x,y): d(x,y)< \theta^{n+1}\diam(X,d)\}.$$ Then, 
\begin{align*}
\|h\|_{\gamma,p}^p&=\frac{1}{2}\sum_{n=0}^{\infty} \int \int_{B_n} \frac{|h(x)-h(y)|^p}{d(x,y)^{\df+\gamma p}}\dd \mu(y) \dd \mu(x)\\
&\lesssim \sum_{n=0}^{\infty} \theta^{-(\df + \gamma p)n}\int \int_{B_n}|h(x)-h(y)|^p\dd \mu(y) \dd \mu(x)\\
&\leq \sum_{n=0}^{\infty} \theta^{-(\df + \gamma p)n}\sum_{k=0}^{\# \D_n^1-1} \int_{D_{n,k}^1}\int_{B(x_{{n},k}^1, (C_1+\diam(X,d))\theta^{n})}|h(x)-h(y)|^p\dd \mu(y) \dd \mu(x)\\
{}^{(\ast)}&\leq \sum_{t=1}^{K}\sum_{n=0}^{\infty} \theta^{-(\df + \gamma p)n}\sum_{D^t\in \D^t_{n-n_0}}\int_{D^t}\int_{D^t}|h(x)-h(y)|^p\dd \mu(y) \dd \mu(x)\\
{}^{(\ast \ast)}&\lesssim \sum_{t=1}^{K}\sum_{n=0}^{\infty} \theta^{-(\df + \gamma p)n} \sum_{D^t\in \D^t_{n-n_0}}\int_{D^t}\int_{D^t} |h(x)-\Ex_{n-n_0}^th(x)|^p+|h(y)-\Ex_{n-n_0}^th(y)|^p \dd \mu(y) \dd \mu(x)\\
&\lesssim \sum_{t=1}^{K}\sum_{n=0}^{\infty} \theta^{-(\df + \gamma p)n} \sum_{D^t\in \D^t_{n-n_0}} \mu(D^t) \int_{D^t}|h(x)-\Ex_{n-n_0}^th(x)|^p \dd \mu(x)\\
&\leq \sum_{t=1}^{K}\sum_{n=0}^{\infty} \theta^{-\gamma p n}\|\Ex_{n-n_0}^t h-h\|_{L^p}^p\\
&\lesssim \sum_{t=1}^{K}\sum_{n=0}^{\infty} \theta^{-\gamma p n}\|\Ex_{n}^t h-h\|_{L^p}^p\\
{}^{(\ast \ast \ast)}&\lesssim \|K_{\alpha,h}\|_{\Sp}^p.
\end{align*}
The inequality ($\ast$) follows from \eqref{eq:Sobolev_embed_4}, the inequality ($\ast \ast$) from the fact that the expectation operators are constant on each $D^t\in \D_{n-n_0}^t$, and the inequality ($\ast \ast \ast$) from Lemma \ref{lem:Schatten_norm_est_1}. This completes the proof.
\end{proof}

\section{Ideal spectral metric spaces}\label{sec:ISMS}
In this section we introduce the concept of \textit{ideal spectral triples} and \textit{ideal spectral metric spaces} (Definition \ref{def:p-spectral}), which is a new class of Compact Quantum Metric Spaces (CQMS). The prototypical examples will be obtained from the analysis developed in Section \ref{sec:Lap}. Moreover, in this section we shift our focus to complex algebras and complex separable Hilbert spaces. Also, we note that CQMS can be defined in the setting of operator systems, but for the sake of simplicity we shall not be that general. 

Assume that $\mathcal{A}$ is a dense $*$-subalgebra of a unital $C^*$-algebra $A$ and $L:\mathcal{A}\to \mathbb [0,\infty)$ is a seminorm that is $*$-invariant, meaning $L(a)=L(a^*)$ for all $a\in \mathcal{A}$, and $\ker L=\mathbb C 1$. Then, there is an associated (extended) metric $\rho_L$ on the state space $S(A)$, known as the \textit{Monge--Kantorovi\v{c} metric}, that is given by 
\begin{equation}\label{eq:MKmetric}
\rho_L(\varphi,\psi):=\sup \{|\varphi(a)-\psi(a)|: a\in \mathcal{A},\,\,L(a)\leq 1\}.
\end{equation}
In general, $\rho_L$ generates a topology on $S(A)$ and we say that the pair $(\mathcal{A},L)$ is a \textit{CQMS} if the topology generated by $\rho_L$ coincides with the weak $*$-topology. In that case, it also follows that $\rho_L$ attains only finite values and $L$ is called a \textit{Lip-seminorm}.

The prototypical example in the commutative setting is that of a compact metric space $(Z,\rho)$ where one sets $A=C(Z),\,\mathcal{A}=\Hol_{\beta}(Z,\rho)$ with $0<\beta\leq 1$ and $L(a)=\Hol_{\beta}(a)$. In the noncommutative setting significant examples come from spectral triples.

\begin{definition}\label{def:spectral_triple}
A spectral triple $(\mathcal{A},H,D)$ over a unital $C^*$-algebra $A$ consists of a dense $*$-subalgebra $\mathcal{A}$ of $A$, a Hilbert space $H$ and a self-adjoint operator $D:\Dom D\to H$ such that 
\begin{enumerate}
\item $\mathcal{A}\subset \mathcal{B}(H)$;
\item $D$ has compact resolvent;
\item for all $a\in \mathcal{A}$, $a:\Dom D\to \Dom D$ and $[D,a]$ extends to a bounded operator.
\end{enumerate}
Moreover, we say that $(\mathcal{A},H,D)$ is $p$-summable for some $p>0$ if $(1+|D|)^{-p}$ is trace class.
\end{definition}
Given a spectral triple $(\mathcal{A},H,D)$ over $A$ we obtain the seminorm $L_D:\mathcal{A}\to [0,\infty)$ defined as 
\begin{equation}\label{eq:STSeminorm}
L_D(a):=\|\cl ([D,a])\|,
\end{equation}
where $\cl ([D,a]):H\to H$ is the bounded extension of $[D,a]$ on $H$. The associated metric $\rho_{L_D}$ is known as \textit{Connes metric} and if it generates the weak $*$-topology on $S(A)$, the pair $(\mathcal{A},L_D)$ is called a \textit{spectral metric space.} 

Over the last two decades, CQMS has been an area of intensive study, with important examples including the Sierpi\'nski gasket \cite{CGIS,LLL}, length spaces admitting local diffusion \cite{HKT}, and abstract (non-intrinsic) approximation methods for compact metric spaces \cite{CI}. Other examples are hyperbolic \cite{OR} and finitely generated nilpotent-by-finite group $C^*$-algebras \cite{CR}, compact connected Lie groups \cite{Ar}, $C^*$-algebras related to quantum tori \cite{Lat} and one-dimensional solenoids \cite{FLP}, Podle\'s spheres \cite{AK,AKKyed}, and certain crossed products by finitely generated discrete groups \cite{HSZW,Kl}, among other. 

In the context of fractional Dirichlet Laplacians it is more suitable to introduce a refined version of spectral metric spaces that utilises the Schatten regularity of the commutators with H\"older continuous functions (see Proposition \ref{prop:Schatten_com}). In fact, a useful characteristic of the aforementioned Schatten regularity is that Schatten ideals are special examples of symmetrically normed ideals. Although in this paper we focus only on Schatten ideals we shall introduce the general notion, thus allowing future investigations.

Recall that a symmetrically normed ideal is a proper non-zero two-sided ideal $\mathcal{I}$ of $B(H)$ with norm $\|\cdot \|_{\mathcal{I}}$ such that 
\begin{enumerate}[(i)]
\item $\|SRT\|_{\mathcal{I}}\leq \|S\|\|R\|_{\mathcal{I}}\|T\|$, for all $S,T\in B(H)$ and $R\in \mathcal{I}$;
\item $\mathcal{I}$ is a Banach space with the norm $\|\cdot \|_{\mathcal{I}}$.
\end{enumerate}

\begin{definition}[Ideal spectral triples]\label{def:p-spectral}
Let $\mathcal{I}$ be a symmetrically normed ideal in $B(H)$ and $(\mathcal{A},H,D)$ be a spectral triple over a unital $C^*$-algebra $A$ such that for every $a\in \mathcal{A}$ it holds 
$$\cl ([D,a])\in \mathcal{I}.$$ We will call such $(\mathcal{A},H,D)$ an $\mathcal{I}$\textit{-spectral triple}. If the Monge--Kantorovi\v{c} metric $\rho_{L_{D,\mathcal{I}}}$, associated to the seminorm $L_{D,\mathcal{I}}:\mathcal{A}\to [0,\infty)$ given by $$L_{D,\mathcal{I}}(a):=\|\cl ([D,a])\|_{\mathcal{I}},$$ generates the weak $*$-topology on $S(A)$, we will call $(\mathcal{A}, L_{D,\mathcal{I}})$ an $\mathcal{I}$\textit{-spectral metric space}. If $\mathcal{I}=\Sp(H)$ for some $p\geq 1$, for brevity we shall write $L_{D,p}$ instead of $L_{D,\Sp}$.
\end{definition}

Before delving into the quantum metric space structure of $(X,d,\mu)$ and dynamics thereon, we point that from now on we view $L^2(X,\mu)$ as a complex Hilbert space (by tensoring it over $\mathbb R$ with $\mathbb C$) and for every $0<\beta \leq 1$ we assume that $\Hol_{\beta}(X,d)$ consists of $\mathbb C$-valued functions. Then, for every $0<\alpha <1$ the relevant $\alpha$-fractional Sobolev space on $(X,d,\mu)$ is derived from the bilinear form 
\begin{equation}\label{eq:complexform}
(f,g)\mapsto \frac{1}{2}\int_X\int_X \frac{\overline{(f(x)-f(y))}(g(x)-g(y))}{d(x,y)^{\df+2\alpha}}\dd\mu(y)\dd\mu(x).
\end{equation}
We shall denote \eqref{eq:complexform} again by $\E$ and the associated (complexified) $\alpha$-fractional Dirichlet Laplacian by $\Da$. Moreover, for every $0<2\alpha <\beta \leq 1$ and $h\in \Hol_{\beta}(X,d)$, we have that $\Da h$ admits the same principal value integral representation as that in Proposition \ref{prop:Hol_in_Dom}, and the commutator $[\Da, \m_h]$ extends on $L^2(X,\mu)$ to the bounded integral operator $\K_{\alpha,h}$ whose formula is the same as that in Lemma \ref{lem:commutator_bounded}.

\subsection{Lip-seminorms from Schatten commutators} Following Proposition \ref{prop:Schatten_com}, there is a constant $p(\alpha,\beta)>0$ for which $\K_{\alpha,h} \in \Sp(L^2(X,\mu))$ whenever $p>p(\alpha,\beta)$. 
Summing up the results of Section \ref{sec:Lap} for the fractional Dirichlet Laplacians we obtain the following. 

\begin{thm}\label{thm:SchattenCQMS}
For every $0<2\alpha <\beta \leq 1$, the triple $(\Hol_{\beta}(X,d), L^2(X,\mu), \Da)$ is a spectral triple over $C(X)$ that is $p$-summable for $p>\frac{\df}{2\alpha}$. The threshold $\frac{\df}{2\alpha}$ is sharp, meaning that the spectral triple is not $\frac{\df}{2\alpha}$-summable. Moreover, for every $p>p(\alpha,\beta)$ the seminorm $L_{\Da,p}:\Hol_{\beta}(X,d)\to [0,\infty)$ given by $$L_{\Da,p}(h):=\|\K_{\alpha,h}\|_{\Sp}$$ is $*$-invariant and $\ker L_{\Da, p}=\mathbb C 1$. In addition, if $p>\max \{p(\alpha,\beta),1\}$ then for every $h\in \Hol_{\beta}(X,d)$ it holds that 
\begin{equation}\label{eq:biHol}
\Hol_{2\alpha}(h)\lesssim L_{\Da,p}(h)\lesssim \Hol_{\beta}(h).
\end{equation}
Consequently, the pair $(\Hol_{\beta}(X,d),L_{\Da,p})$ is an $\Sp$-spectral metric space.
\end{thm}

\subsection{Fractals of low spectral dimension}\label{sec:low_spectral_dim}
Focusing on $\mathcal{S}_2$-spectral metric spaces we show that the associated Monge--Kantorovi\v{c} metric admits an explicit description in terms of the spectral theory of fractional Dirichlet Laplacians. With the assumptions of Theorem \ref{thm:SchattenCQMS} the commutators $\K_{\alpha,h}$ are Hilbert--Schmidt only if the fractal dimension $\df$ is small enough. However, in special cases by choosing $h$ in a different dense $*$-subalgebra of $C(X)$ instead of $\Hol_{\beta}(X,d)$ it can happen that $\K_{\alpha,h}$ is Hilbert--Schmidt, independently of how large $\df$ is. As we will see, the existence of such subalgebras depends on the fine structure of the metric space $(X,d)$ and in fact on the spectral dimension, for which we now present some background.

Recall that as the parameter $0<\alpha<1$ grows, the domain $\Ha$ of $\E$ shrinks and eventually after some $\alpha \geq 1$, even if $\E$ is still well-defined, its domain might not be dense in $L^2(X,\mu)$. This is measured by the following metric-invariant. 

\begin{definition}[{\cite{GHL}}]\label{def: walkdim}
The \emph{walk dimension} of $(X,d)$ is defined as the supremum $\dw$ of all $2 \alpha>0$, for which there is a $\|\cdot \|_{\E}$-closed subspace $\mathcal{F}_{\alpha}\subset \Ha$ such that $\E:\mathcal{F}_{\alpha}\times \mathcal{F}_{\alpha}\to \mathbb R$ is a regular Dirichlet form.
\end{definition}

This is indeed a metric-invariant of $(X,d)$ since $\df$ is a metric-invariant and $\mu$ is comparable to the $\df$-dimensional Hausdorff measure. Further, $\dw$ is called the walk dimension as for spaces admitting \textit{fractal diffusion} (see \cite[Definition 3.5]{Bar}), the mean time of the associated process to move away to a distance $r$ from the origin is of order $r^{\dw}$. In general, it holds that $\dw \geq 2$. This follows from simply setting $\mathcal{F}_{\alpha}$ to be the $\|\cdot \|_{\E}$-closure of Lipschitz continuous functions, for $0<\alpha <1$, or from Proposition \ref{eq: regularform}.

In the classification of metric spaces in terms of walk dimension, the two extremes are manifolds for which $\dw=2$ independently of the topological dimension, and ultrametric spaces for which $\dw=\infty$. In between lie several fractals, for which typically one has $\dw>2$. In fact, for spaces satisfying the chain-condition it holds that $2\leq \dw \leq \df +1$. A remarkable fact is that for every pair of parameters $p,q$ such that $2\leq p\leq q+1$, there is a metric space with fractal dimension $q$ and walk dimension $p$. For all these we refer to \cite{Bar, GHL}.

The relation between the fractal and walk dimensions is of particular interest and is captured by the following quantity.

\begin{definition}[{\cite{Bar}}]
The quantity $\ds:= 2\df/\dw$ is called the \textit{spectral dimension} of $(X,d)$.
\end{definition}

From a spectral standpoint, for \textit{fractional metric spaces} (see \cite[Definition 3.2]{Bar}) with a fractal diffusion, like compact Riemannian manifolds or the Sierpi\'nski gasket, there is a Weyl law for the generator of the diffusion in terms of $\ds$. Namely, the eigenvalue counting function $N(\lambda):=\# \{\text{eigenvalues}\leq \lambda\}$ of the generator is of order $\lambda^{\ds /2},$ see \cite[p. 46]{Bar}.

From a stochastic point of view the spectral dimension $\ds$ measures how likely is for the process to return to its origin. Of particular importance in this context is the dichotomy $\ds <2$ and $\ds \geq 2$. In the first case the process will be recurrent, whereas in the second the process almost surely never returns to its origin. For this we refer to \cite[Corollary 3.29]{Bar}.

In this section we focus on $(X,d)$ for which $\ds<2$. First, we note that if $(X,d)$ is a compact Riemannian manifold, then $\ds=\df$, thus one recovers the classical Weyl law estimates and behaviour of Brownian motion. In that case, $\ds<2$ only when $X$ is $1$-dimensional. Now, if $(X,d)$ is an ultrametric space then $\ds=0$. Moreover, if $(X,d)$ is the Sierpi\'nski gasket in topological dimension $n$, then $d_s=2\log(n+1)/\log(n+3)$, see \cite{RT}. In fact, for finitely ramified fractals one has $d_s<2$, see \cite[Proposition 3.42]{Bar}. A striking fact is that the case $d_s<2$ is way more ubiquitous than it first seems. It can be found  for instance in the context of self-similar groups \cite{Nek}, quantum gravity \cite{Hor}, Einstein relations on resistance networks \cite{Frei}, karst networks \cite{HR} as well as in the study of hemoproteins and ferredoxin \cite{HCT}. 

From now on we assume that $\ds <2$ and focus on $\alpha$ such that 
\begin{equation}\label{eq:alpha,beta}
0<2\alpha<\min\{(\dw-\df)/2,1\},\qquad \beta(\alpha)=\df/2+2\alpha.
\end{equation}
Then, since $2\beta(\alpha)<\dw$ we have that $C(\Eba)=\mathcal{H}_{\beta(\alpha)} \cap C(X)$ is $\|\cdot\|_{\infty}$-dense in $C(X)$. Note that since the measure $\mu$ is strictly positive due to Ahlfors regularity, every class in $\mathcal{H}_{\beta(\alpha)}$ can be represented by at most one function in $C(X)$. Now the fractional Sobolev embedding \cite[Theorem 3.2]{GS} yields a continuous embedding $\mathcal{H}_{\beta(\alpha)} \hookrightarrow \Hol_{2\alpha}(X,d)$. As a result, every class in $\mathcal{H}_{\beta(\alpha)}$ has a unique representative in $C(X)$ and $C(\Eba)$ is $\|\cdot \|_{\Eba}$-dense in $\mathcal{H}_{\beta(\alpha)}$. For clarity, we shall often use the notation $f\in C(\Eba)$ to highlight that $f$ is a function and $f\in \mathcal{H}_{\beta(\alpha)}$ to indicate that we consider the class of $f$ in $L^2(X,\mu)$.

\begin{lemma}\label{lem:module_Sob}
For every $h\in C(\Eba),f\in \mathcal{H}_{\beta(\alpha)}$ we have $\m_{h}f=hf\in \mathcal{H}_{\beta(\alpha)}$. In particular, $C(\Eba)$ forms a $\|\cdot \|_{\infty}$-dense $*$-subalgebra of $C(X)$.
\end{lemma}

\begin{proof}
The proof is similar to that of Lemma \ref{lem:module_Hol}. 
\end{proof}

The proof of the next result is straightforward.

\begin{lemma}\label{lem:commutator_bounded_Sob}
For every $h\in C(\Eba)$, the kernel
\[K_{\alpha,h}:X\times X\setminus D\to \mathbb{R},\,\,K_{\alpha,h}(x,y):=\frac{h(x)-h(y)}{d(x,y)^{\df +2\alpha}},\]
defines a Hilbert--Schmidt operator $\K_{\alpha,h}:L^2(X,\mu)\to L^2(X,\mu)$ given by $$\K_{\alpha,h}f(x)=\int_X\frac{h(x)-h(y)}{d(x,y)^{\df+2\alpha}}f(y)\dd\mu(y),$$ with Hilbert--Schmidt norm $\|\K_{\alpha,h}\|_{\mathcal{S}_2}=\Eba(h,h)^{1/2}.$ 
\end{lemma}

Consequently, with a proof similar to that of Theorem \ref{thm:SchattenCQMS} we obtain the following.

\begin{prop}\label{prop:domain_module_Sob}
It holds that $(C(\Eba), L^2(X,\mu), \Da)$ is a $p$-summable spectral triple over $C(X)$ for $p>\frac{\df}{2\alpha}$. The threshold is sharp. Moreover, the seminorm $L_{\Da,2}:C(\Eba)\to [0,\infty)$ given by $$L_{\Da,2}(h):=\|\K_{\alpha,h}\|_{\mathcal{S}_2}$$ is $*$-invariant, $\ker L_{\Da, 2}=\mathbb C 1$ and the pair $(C(\Eba), L_{\Da,2})$ is an $\mathcal{S}_2$-spectral metric space.
\end{prop}

\begin{proof}
In order to see that $(C(\Eba), L^2(X,\mu), \Da)$ is a spectral triple first note that for every $h\in C(\Eba)$ it holds $\m_h(\Ha)\subset \Ha$. Then, working as in Theorem \ref{thm:SchattenCQMS} we get $\m_{h}(\Dom \Da )\subset \Dom\Da$ and $[\Da,\m_{h}]:\Dom\Da \to L^{2}(X,\mu)$ extends to the Hilbert--Schmidt operator $\K_{\alpha,h}$ of \ref{lem:commutator_bounded_Sob}, similarly as in Theorem \ref{theorem:domain_module}. The summability is a property of $\Da$ so it immediately follows from Corollary \ref{cor:Weyl_law}. Also, $L_{\Da,2}$ is $*$-invariant and $\ker L_{\Da, 2}=\mathbb C 1$ as in Theorem \ref{thm:SchattenCQMS}. 

To see that the Monge--Kantorovi\v{c} metric $\rho_{L_{\Da, 2}}$ generates the weak $*$-topology on $S(C(X))$ note that $L_{\Da,2}(h)= \Eba(h,h)^{1/2}$ (see Lemma \ref{lem:commutator_bounded_Sob}) and so from the fractional Sobolev embedding \cite[Theorem 3.2]{GS} we get that 
\begin{equation}\label{eq:domain_module_Sob}
\Hol_{2\alpha}(h)\lesssim L_{\Da,2}(h). 
\end{equation} 
This implies that the topology generated by $\rho_{L_{\Da, 2}}$ is weaker than the weak $*$-topology. The proof is complete by a general argument showing that $\rho_{L_{\Da, 2}}$ is also stronger than the weak $*$-topology, see \cite[Proposition 1.4]{Rie}.
\end{proof}

\begin{remark}
In the generality of Proposition \ref{prop:domain_module_Sob} we can only obtain \eqref{eq:domain_module_Sob} and not a double inequality as \ref{eq:biHol}. The reason is that the space $C(\Eba)$ is typically small for containing a non-trivial H\"older space. Further, the square root of the metric $\rho_{L_{\Da, 2}}$ restricted on pure states is a resistance metric on $X$ in the sense of Kigami \cite{Kig}. Here we aim to derive explicit formulas for $\rho_{L_{\Da, 2}}$ on the whole $S(C(X))$.
\end{remark}

We now study the Monge--Kantorovi\v{c} metric associated to the spectral triple of Proposition \ref{prop:domain_module_Sob}. First, we focus on the collection $M_0(X)$ of complex-valued finite Borel measures on $X$ of zero mean, in other words, $$M_0(X):=\{\varphi \in C(X)^*: \varphi (1) =0\}.$$ Note that $M_0(X)$ is weak $*$-closed in $C(X)^*$ and its utility for us is that it contains the set $$ S(C(X))-S(C(X)):=\{\varphi-\psi:\varphi,\psi \in S(C(X))\}.$$

\begin{lemma}\label{lem:loc_C}
Let $x_0\in X$ be arbitrary. Consider the $\|\cdot \|_{\infty}$-closed subspace of $C(X)$, $$C(X,x_0):=\{f\in C(X): f(x_0)=0\},$$ and the linear surjection $s_{x_0}:C(X)\to C(X,x_0)$ given by $s_{x_0}(f)=f-f(x_0)$. Then, for the dual Banach space $C(X,x_0)^*$ it holds $$M_0(X)=s_{x_0}^*(C(X,x_0)^*).$$
\end{lemma}

\begin{proof}
Let $\varphi \in C(X,x_0)^*$. Then, $s_{x_0}^*\varphi(1)=\varphi(s_{x_0}(1))=\varphi(0)=0$. For the converse, consider the inclusion $j_{x_0}:C(X,x_0)\to C(X)$ and let $\varphi \in M_0(X)$. Then, $j_{x_0}^*\varphi \in C(X,x_0)^*$ and for every $f\in C(X)$ we have
\begin{align*}
s_{x_0}^* j_{x_0}^* \varphi (f)&= j_{x_0}^*\varphi (s_{x_0}(f))\\
&=\varphi(s_{x_0}(f))\\
&=\varphi(f)-f(x_0)\varphi(1)\\
&=\varphi(f).
\end{align*} 
Consequently, we have $\varphi=s_{x_0}^* j_{x_0}^* \varphi$ and the proof is complete. 
\end{proof}

The goal now is to embed $M_0(X)$ in localisations of the fractional Sobolev space $\mathcal{H}_{\beta(\alpha)}$. To this end, consider an arbitrary $x_0\in X$ and the subspace $C(\Eba,x_0)$ of $C(\Eba)$ defined as $$C(\Eba,x_0):= \{ f\in C(\Eba): f(x_0)=0\}.$$ Note that the linear surjection $s_{x_0}:C(X)\to C(X,x_0)$ of Lemma \ref{lem:loc_C} descends to a linear surjection $s_{x_0}: C(\Eba)\to C(\Eba,x_0)$ since the bilinear form $\Eba$ is $s_{x_0}$-invariant. Then, with a generic argument involving also the map $j_{x_0}$ of Lemma \ref{lem:loc_C} one can show that if $\varphi \in C(X,x_0)^*$ vanishes on $C(\Eba,x_0)$ then $\varphi=0$, since $C(\Eba)$ is $\|\cdot \|_{\infty}$-dense in $C(X)$. In other words, $C(\Eba,x_0)$ is $\|\cdot \|_{\infty}$-dense in $C(X,x_0)$ too.

Further, observe that the bilinear form $\Eba$ defines an inner product on $C(\Eba,x_0)$ since the only $f\in C(\Eba,x_0)$ with $\Eba(f,f)=0$ is $f=0$. Also, from Lemma \ref{lem:commutator_bounded_Sob} the associated norm is the Hilbert--Schmidt norm $L_{\Da,2}$ of the commutators. Therefore, we can define the Hilbert space $$\Hba:=\overline{C(\Eba,x_0)}^{L_{\Da,2}}.$$ Notice that up to unitary equivalence the space $\Hba$ does not depend on $x_0\in X$. Indeed, for $x_1\in X$ the map $C(\Eba,x_0) \to C(\Eba,x_1), f\mapsto f-f(x_1)$ is a linear bijection that preserves the inner product $\Eba$.

\begin{lemma}\label{lem:loc_M}
For every $x_0\in X$ there is a continuous embedding $J_{x_0}:M_0(X)\to \Hba^*$ with $J_{x_0}(\varphi)(f)=\varphi(f)$ for every $\varphi\in M_0(X)$ and $f\in C(\Eba,x_0)$.
\end{lemma}

\begin{proof}
Denote by $\Hol_{2\alpha}(X,d,x_0)$ the subset of $\Hol_{2\alpha}(X,d)$ consisting of functions vanishing at $x_0$. Equip $\Hol_{2\alpha}(X,d,x_0)$ with the norm $\Hol_{2a}$ (which is only a seminorm on $\Hol_{2\alpha}(X,d)$). Since $\diam(X,d)<\infty$, there is a continuous embedding $J_{2,x_0}:\Hol_{2\alpha}(X,d,x_0)\to C(X,x_0)$. This also means that $\Hol_{2\alpha}(X,d,x_0)$ is a Banach space since it embeds as a closed subspace of $\Hol_{2\alpha}(X,d)$. Further, from the fractional Sobolev embedding (see \eqref{eq:domain_module_Sob}) there is a continuous embedding $J_{1,x_0}: \Hba \to \Hol_{2\alpha}(X,d,x_0)$ such that $J_{1,x_0}(f)=f$ for every $f\in C(\Eba,x_0)$.

Now the continuous embedding $J_{2,x_0}J_{1,x_0}:\Hba \to C(X,x_0)$ has dense image since $J_{2,x_0}J_{1,x_0}(C(\Eba,x_0))=C(\Eba,x_0)$ and the latter is $\|\cdot \|_{\infty}$-dense in $C(X,x_0)$. Consequently, the dual map $$J_{1,x_0}^*J_{2,x_0}^*:C(X,x_0)^*\to \Hba^*$$ is a continuous embedding. Finally, the maps $j_{x_0},s_{x_0}$ of Lemma \ref{lem:loc_C} satisfy $j_{x_0}^*s_{x_0}^*=\id_{C(X,x_0)^*}$ and hence the map $$J_{x_0}:=J_{1,x_0}^*J_{2,x_0}^*j_{x_0}^*:M_0(X)\to \Hba^*$$ is the desired continuous embedding.
\end{proof}

We can now calculate the Monge--Kantorovi\v{c} metric in terms of the spectrum of $\Delta_{\beta(\alpha)}$. Note that $\Delta_{\beta(\alpha)}$ has compact resolvent even when $\beta(\alpha) \geq 1$, see Remark \ref{rem:com_emb}. Therefore, let $(\lambda_{\beta(\alpha),n})_{n\geq 0}$ be the sequence of (positive) eigenvalues of $\Delta_{\beta(\alpha)}$ in increasing order, counting multiplicities, and $(h_{\beta(\alpha),n})_{n\geq 0}$ be the associated orthonormal eigenbasis for $L^2(X,\mu)$. From the fractional Sobolev embedding the eigenbasis can be chosen in $C(\Eba)$. Note that $\lambda_{\beta(\alpha),0}=0$ is a simple eigenvalue and $h_{\beta(\alpha),0}=\mu(X)^{-1/2}\cdot 1$, where $1\in L^{\infty}(X,\mu)$ is the constant function.

\begin{thm}\label{thm:Connes_distance}
For every $\tau_1,\tau_2 \in S(C(X))$ we have that $$\rho_{L_{\Da, 2}}(\tau_1,\tau_2)=\left(\sum_{n\in \mathbb N} \lambda^{-1}_{\beta(\alpha),n} |(\tau_1-\tau_2)(h_{\beta(\alpha),n})|^2\right)^{1/2}.$$
\end{thm}

\begin{proof}
Let $x_0\in X$ be arbitrary and $L_{\Da, 2}^*$ be the dual norm on $\Hba^*$ given by $$L_{\Da, 2}^*(\varphi):=\sup \{ |\varphi (f)|: f\in \Hba,\,\, L_{\Da, 2}(f)\leq 1\}.$$ Also, consider the isometric anti-isomorphism $F:\Hba\to \Hba^*$ given by $$F(f)(g):=\Eba(f,g).$$ Then, for an arbitrary $\varphi \in \Hba^*$ there exists a unique $f\in \Hba$ such that $F(f)=\varphi$ and
\begin{equation}\label{eq:Connes_distance_1}
L_{\Da, 2}^*(\varphi)=\Eba(f,f)^{1/2}.
\end{equation}
We aim to find such $f$ explicitly. For this, it suffices to solve $F(f)=\varphi$ at the level of $\mathcal{H}_{\beta(\alpha)}.$ Namely, we will find $h\in \mathcal{H}_{\beta(\alpha)}$ such that 
\begin{equation}\label{eq:Connes_distance_2}
\Eba(h,g)=s_{x_0}^*\varphi(g), \qquad (g\in \mathcal{H}_{\beta(\alpha)}).
\end{equation}
In particular, \eqref{eq:Connes_distance_2} holds for all $g\in \Hba$. Note that for such $g$ we have $s_{x_0}^*\varphi(g)=\varphi(g)$ and from the $s_{x_0}$-invariance of $\Eba$ we also have $$\Eba(s_{x_0}(h),g)=\Eba(h,g)=\varphi(g).$$ In other words, $F(s_{x_0}(h))=\varphi$ and hence $f=s_{x_0}(h)$. 

For \eqref{eq:Connes_distance_2} it suffices to find $h\in \mathcal{H}_{\beta(\alpha)}$ so that for every $n\in \mathbb N$ it holds $$\Eba(h,h_{\beta(\alpha),n})=s_{x_0}^*\varphi(h_{\beta(\alpha),n}).$$ Since $\Eba(h,h_{\beta(\alpha),n})= \langle h, \Delta_{\beta(\alpha)} h_{\beta(\alpha),n} \rangle_{L^2}= \lambda_{\beta(\alpha),n}\langle h, h_{\beta(\alpha),n} \rangle_{L^2}$ such an $h$ should have $$\langle h_{\beta(\alpha),n}, h \rangle_{L^2} =\lambda_{\beta(\alpha),n}^{-1} \overline{s_{x_0}^*\varphi(h_{\beta(\alpha),n})}.$$ Consequently, it suffices to choose $$h= \sum_{n=1}^{\infty} \lambda_{\beta(\alpha),n}^{-1} \overline{s_{x_0}^*\varphi(h_{\beta(\alpha),n})} h_{\beta(\alpha),n}.$$ Then, we have that 
\begin{align*}
L_{\Da, 2}^*(\varphi)&=\Eba(s_{x_0}(h),s_{x_0}(h))^{1/2}\\
&=\Eba(h,h)^{1/2}\\
&= \langle \Delta_{\beta(\alpha)}^{1/2}h,\Delta_{\beta(\alpha)}^{1/2}h \rangle_{L^2}^{1/2}\\
&= \left(\sum_{k\in \mathbb N} \lambda_{\beta(\alpha),n}^{-1}|s_{x_0}^*\varphi(h_{\beta(\alpha),n})|^2\right)^{1/2}.
\end{align*}
Consider now the continuous embedding $J_{x_0}:M_0(X)\to \Hba^*$ from Lemma \ref{lem:loc_C} and let $\tau_1, \tau_2 \in S(C(X))$. Then, we have that 
\begin{align*}
L_{\Da, 2}^*(J_{x_0}(\tau_1-\tau_2))&=\sup \{|J_{x_0}(\tau_1-\tau_2)(f)|: f\in \Hba,\,\, L_{\Da, 2}(f)\leq 1\}\\
&= \sup \{|J_{x_0}(\tau_1-\tau_2)(f)|: f\in C(\Eba,x_0),\,\, L_{\Da, 2}(f)\leq 1\}\\
&= \sup \{|(\tau_1-\tau_2)(f)|: f\in C(\Eba,x_0),\,\, L_{\Da, 2}(f)\leq 1\}\\
&=\sup \{|(\tau_1-\tau_2)(f)|: f\in C(\Eba),\,\, L_{\Da, 2}(f)\leq 1\}\\
&= \rho_{L_{\Da, 2}}(\tau_1,\tau_2).
\end{align*}
As a result, we obtain that 
\begin{align*}
\rho_{L_{\Da, 2}}(\tau_1,\tau_2)&=\left(\sum_{n\in \mathbb N}\lambda_{\beta(\alpha),n}^{-1}|s_{x_0}^*J_{x_0}(\tau_1-\tau_2)(h_{\beta(\alpha),n})|^2\right)^{1/2}\\
&= \left(\sum_{n\in \mathbb N} \lambda^{-1}_{\beta(\alpha),n} |(\tau_1-\tau_2)(h_{\beta(\alpha),n})|^2\right)^{1/2},
\end{align*}
The last equality holds since $s_{x_0}^*J_{x_0}(\tau_1-\tau_2)(h_{\beta(\alpha),n})=(\tau_1-\tau_2)(h_{\beta(\alpha),n})$ as $s_{x_0}(h_{\beta(\alpha),n})\in C(\Eba,x_0)$ and $(\tau_1-\tau_2)(s_{x_0}(h_{\beta(\alpha),n}))= (\tau_1-\tau_2)(h_{\beta(\alpha),n}).$ This completes the proof.
\end{proof}

\begin{remark}
An interesting fact is that Theorem \ref{thm:Connes_distance} yields an isometric embedding of $S(C(X))$ into $\ell^2(\mathbb N)$. First assume without loss of generality that $\mu(X)=1$ and denote by $\tau_{\mu}$ the associated state. Then, since $\Im \Delta_{\beta(\alpha)} = \ker \Ex_0$, where $\Ex_0$ is the projection onto $\mathbb C \cdot 1$ (see proof of \cite[Proposition 3.4]{GM}) we have that $\tau_{\mu}(h_{\beta(\alpha),n})=0$ for all $n\in \mathbb N$. In particular for every $\tau\in S(C(X))$ it holds $$\|\left( \lambda^{-1/2}_{\beta(\alpha),n} \tau(h_{\beta(\alpha),n})\right)_{n\in \mathbb N}\|_{\ell^2}=\rho_{L_{\Da, 2}}(\tau,\tau_{\mu})<\infty.$$ Then, the map $J:S(C(X))\to \ell^2(\mathbb N)$ given by $$J(\tau)=\left( \lambda^{-1/2}_{\beta(\alpha),n} \tau(h_{\beta(\alpha),n})\right)_{n\in \mathbb N}$$ is well-defined and isometric since $\|J(\tau_1)-J(\tau_2)\|_{\ell^2}=\rho_{L_{\Da, 2}}(\tau_1,\tau_2).$ Moreover, it preserves convexity and so $J(S(C(X)))$ is a convex compact subset of $\ell^2(\mathbb N).$
\end{remark}

We now illustrate how Theorem \ref{thm:Connes_distance} can be used in practice.

\begin{example}\label{exm:SFT}
Consider an integer $N\geq 2$ and the $N$\textit{-shift} space $X=\{1,\ldots, N\}^{\mathbb N}$. Equipped with the product topology, $X$ is a totally disconnected compact Hausdorff space. Further, the finite words $w=w_1\ldots w_n$ with $w_k\in \{1,\ldots N\}$ are associated to cylinder sets $$C_w:=\{x\in X: x_k=w_k, \,\,\text{for }1\leq k\leq n\}$$ that form a clopen basis for the topology on $X$. For $\lambda >1$, this topology can also be generated by the ultrametric $$d(x,y):=\lambda^{-\inf\{k-1:x_k\neq y_k\}}.$$ Note that for every $x\in X$ and $n\in \mathbb N$ it holds $B(x,\lambda^{-n})=C_{x_1\ldots x_n}.$ Moreover, the Bernoulli measure $\mu$ defined on cylinder sets by $\mu(C_w):=N^{-|w|}$, where $|w|$ is the length of $w$, is Ahlfors $\df$-regular for $$\df=\frac{\log N}{\log \lambda}.$$ Finally, $(X,d)$ being an ultrametric space means $\dw=\infty$, hence $\ds=0$.

Now for $0<2\alpha<1$ and $\beta(\alpha)=\df/2+2\alpha$ (as in Theorem \ref{thm:Connes_distance}) we aim to spectrally decompose $\Delta_{\beta(\alpha)}.$ Making the convention that $X$ is the cylinder set associated to the empty word of length zero, we define for every $n\geq 0$ the subspace $V_n\subset L^2(X,\mu)$ spanned by characteristic functions $\chi_{C_w}$ with $|w|=n$. Notice that each $V_n\subset V_{n+1}$ and the union $$C^{\infty}(X):=\bigcup_{n\geq 0} V_n$$ is dense in $L^2(X,\mu)$. Then, working as in Proposition \ref{prop:Hol_in_Dom}, the fact that functions in $C^{\infty}(X)$ are $\gamma$-H\"older for any $\gamma>0$ (even $\gamma>1$) implies that for $f\in C^{\infty}(X)$ we can write $$\Delta_{\beta(\alpha)}f(x)=\int_X \frac{f(x)-f(y)}{d(x,y)^{\df+2\beta(\alpha)}} \dd \mu (y).$$ 

Now the collection of cylinder sets forms a system of dyadic cubes as per Subsection \ref{sec:Dyadiccubes}. For every $n\geq 0$, consider the projections $\Ex_n$ of $L^2(X,\mu)$ onto $V_n$ as well as the projections $Q_n:=\Ex_{n+1}-\Ex_n$. The latter have rank $N^n(N-1)$.  According to \cite[Example 3.3]{GGM}, for $n\geq 0$, the family $\{h_{C_w,u}:|w|=n,\,\,1\leq u\leq N-1\}$ with $$h_{C_w,u}:=N^{|w|/2}\sum_{k=1}^{N}e^{\frac{2\pi i u (k-1)}{N}} \chi_{C_{wk}},$$ is a Haar wavelet orthonormal basis of $\Im Q_n$, see Theorem \ref{thm:Haar_wav} for properties of wavelets. Then, working as in \cite[Subsection 4.1]{GGM} we see that the $Q_n$'s are the spectral projections of $\Delta_{\beta(\alpha)}$ and each eigenfunction $h_{C_w,u}$ has associated eigenvalue $$\lambda_{\beta(\alpha),C_w,u}:=\lambda^{2\beta(\alpha)|w|}+\frac{N-1}{N(\lambda^{2\beta(\alpha)}-1)}(\lambda^{2\beta(\alpha)|w|}-1).$$ The sum in the formula of Theorem \ref{thm:Connes_distance} can then be computed over the pairs $(C_w,u).$
\end{example}

\subsection{Crossed products and Schatten ideals}\label{sec:Crossed_Schatten} We aim to build Monge--Kantorovi\v{c} metrics associated to $\Sp$-spectral metric spaces and dynamical systems. We note that our point here is not to be as general as possible but only to showcase the utility of the results of the previous sections in the setting of dynamical systems. For this endeavour we will require the $C^*$-algebraic notion of crossed products, which encode the orbits spaces of dynamical systems.

Let $Z$ be a compact metric space, $\Gamma$ be a countable discrete group and $c:\Gamma \to \text{Aut}(C(Z))$ be an action of $\Gamma$ on $C(Z)$ by $*$-automorphisms. Note this is equivalent to an action of $\Gamma$ by homeomorphisms on $Z$. We denote by $C_{c}(\Gamma, C(Z))$ the set of finitely supported functions on $\Gamma$ with values in $C(Z)$. It is a convolution $*$-algebra in the operations
\begin{equation}
\label{staralgebra}
f^{*}(\gamma):=\overline{c_{\gamma}(f(\gamma^{-1}))}, \quad f\star g(\gamma)=\sum_{\delta\in \Gamma}f(\delta)c_{\delta}(g(\delta^{-1}\gamma)).
\end{equation}
The $*$-algebra $C_{c}(\Gamma,C(Z))$ admits a family of $*$-representations $\{\pi_{z}\}_{z\in Z}$ on $\ell^{2}(\Gamma)$ parametrized by $Z$, defined by
\[\pi_{z}(f)\xi(\gamma):=\sum_{\delta\in\Gamma}c_{\gamma^{-1}}(f(\delta))(z)\xi(\delta^{-1}\gamma).\]
Then, the reduced $C^{*}$-algebra $C(Z)\rtimes_{r}\Gamma$ is the completion of $C_{c}(\Gamma,C(Z))$ in the norm
\[\|f\|:=\sup_{z\in Z}\|\pi_{z}(f)\|_{B(\ell^{2}(\Gamma))}.\] 
In what follows, we will write
\begin{equation}
\label{eq:crossedprodrep}
\pi:C(Z)\rtimes_{r}\Gamma\to C(Z,B(\ell^{2}(\Gamma))),\quad \pi(f)(z):=\pi_{z}(f),
\end{equation} and thus view $C(Z)\rtimes_{r}\Gamma$ as a subspace of $C(Z,B(\ell^{2}(\Gamma)))$.

If $Z=\{\text{pt}\}$ we obtain the \textit{reduced group} $C^*$\textit{-algebra} $C_r^*(\Gamma):=\mathbb C \rtimes_{r} \Gamma$. In this case the convolution operation in \eqref{staralgebra} gives an injective $*$-representation
\[\pi:C^{*}_{r}(\Gamma)\to B(\ell^{2}(\Gamma)),\]
known as the \emph{left regular representation}. In the sequel, for $\gamma\in\Gamma$ we shall write $\lambda_\gamma$ for the indicator function at $\gamma\in\Gamma$, and then $\pi(\lambda_{\gamma})$ will act on $\ell^2(\Gamma)$ as left translation by $\gamma$. Further, for $f\in C_{c}(\Gamma)$ we have $f=\sum_{\gamma\in\Gamma}f(\gamma)\lambda_{\gamma}$.

At this point, having introduced crossed products, we highlight that the gist is to merge CQMS structures on $C(Z)$ and $C^*_r(\Gamma)$ into a CQMS structure on $C(Z)\rtimes_r \Gamma$. We now present the CQMS structures on $C^*_r(\Gamma)$ that we are interested in.

\begin{definition}
\label{ptranslationbounded}
Let $\Gamma$ be a discrete group with unit $e$ and $1\leq p\leq\infty$. A \emph{proper translation $p$-summable function} is a map $\ell:\Gamma \to [0,\infty)$ such that 
\begin{enumerate}
\item $\ell(\gamma)=0$ if and only if $\gamma=e$;
\item $\ell(\gamma)=\ell(\gamma^{-1})$ for all $\gamma \in \Gamma$;
\item $\ell$ is \emph{translation $p$-summable}: for all $\gamma \in \Gamma$ the function $\ell_{\gamma}(\eta):=|\ell(\eta)-\ell(\gamma^{-1}\eta)|$ is in $\ell^{p}(\Gamma)$;
\item $\ell$ is \emph{proper}: for all $t\in [0,\infty)$ we have that $\# \ell^{-1}(t)<\infty$.
\end{enumerate}
\end{definition}

Given a proper translation function as in Definition \ref{ptranslationbounded} we consider the essentially self-adjoint operator $M_{\ell}:C_{c}(\Gamma) \to \ell^{2}(\Gamma)$ given by
\[M_{\ell}(f)(\gamma):=\ell(\gamma)f(\gamma).\]
Then, it is clear that for $f\in C_{c}(\Gamma)$ and $\gamma,\eta\in\Gamma$ we have 
\begin{equation*}
[M_{\ell}, \pi(\lambda_\gamma)]f(\eta)=(\ell(\eta)-\ell(\gamma^{-1}\eta))(\pi(\lambda_{\gamma})f)(\eta),
\end{equation*}
which extends to the bounded operator $M_{\ell_{\gamma}}\pi(\lambda_{\gamma})\in B(\ell^2(\Gamma))$. Here $M_{\ell_{\gamma}}$ is the diagonal operator given by multiplication by the $\ell^{p}(\Gamma)$-function $\ell_{\gamma}$, and hence $M_{\ell_{\gamma}}\in \mathcal{S}_{p}(\ell^2(\Gamma))$. In particular, 
\[\mathrm{cl}([M_{\ell},\pi(\lambda_{\gamma})])=M_{\ell_{\gamma}}\pi(\lambda_{\gamma})\in \mathcal{S}_{p}(\ell^2(\Gamma)).\] For brevity we shall write $\Sp:=\Sp(\ell^2(\Gamma))$. Note that $\mathcal{S}_{\infty}$ stands for $B(\ell^2(\Gamma)).$ Eventually, we can define the seminorm $L_{\ell,p}:C_c(\Gamma)\to [0,\infty)$,
\begin{equation}\label{eq:L_l}
L_{\ell,p}(f):= \left\| \mathrm{cl}([M_{\ell},\pi(f)]\right\|_{\Sp}=\left\|\sum_{\gamma\in\Gamma} f(\gamma)M_{\ell_{\gamma}}\pi(\lambda_\gamma)\right\|_{\Sp}.
\end{equation}

For the remainder of this section we assume that the associated Monge--Kantorovi\v{c} metric $\rho_{L_{\ell,p}}$ generates the weak $*$-topology on $S(C_r^*(\Gamma))$. In our terminology, for $p<\infty$ the pair $(C_c(\Gamma),L_{\ell,p})$ is an $\Sp$-spectral metric space. Furthermore, we assume that we are given a CQMS structure $(\mathcal{A},L)$ on $C(Z)$. In order to assemble these data into a CQMS structure on  $C(Z)\rtimes_{c}\Gamma$ we follow \cite[Theorem 3.6]{AKK}. The seminorm $L$ can be extended to a \textit{horizontal} seminorm 
\begin{equation}\label{eq:L_H}
L_H:C_c(\Gamma, \mathcal{A})\to [0,\infty),\quad L_H(f):=\sup_{\gamma \in \Gamma} L(f(\gamma)).
\end{equation}
Further, we write $\mathrm{d}_{\ell}:C_{c}(\Gamma, \mathcal{A})\to C(Z,\mathcal{S}_{p})$ for the derivation determined by the formula
\[\mathrm{d}_{\ell}^{z}(f):=\mathrm{cl}([M_{\ell},\pi_{z}(f)]).\] We extend $L_{\ell,p}$ to a \textit{vertical} seminorm $L_{V,p}$ 
by setting
\begin{equation}\label{eq:L_V}
L_{V,p}:C_c(\Gamma, \mathcal{A})\to [0,\infty),\quad L_{V,p}(f):=\|\mathrm{d}_{\ell}(f)\|_{C(Z,\mathcal{S}_{p})}=\sup_{z\in Z}\left\|\mathrm{d}_{\ell}^{z}(f)\right\|_{\mathcal{S}_{p}}.
\end{equation}
For $L_H$ and $L_{V,p}$ to make sense it is not needed that $L$ or $L_{\ell,p}$ generate the weak $*$-topology on $S(C(Z))$ and $S(C_r^*(\Gamma))$, respectively. We equip $C_{c}(\Gamma,\mathcal{A})$ with the seminorm
\begin{equation}\label{eq:LL}
\mathcal{L}_p(f):=\max \{L_{V,p}(f),L_H(f),L_H(f^*)\}.
\end{equation}

We now introduce three structural maps and the conditions they need to satisfy in order for $\mathcal{L}_p$ to generate the weak $*$-topology on $C(Z)\rtimes_r \Gamma$, according to \cite[Theorem 3.6]{AKK}. 
\begin{enumerate}
\item The first is the \emph{Berezin transform} associated to a finite set $F\subset \Gamma$. This map is given by
\[\beta_{F}:C_{c}(\Gamma,\mathcal{A})\to C_{c}(\Gamma,\mathcal{A}),\quad \beta_{F}(f)(\gamma):=\frac{ \# (F\cap\gamma F)}{\#F}f(\gamma),\]
and we need to show that $\beta_{F}$ is $\mathcal{L}_p$-bounded, that is \begin{equation}
\label{eq:condition1}
\mathcal{L}_p(\beta_{F}(f))\leq C_{F}\mathcal{L}_p(f).
\end{equation}
We will in fact show that \eqref{eq:condition1} holds with $C_{F}=1$.
\item The second map we need is the $*$-homomorphism
\begin{align*}
\delta:C_{c}(\Gamma,\mathcal{A})\to C_{c}(\Gamma,\mathcal{A})\otimes^{\mathrm{alg}}C_{c}(\Gamma),\quad\delta(f)(\gamma):=f(\gamma)\otimes\lambda_{\gamma}.
\end{align*}
For $a\in \mathcal{A}$ and $\gamma\in\Gamma$ denote by $a\lambda_{\gamma}\in C_{c}(\Gamma,\mathcal{A})$ the function that is supported at $\gamma$ where it takes the value $a$. One verifies that 
\begin{equation}
\label{eq:elementarytensor}
\delta(a\lambda_{\gamma})=a\lambda_{\gamma}\otimes \lambda_{\gamma}.
\end{equation} 
Since functions of the form $a\lambda_{\gamma}$ generate $C_{c}(\Gamma,\mathcal{A})$, Equation \eqref{eq:elementarytensor} determines the map $\delta$. We need to show that for every $C^{*}$-norm contraction $\eta:C(Z)\rtimes_{r}\Gamma\to \mathbb{C}$ the map $(\eta\otimes 1)\circ \delta: C_{c}(\Gamma,\mathcal{A})\to C_{c}(\Gamma)$ is contractive for the respective seminorms, that is
\begin{equation}
\label{eq:condition2}
L_{\ell,p}((\eta\otimes 1)\circ \delta(f))\leq \mathcal{L}_p(f).
\end{equation}
\item Lastly, we need to show that the evaluation maps 
\[E_{g}:C_{c}(\Gamma,\mathcal{A})\to \mathcal{A},\quad f\mapsto f(g), \,\, g\in\Gamma,\]
are seminorm contractions. This rather straightforward since
$$L(E_{g}f)=L(f(g))\leq\sup_{\gamma}L(f(\gamma))=L_{H}(f)\leq \mathcal{L}_p(f).$$
\end{enumerate}
In order to verify conditions (1) and (2) of \cite[Theorem 3.6]{AKK}, we need both the projective and injective tensor product of Banach spaces and their properties, as well as the technical Lemma \ref{lem:wisacontraction} below. 

We denote by $\otimes_{\varepsilon}$ the injective tensor product and by $\otimes_{\pi}$ the projective tensor product. For an introduction to those tensor products we refer to \cite{Ryan}. We should indicate that the notation $\otimes_{\pi}$ is unrelated to the representation $\pi$ in \eqref{eq:crossedprodrep}. Moreover let $\{p_{s}\}_{s\in\Gamma}$ be the family of rank-one projections associated to the orthonormal basis $\{e_{\gamma}\}_{\gamma \in\Gamma}$ of $\ell^{2}(\Gamma)$ and recall that the family of unitaries $\{\pi(\lambda_{\gamma})\}_{\gamma \in\Gamma}$ generates $C^{*}_{r}(\Gamma)$. The injective tensor norm has the following structural property: For any Banach space $X$, the map 
\[C(Z)\otimes_{\varepsilon} X\rightarrow C(Z, X),\quad (f\otimes x)(z):=f(z)x\]
is an isometric isomorphism. In particular, for Banach spaces $X,Y$ there are isomorphisms $C(Z,X)\otimes_{\varepsilon} Y\simeq C(Z,X\otimes_{\varepsilon} Y)\simeq C(Z)\otimes_{\varepsilon}X\otimes_{\varepsilon}Y$, which will be used repeatedly.

For the next lemma we use results about Schur multipliers and Grothendieck's Theorem on Schur multipliers on $\mathcal{S}_{\infty}$. 

\begin{lemma}
\label{lem:wisacontraction}
Let  $1\leq p \leq\infty$. The map
\[w:\mathcal{S}_{p}\otimes_{\pi} \mathcal{S}_{\infty}\to \mathcal{S}_{p}\otimes_{\varepsilon} \mathcal{S}_{\infty},\quad w(S\otimes T):=\sum_{s,t}p_{s}Sp_{t}\otimes \pi(\lambda_{s})T\pi(\lambda_{t^{-1}}). \]
is a contraction.
\end{lemma}
\begin{proof}
By definition of the injective tensor norm we have
\[
\left\|w(S\otimes T)\right\|_{\varepsilon}=\sup_{\phi\in B^{*}_{1}}\left\|\sum_{s,t} p_{s}Sp_{t}\phi\left(\pi(\lambda_{s})T\pi(\lambda_{t^{-1}})\right)\right\|_{\mathcal{S}_{p}},
\]
where we have written $B^{*}_{1}$ for the unital ball in the dual space $\mathcal{S}_{\infty}^{*}.$ We claim that the map 
$$w_{\phi}(S):=\sum_{s,t} p_{s}Sp_{t}\phi\left(\pi(\lambda_{s})T\pi(\lambda_{t^{-1}})\right),$$
is a Schur multiplier on $\mathcal{S}_{p}$ and $\|w_{\phi}\|_{\mathcal{S}_{p}}\leq\|T\|_{\mathcal{S}_{\infty}}$. First suppose $\phi$ is a state, so that 
\[\phi(\pi(\lambda_{s})T\pi(\lambda_{t^{-1}}))=\langle \pi(\lambda_{s^{-1}}),T\pi(\lambda_{t^{-1}})\rangle_{\phi},\]
where $\langle\cdot,\cdot\rangle_{\phi}$ denotes the inner product in the GNS-representation of $\mathcal{S}_{\infty}$ associated to $\phi$. Grothendieck's Theorem \cite[Theorem 5.11]{Pisier}, \cite[Theorem 1.7]{LafSal} then gives that
\begin{align*}
\label{contraction}
\|w_{\phi}(S)\|_{\mathcal{S}_{\infty}}&=\left\|\sum_{s,t} p_{s}Sp_{t}\phi\left(\pi(\lambda_{s})T\pi(\lambda_{t^{-1}})\right)\right\|_{\mathcal{S}_{\infty}}\\
&\leq \|S\|_{\mathcal{S}_{\infty}}\sup_{s\in \Gamma}\|\pi(\lambda_{s^{-1}})\|_{\phi}\sup_{t\in \Gamma}\|T\pi(\lambda_{t^{-1}})\|_{\phi}\\
&\leq\|S\|_{\mathcal{S}_{\infty}}\|T\|_{\mathcal{S}_{\infty}}.
\end{align*}
This proves that $\|w_{\phi}\|_{\mathcal{S}_{\infty}}\leq \|T\|_{\mathcal{S}_{\infty}}$ and since $\|w_{\phi}\|_{\mathcal{S}_{p}}\leq \|w_{\phi}\|_{\mathcal{S}_{\infty}}$ (see \cite[Remark 1.4]{LafSal}), the statement follows. Now it is well-known that a general element $\phi\in B_1^{*}$ can be written as $\phi=\sum_{k=0}^{3}c_{k}\phi_{k}$, with each $\phi_{k}$ being a state and $\sum_{k=0}^{3}|c_{k}|=\|\phi\|\leq 1$. Therefore, the above estimation can be done for all $\phi\in B^{*}_{1}$. Thus for any $\xi\in \mathcal{S}_{p}\otimes^{\mathrm{alg}} \mathcal{S}_{\infty}$ and for any representation $\xi=\sum S_{i}\otimes T_{i}$ we have
\[\|w(\xi)\|_{\varepsilon}\leq \sum_{i}\|S_{i}\|_{\mathcal{S}_{p}}\|T_{i}\|_{\mathcal{S}_{\infty}},\]
so that
\[\|w(\xi)\|_{\varepsilon}\leq \inf \left\{\sum_{i}\|S_{i}\|_{\mathcal{S}_{p}}\|T_{i}\|_{\mathcal{S}_{\infty}}:\xi=\sum S_{i}\otimes T_{i}\right\}=:\|\xi\|_{\pi},\]
which completes the proof. 
\end{proof}
\begin{remark} In \cite{AKK} the authors consider the unitary operator
\[W\in B(\ell^{2}(\Gamma\times \Gamma)),\quad W(e_{s}\otimes e_{t}):=e_{s}\otimes \pi(\lambda_{s})e_{t}=e_{s}\otimes e_{st}\]
and the induced automorphism
\[\mathrm{Ad}(W):B(\ell^{2}(\Gamma\times \Gamma))\to B(\ell^{2}(\Gamma\times \Gamma)), \quad F\mapsto WFW^{*}.\]
In our arguments below, we will use the map $w$ and its properties as a replacement for $\mathrm{Ad}(W)$.
\end{remark}

We now show that both $\delta$ and $\beta_{F}$ can be defined in terms of the contraction $w$, and are compatible with the representation $\pi$ and the derivation $\mathrm{d}_{\ell}$.
\begin{lemma} 
\label{lem:slice}
For $1\leq p\leq \infty$, the map
\[D:C(Z,\mathcal{S}_{p})\to C(Z,\mathcal{S}_{p})\otimes_{\varepsilon}\mathcal{S}_{\infty},\quad D(R)(z):=w(R(z)\otimes 1),\]
is a contraction and there are commutative diagrams
$$\xymatrix{C_c(\Gamma,\mathcal{A})\ar[r]^{\delta}\ar[d]^{\pi}& C_c(\Gamma,\mathcal{A})\otimes^{\mathrm{alg}}C_{c}(\Gamma)\ar[d]^{\pi\otimes \pi} & C_c(\Gamma,\mathcal{A})\ar[r]^{\delta}\ar[d]^{\mathrm{d}_{\ell}}& C_c(\Gamma,\mathcal{A})\otimes^{\mathrm{alg}}C_{c}(\Gamma)\ar[d]^{\mathrm{d}_{\ell}\otimes \pi} \\
C(Z,\mathcal{S}_{\infty})\ar[r]^{D}&C(Z,\mathcal{S}_{\infty})\otimes_{\varepsilon}\mathcal{S}_{\infty}&C(Z,\mathcal{S}_{p})\ar[r]^{D}&C(Z,\mathcal{S}_{p})\otimes_{\varepsilon}\mathcal{S}_{\infty}.}$$
\end{lemma}
\begin{proof}
We have
\begin{align*}
\|D(R)\|_{C(Z,\mathcal{S}_{p})\otimes_{\varepsilon}\mathcal{S}_{\infty}}&=\sup_{z\in Z}\|D(R)(z)\|_{\Sp\otimes_{\varepsilon}\mathcal{S}_{\infty}}\\
&=\sup_{z\in Z}\|w(R(z)\otimes 1)\|_{\Sp\otimes_{\varepsilon}\mathcal{S}_{\infty}}\\
&\leq \sup_{z\in Z}\|R(z)\otimes 1\|_{\Sp\otimes_{\pi}\mathcal{S}_{\infty}}\\
&=\|R\|_{C(Z,\Sp)},
\end{align*}
so $D$ is a contraction. Commutation of the diagrams is verified by direct computation. For both diagrams, it suffices to consider functions of the form $a\lambda_{\gamma}$, for $a\in\mathcal{A}$ and $\gamma\in\Gamma$. We have
\begin{align*}
(\pi\otimes\pi)\delta(a\lambda_{\gamma})=\pi(a\lambda_{\gamma})\otimes\pi(\lambda_{\gamma})=\pi(a)\pi(\lambda_{\gamma})\otimes\pi(\lambda_{\gamma})=D(\pi(a\lambda_{\gamma})).
\end{align*}
Similarly,
\begin{equation*}
(\mathrm{d}_{\ell}\otimes\pi)(\delta(a\lambda_\gamma))=\pi(a)M_{\ell_{\gamma}}\pi(\lambda_{\gamma})\otimes\pi(\lambda_{\gamma})=D(\pi(a)M_{\ell_{\gamma}}\pi(\lambda_{\gamma}))=D(\mathrm{d}_{\ell}(a\lambda_{\gamma})),
\end{equation*}
where $M_{\ell_{\gamma}}$ is viewed as the constant function in $C(Z,\Sp)$. This finishes the proof.
\end{proof}
A finite subset $F\subset\Gamma$ defines a unit vector $\xi_{F}\in\ell^{2}(\Gamma)$ and associated vector state $\chi_{F}:\mathcal{S}_{\infty}\to\mathbb{C}$ via
\[\xi_{F}:=|F|^{-1/2}\sum_{s\in F}e_{s}\in \ell^{2}(\Gamma), \quad \chi_{F}(R):=\langle \xi_{F},R\xi_{F}\rangle.\]
For $1\leq p \leq \infty$ we consider the contraction
\[B_{F}:C(Z,\mathcal{S}_{p})\to C(Z,\mathcal{S}_{p}),\quad B_{F}(R):=(1\otimes\chi_{F})\circ D(R).\]
\begin{lemma} 
\label{lem:Berezin}
For $1\leq p \leq \infty$ there are commutative diagrams
$$\xymatrix{C_c(\Gamma,\mathcal{A})\ar[r]^{\beta_{F}}\ar[d]^{\pi}& C_c(\Gamma,\mathcal{A})\ar[d]^{\pi} & C_c(\Gamma,\mathcal{A})\ar[r]^{\beta_{F}}\ar[d]^{\mathrm{d}_{\ell}}& C_c(\Gamma,\mathcal{A})\ar[d]^{\mathrm{d}_{\ell}} \\
C(Z,\mathcal{S}_{\infty})\ar[r]^{B_{F}}&C(Z,\mathcal{S}_{\infty})&C(Z,\mathcal{S}_{p})\ar[r]^{B_{F}}&C(Z,\mathcal{S}_{p}).}$$
\end{lemma}
\begin{proof}
We again verify the identities for functions of the form $a\lambda_{\gamma}\in C_{c}(\Gamma,\mathcal{A})$ with $a\in\mathcal{A}$ and $\gamma\in\Gamma$. We have
\begin{align*}
\pi(\beta_{F}(a\lambda_{\gamma}))&=\pi(a)\frac{\#(F\cap\gamma F)}{\#F}\pi(\lambda_{\gamma})=\frac{\pi(a\lambda_{\gamma})}{\#F}\sum_{s,t\in F}\langle e_{s},\pi(\lambda_{\gamma})e_{t}\rangle\\
&=\pi(a\lambda_{\gamma})\langle\xi_{F},\pi(\lambda_{\gamma})\xi_{F}\rangle=(1\otimes\chi_{F})\circ D(\pi(a\lambda_{\gamma}))=B_{F}(\pi(a\lambda_{\gamma})).
\end{align*}
Similarly,
\begin{align*}
\mathrm{d}_{\ell}(\beta_{F}(a\lambda_{\gamma}))&=\frac{\#(F\cap\gamma F)}{\#F}\mathrm{d}_{\ell}(a\lambda_{\gamma})=\frac{\mathrm{d}_{\ell}(a\lambda_{\gamma})}{\#F}\sum_{s,t\in F}\langle e_{s},\pi(\lambda_{\gamma})e_{t}\rangle\\
&=\mathrm{d}_{\ell}(a\lambda_{\gamma})\langle\xi_{F},\pi(\lambda_{\gamma})\xi_{F}\rangle=(1\otimes\chi_{F})\circ D(\mathrm{d}_{\ell}(a\lambda_{\gamma}))=B_{F}(\mathrm{d}_{\ell}(a\lambda_{\gamma})),
\end{align*}
as desired.
\end{proof}

\begin{thm}\label{thm:CQMS_crossed}
Let $Z$ be a compact metric space, $\Gamma$ an amenable, countable discrete group acting on $C(Z)$ and equipped with a proper translation $p$-summable function $\ell$, for $1\leq p\leq \infty$. If $(\mathcal A, L)$ and $(C_c(\Gamma),L_{\ell,p})$ are CQMS structures on $C(Z)$ and $C_r^*(\Gamma)$, and $\mathcal{A}$ is $\Gamma$-invariant, then the seminorm $\mathcal{L}_p: C_c(\Gamma, \mathcal{A})\to [0,\infty)$ given by $$\mathcal{L}_p(f):=\max \{L_{V,p}(f),L_H(f),L_H(f^*)\}$$ provides $C(Z)\rtimes_{r} \Gamma$ with a CQMS structure. 
\end{thm}
\begin{proof} We need to verify conditions (1) and (2) of \cite[Theorem 3.6]{AKK}, condition (3) was established already. For condition (1) we need to check that for every finite subset $F\subset \Gamma$, the Berezin transform
\begin{align*}
\beta_{F}:C_{c}(\Gamma,\mathcal{A})\to C_{c}(\Gamma,\mathcal{A})
\end{align*}
is bounded for the seminorm $\mathcal{L}_p$. 
To this end, using Lemma \ref{lem:Berezin} we have
\begin{align*}
L_{V,p}(\beta_{F}(f))&=\sup_{z\in Z}\|\mathrm{d}_{\ell}(\beta_{F}(f))(z)\|_{\mathcal{S}_{p}}=\sup_{z\in Z}\|B_{F}(\mathrm{d}_{\ell}(f))(z)\|_{\mathcal{S}_{p}}\\
&\leq\sup_{z\in Z} \|\mathrm{d}_{\ell}(f)(z)\|_{\mathcal{S}_{p}}=L_{V,p}(f).
\end{align*}
The estimate for $L_{H}$ follows from the fact that $\frac{\#(F\cap \gamma F)}{\#F}\leq 1$ for all $\gamma\in \Gamma$. 

For condition (2) first note that for Banach spaces $X,Y$ the flip map
\[\sigma:X\otimes_{\varepsilon} Y\to Y\otimes_{\varepsilon}X,\quad x\otimes y\mapsto y\otimes x,\]
is an isometry. Let $\eta: C(Z)\rtimes_{r}\Gamma\to \mathbb{C}$ be a linear contraction and estimate
\begin{align*}
L_{\ell,p}((\eta\otimes 1 )(\delta(f)))&=\|\mathrm{d}_{\ell}(\eta\otimes 1 )(\delta(f))\|_{\mathcal{S}_{p}}\\
&=\|(\eta\otimes 1)\circ(1\otimes \mathrm{d}_{\ell})\circ\delta(f)\|_{\mathcal{S}_{p}}\\
&\leq \|(1\otimes \mathrm{d}_{\ell})\circ\delta(f)\|_{C(Z)\rtimes_{r}\Gamma\otimes_{\varepsilon}\mathcal{S}_{p}}\\
&=\|(\pi\otimes \mathrm{d}_{\ell})\circ\delta(f)\|_{C(Z,\mathcal{S}_{\infty})\otimes_{\varepsilon}\mathcal{S}_{p}}\\
&=\sup_{z\in Z}\|(\pi_{z}\otimes \mathrm{d}_{\ell})\circ\delta(f)\|_{\mathcal{S}_{\infty}\otimes_{\varepsilon}\mathcal{S}_{p}}\\
&=\sup_{z\in Z}\|\sigma ((\mathrm{d}_{\ell}^{z}\otimes \pi)\circ\delta(f))\|_{\mathcal{S}_{\infty}\otimes_{\varepsilon}\mathcal{S}_{p}}\\
&=\sup_{z\in Z}\|(\mathrm{d}_{\ell}^{z}\otimes \pi)\circ\delta(f)\|_{\mathcal{S}_{p}\otimes_{\varepsilon}\mathcal{S}_{\infty}}\\
&=\|(\mathrm{d}_{\ell}\otimes \pi)\circ\delta(f)\|_{C(Z,\mathcal{S}_{p})\otimes_{\varepsilon}\mathcal{S}_{\infty}}\\
\textnormal{by Lemma \ref{lem:slice}}&=\|D(\mathrm{d}_{\ell}(f))\|_{C(Z,\mathcal{S}_{p})\otimes_{\varepsilon}\mathcal{S}_{\infty}}\\
&\leq\|\mathrm{d}_{\ell}(f)\|_{C(Z,\mathcal{S}_{p})}\\
&=L_{V,p}(f).
\end{align*}
This completes the proof.
\end{proof}

\subsection{Crossed products by Pontryagin duals}\label{sec:Pon_dual}
We apply the technique of Subsection \ref{sec:Crossed_Schatten} to study dynamics on $(X,d,\mu)$. Specifically, we now assume that $X$ is a compact group with unit $1_X$, $d$ is a bi-invariant metric and $\mu$ is the bi-invariant normalised Haar measure. We also assume that $\mu$ is Ahlfors regular. For instance, compact Lie groups and groups admitting hyperbolic automorphisms \cite{Ger1} fall into this category. Also, let $0<\alpha<1/2$ and for brevity we shall only work with functions in $\Hol_1(X,d)$, also called Lipschitz continuous functions.

First we make an observation which seems folklore, but since we could not find it in the literature we present its proof.

\begin{prop}\label{prop:Lip_PW}
The irreducible summands in the Peter--Weyl decomposition of $L^2(X,\mu)$ are spanned by Lipschitz continuous functions.
\end{prop}

\begin{proof}
Following Subsection 5.4 and the proof of Proposition 5.13 in \cite{GM}, one can construct a sequence of bounded operators $F_n:L^2(X,\mu)\to \Hol_1(X,d)$ such that $\|F_nf-f\|_{L^2}\to 0$ as $n\to \infty$ and $F_n$ act by scalars on the irreducible summands in the Peter--Weyl decomposition of $L^2(X,\mu)$. Since every summand is finite dimensional, the conclusion follows.
\end{proof}

Moreover, working as in \cite[Section 5.4]{GM} in the study of the logarithmic Dirichlet Laplacian we see that $\Da$ acts as a scalar on the irreducible summands in the Peter--Weyl decomposition of $L^2(X,\mu)$. Then, from Proposition \ref{prop:Lip_PW} we obtain the following.

\begin{cor}\label{cor:Lip_PW}
The eigenfunctions of $\Da$ are Lipschitz continuous.
\end{cor}

From now on we assume that $X$ is also abelian. In that case, the eigenfunctions of $\Da$ are precisely the characters forming the discrete Pontryagin dual group $\hat{X}$, which from the above are Lipschitz continuous. Using Theorem \ref{thm:SchattenCQMS} and the Fourier transform on $X$ we can then produce tractable proper translation $p$-summable functions on $\hat{X}$, for large enough $p>1$. 

\begin{remark}
We believe this idea can be utilised further. Namely, for constructing canonical functions $\ell$ on a discrete abelian group $\Gamma$ as in Definition \ref{ptranslationbounded} one should first aim to construct \enquote{singular} integral operators on the dual $\widehat{\Gamma}$. The passage to $\widehat{\Gamma}$ allows to use analytic tools for metric-measure spaces. Then, the singularity type of the kernels will be reflected to desired translation properties of $\ell$ (e.g. related to symmetrically normed ideals).
\end{remark}

Specifically, denote the eigenvalue of a character $\varphi \in \hat{X}$ by $\lambda_{\alpha,\phi}\geq 0$. To this end, the first step is the following.

\begin{prop}\label{prop:length_function}
The map $\ell^{\alpha}:\hat{X} \to [0,\infty)$ given by $$\ell^{\alpha}(\varphi):=\lambda_{\alpha,\varphi}$$ is a proper translation $p$-summable function, for $p>p(\alpha,1)$.
\end{prop}

\begin{proof}
Since the $0$-eigenspace of $\Da$ is $\mathbb C  1_{\hat{X}}$, property (i) follows. For property (ii) note that for every $\varphi \in \hat{X}$, the fact that $\varphi$ is a group homomorphism from $X$ to $\mathbb T$ and $d,\mu$ are invariant under the transformation $y\mapsto y^{-1}$ gives
\begin{align*}
\ell^{\alpha}(\varphi)&= \Da \varphi (1_X)\\
&= \int_X\frac{\varphi(1_X)-\varphi(y)}{d(1_X,y)^{\df+2\alpha}}\dd \mu(y)\\
&= \int_X\frac{\varphi^{-1}(1_X)-\varphi^{-1}(y)}{d(1_X,y)^{\df+2\alpha}}\dd \mu(y)\\
&= \ell^{\alpha}(\varphi^{-1}).
\end{align*}
For property (iii) let $\varphi,\psi\in \hat{X}$ be arbitrary. Then, 
\begin{align*}
\ell^{\alpha}(\varphi \psi)&=\int_X\frac{\varphi(1_X)\psi(1_X)-\varphi(y)\psi(y)}{d(1_X,y)^{\df+2\alpha}} \dd \mu(y)\\
&=\varphi(1_X)\int_X \frac{\psi(1_X)-\psi(y)}{d(1_X,y)^{\df+2\alpha}} \dd \mu(y)+\int_X \frac{(\varphi(1_X)-\varphi(y))\psi(y)}{d(1_X,y)^{\df+2\alpha}} \dd \mu(y)\\
&= \ell^{\alpha}(\psi) + \int_X \frac{(\varphi(1_X)-\varphi(y))\psi(y)}{d(1_X,y)^{\df+2\alpha}} \dd \mu(y).
\end{align*}
The proof for property (iii) is complete since 
\begin{align*}
\left \lvert \int_X \frac{(\varphi(1_X)-\varphi(y))\psi(y)}{d(1_X,y)^{\df+2\alpha}} \dd \mu(y)\right \rvert &\leq \int_X \frac{|\varphi(1_X)-\varphi(y)|}{d(1_X,y)^{\df+2\alpha}} \dd \mu(y)\\
&\lesssim \Hol_1(\varphi),
\end{align*}
where the latter inequality comes from the Ahlfors regularity estimates in Lemma \ref{lem:Ahlfors_estimates}. Property (iv) holds since $\Da$ has compact resolvent. Finally, the translation $p$-summability of $\ell^{\alpha}$ follows from the discussion below.
\end{proof}

Recalling the notation of Subsection \ref{sec:Crossed_Schatten}, we consider the unique self-adjoint extension $M_{\ell^{\alpha}}$ of the multiplication by $\ell^{\alpha}$ on $\ell^2(\hat{X})$ with domain $C_c(\hat{X})\subset \ell^2(\hat{X})$, as well as the left multiplication operators $\pi(\lambda_{\varphi}):\ell^2(\hat{X})\to \ell^2(\hat{X})$ for $\varphi \in \hat{X}$, which generate $C_r^*(\hat{X})$. Then, from Proposition \ref{prop:length_function} it is clear that $(C_c(\hat{X}), \ell^2(\hat{X}), M_{\ell^{\alpha}})$ is a spectral triple over $C_r^*(\hat{X})$. 

We will compare it to the spectral triple $(\mathbb C \hat{X}, L^2(X,\mu), \Da)$ over $C(X)$ of Theorem \ref{thm:SchattenCQMS}. We note that $\mathbb C \hat{X}\subset \Hol_1(X,d)$, and the $\|\cdot \|_{\infty}$-density of $\mathbb C\hat{X}$ in $C(X)$ immediately gives that, for $p>\max \{p(\alpha,1),1\}$ where $p(\alpha,1)>0$ is given in Proposition \ref{prop:Schatten_com}, the pair $(\mathbb C \hat{X}, L_{\Da,p})$ is an $\Sp$-spectral metric space too, for instance see \cite[Comparison Lemma 1.10]{Rie}.

To this end, let $U:L^2(X,\mu)\to \ell^2(\hat{X})$ be the unitary associated to the Fourier transform. Note that since $U(\varphi)=e_{\varphi}$ for $\varphi \in \hat{X},$ and $\mathbb C \hat{X},C_c(\hat{X})$ are cores for $\Da,M_{\ell^{\alpha}}$ respectively, the unitary $U$ descends to a unitary $U:\Dom \Da \to \Dom M_{\ell^{\alpha}}$ such that $$U\Da U^*=M_{\ell^{\alpha}}.$$ Finally, let $\pi_U:C(X)\to C_r^*(\hat{X})$ denote the unital $*$-isomorphism induced by $U$, which for $\varphi \in \hat{X}$ is given by $$\pi_U(\varphi)=U\m_{\varphi}U^*=\pi(\lambda_{\varphi}).$$ The following then holds.

\begin{cor}\label{cor:dual_length}
The spectral triple $(C_c(\hat{X}),\ell^2(\hat{X}),M_{\ell^{\alpha}})$ is $p$-summable for $p>\frac{\df}{2\alpha}.$ The threshold is sharp. Moreover, for every $p>p(\alpha,1)$ and $T\in C_c(\hat{X})\subset C_r^*(\hat{X})$ it holds that $\cl([ M_{\ell^{\alpha}},T])\in \Sp(\ell^2(\hat{X}))$ and the seminorm $L_{M_{\ell^{\alpha}},p}:C_c(\hat{X}) \to [0,\infty)$ given by $$L_{M_{\ell^{\alpha}},p}(T):= \|\cl([ M_{\ell^{\alpha}},T])\|_{\Sp}$$ is $*$-invariant and $\ker L_{M_{\ell^{\alpha}},p} = \mathbb C \lambda_{1_{\hat{X}}}.$ Also, for $p>\max \{p(\alpha,1),1\}$, the pair $(C_c(\hat{X}), L_{M_{\ell^{\alpha}},p})$ is an $\Sp$-spectral metric space and the homeomorphism 
$$\pi_U^*: (S(C_r^*(\hat{X})),\rho_{L_{M_{\ell^{\alpha}},p}})\to (S(C(X)), \rho_{L_{\Da,p}})$$ is isometric.
\end{cor}

From Theorem \ref{thm:SchattenCQMS}, Corollary \ref{cor:dual_length} and Theorem \ref{thm:CQMS_crossed} we can get concrete CQMS structures on crossed products by Pontryagin duals, like irrational rotation algebras \cite[Example 2.13]{DWilliams}. What is interesting though is that our method allows us to also consider actions of non-finitely generated groups, which is often a challenging task to equip with a CQMS structure, see for instance \cite{AKK}. We illustrate this in the framework of algebraic $\mathbb Z^m$-actions \cite{LS,Sch}, while deferring a detailed treatment to future work. For dynamics related terminology (e.g. expansiveness, mixing) we refer to \cite{KH}.

\subsubsection*{Expansive actions}
Our initial data is a compact abelian group $X$, a invariant metric $d$ and the normalised Haar measure $\mu$ that is Ahlfors regular with respect to $d$. Assume now that $X$ admits a $\mathbb Z^m$-action $\varphi$ by continuous group automorphisms. Further, assume that $\varphi$ is expansive and strongly mixing with respect to $\mu$. Then, from \cite[Lemma 4.5]{LS} the homoclinic subgroup $$X^h(1_X):=\{y\in X: \varphi^{n}(y)\to 1_X,\,\, |n|_m\to \infty \}$$ is a countable dense subgroup of $X$. A remarkable fact is that there is a group isomorphism $j:X^h(1_X)\to \hat{X}.$ 

Let us now choose some $p>\max \{p(\alpha,1),1\}$ and consider the CQMS structures $(\mathbb C \hat{X}, L_{\Da,p})$ and $(C_c(\hat{X}), L_{M_{\ell^{\alpha}},p})$ on $C(X)$ and $C_r^*(\hat{X})$ respectively. Then, it is clear that $\ell^{\alpha}\circ j: X^h(1_X)\to [0,\infty)$ is also a proper translation $p$-summable function (for large enough $p>1$) and $(C_c(X^h(1_X)), L_{M_{\ell^{\alpha}\circ j},p})$ is a CQMS structure on $C_r^*(X^h(1_X)).$ Further, since the translation action of $X^h(1_X)$ on $X$ preserves the domain $\mathbb C \hat{X}$ of the seminorm $L_{\Da,p}$, from Theorem \ref{thm:CQMS_crossed} we obtain the following.

\begin{cor}\label{cor:Z^m}
Let $p>\max \{p(\alpha,1),1\}$. Then, $C(X)\rtimes_r X^h(1_X)$ admits a CQMS structure by $(\mathbb C \hat{X}, L_{\Da,p})$ and $(C_c(X^h(1_X)), L_{M_{\ell^{\alpha}\circ j},p})$ as per Theorem \ref{thm:CQMS_crossed}. 
\end{cor}

The primary potential limitation in applying Corollary \ref{cor:Z^m} to general expansive and strongly mixing $\mathbb Z^m$-actions is that it remains unclear whether compact abelian groups, admitting such actions, can be equipped with an invariant metric for which the Haar measure is Ahlfors regular. Nevertheless, it does not seem unlikely as for $m=1$ the next example provides a plethora of groups with hyperbolic dynamics for which this is true.

\begin{example}\label{exm:Smale}
Let $X$ be a compact abelian group and $\mu$ be the normalised Haar measure. We do not make any choice of metric yet. Let $\varphi:X\to X$ be a continuous group automorphism (hence $\mu$ is $\varphi$-invariant) such that 
\begin{enumerate}[(i)]
\item $\varphi$ is expansive;
\item $\varphi$ has the shadowing property;
\item $\varphi$ is topologically mixing.
\end{enumerate}
We claim that these data are compatible with the requirements of Corollary \ref{cor:Z^m}.

First, observe that (i), (ii) and (iii) are topological conditions. Further, (i) and (ii) imply that the dynamical system $(X,\varphi)$ is topologically conjugate to a Smale space, see \cite{Sa}. Smale spaces are compact metric spaces with homeomorphisms implementing hyperbolic dynamics and a local product structure on $X$, see \cite{Rue}. Then, with (iii) we have that $(X,\varphi)$ satisfies the specification property and hence from Bowen's Theorem \cite[Theorem 20.3.7]{KH} there is a unique $\varphi$-invariant probability measure $\mu_B$ that maximises the topological entropy $h(\varphi)$ of $(X,\varphi)$, and is strongly mixing. 

In addition, from \cite[Proposition 7.16]{Rue} the homoclinic subgroup $X^h(1_X)$ is countable dense in $X$ and so from \cite[Theorem 4.2]{LS} the system $(X,\varphi)$ has completely positive entropy. Then, from \cite[Theorem 22.1 and Theorem 22.4]{Sch} we obtain that $\mu$ maximises $h(\varphi)$. Therefore, we have $\mu=\mu_B$. 

Now from \cite{Ger1} we know that there is a metric $d$ on $X$ for which $\mu_B$ is Ahlfors regular. Following the discussion in \cite[Subsection 4.5]{Ger1} and the references therein, it is immediate to see that $d$ can be chosen invariant by starting its construction by an invariant metric $d'$ and using that $\varphi$ is a group automorphism. 

As a result, for every Smale space $(X,\varphi)$ with an abelian group structure we obtain a CQMS structure on $C(X)\rtimes_r X^h(1_X)$. Interestingly, this crossed product is $*$-isomorphic to the homoclinic groupoid $C^*$-algebra $\mathcal{H}(X,\varphi)$ introduced by Ruelle, see \cite{Put,Ruelle_algebras}. The proof of this is straightforward. First, $C(X)\rtimes_r X^h(1_X)$ is $*$-isomorphic to the reduced groupoid $C^*$-algebra $C^*_r(X\rtimes X^h(1_X))$ of the transformation groupoid $X\rtimes X^h(1_X)$. We refer to \cite{Renault_Book} for details on groupoid $C^*$-algebras. The latter as a set is $X\times X^h(1_X)$ and the groupoid structure on it is defined by the partial multiplication $(x,\gamma)(x\cdot \gamma,\delta):=(x,\gamma \delta),$ inversion $(x,\gamma)^{-1}:=(x\cdot \gamma,\gamma^{-1})$, range $r(x,\gamma):=x$ and source $s(x,\gamma):=x\cdot \gamma.$ The topology on $X\rtimes X^h(1_X)$ is the product topology, and as a topological groupoid $X\rtimes X^h(1_X)$  is canonically isomorphic to the homoclinic groupoid $G^h(X,\varphi)$  for which $\mathcal{H}(X,\varphi):=C^*_r(G^h(X,\varphi))$. Namely, we have $$G^h(X,\varphi):= \{(x,y)\in X\times X: d(\varphi^n(x),\varphi^n(y))\to 0, \,\, |n|\to \infty \},$$ with partial multiplication $(x,y)(y,z):=(x,z),$ inversion $(x,y)^{-1}:=(y,x)$, range $r(x,y):=x$ and source $s(x,y):=y$. The topology on $G^h(X,\varphi)$ is given by graphs of holonomy maps, see \cite[Section 2]{Put}. The aforementioned group structure on $(X,\varphi)$ guarantees that the map $\Phi: X\rtimes X^h(1_X)\to G^h(X,\varphi)$ defined as $$\Phi(x,\gamma):=(x,x\cdot \gamma)$$ is a bijective groupoid homomorphism, and from \cite[Lemma 1.4]{Thomsen} it follows that $\Phi$ is a homeomorphism.

However, a Smale space with a group structure is not necessarily abelian if it is not connected \cite[Theorem 2.4]{Sch}. For instance, consider any finite group $F$ and equip the set $F^{\mathbb Z}$ with the product topology and the left shift map $\sigma$. Also, $F^{\mathbb Z}$ is a compact group with coordinate-wise multiplication. It is then clear that $\sigma$ is a group automorphism and $(F^{\mathbb Z},\sigma)$ is the full $\# F$-shift. Full shifts (and subshifts of finite type in general) are primary examples of Smale spaces. Nevertheless, our result covers several interesting examples, for instance Williams solenoids built from expanding maps on tori of arbitrary dimension and more general Wieler solenoids, see \cite{Wie}.
\end{example}

\begin{Acknowledgements} The authors would like to express their gratitude to Magnus Goffeng, David Kyed, Adam Skalski, Evgeny Verbitskiy and Rik Westdorp for stimulating discussions on the subject of this paper. 
\end{Acknowledgements}


\begin{thebibliography}{99}

\bibitem{AK}
K. Aguilar, J. Kaad; The Podle\'s sphere as a spectral metric space, \textit{J. Geom. Phys.} \textbf{133} (2018), 260\texttt{-}278.

\bibitem{AKKyed}
K. Aguilar, J. Kaad and D. Kyed; The Podle\'s spheres converge to the sphere, \textit{Comm. Math. Phys.} \textbf{392} (2022), 1029\texttt{-}1061. 

\bibitem{ABI}
H. Aimar, A. Bernardis and B. Iaffei; Comparison of Hardy--Littlewood and dyadic maximal functions on spaces of homogeneous type, \textit{J. Math. Anal. Appl.} \textbf{312} (2005), 105\texttt{-}120.

\bibitem{Ar}
C. Arhancet; Sobolev algebras on Lie groups and noncommutative geometry, \textit{J. Noncommut. Geom.} \textbf{18} (2024), 451\texttt{-}500. 

\bibitem{A}
P. Assouad; Plongements Lipschitziens dans $\mathbb R^n$, \textit{Bull. Soc. Math. France} \textbf{111} (1983), 429\texttt{-}448.

\bibitem{AKK}
A. Austad, J. Kaad and D. Kyed; Quantum metrics on crossed products with groups of polynomial growth, \textit{Trans. Amer. Math. Soc.} \textbf{378} (2025), 1939\texttt{-}1973.

\bibitem{BJ}
S. Baaj, P. Julg; Th\'eorie bivariante de Kasparov et op\'erateurs non born\'es dans les $C^*$-modules hilbertiens, \textit{C. R. Acad. Sci. Paris S\'er. I Math.} \textbf{296} (1983), 875\texttt{-}878.

\bibitem{Bar}
M. T. Barlow; Diffusions on fractals, \textit{Lect. Notes Math.} \textbf{1690}, Springer, 1998.

\bibitem{BZ}
I. Bengtsson, K. \.Zyczkowski; \textit{Geometry of Quantum States. An Introduction to Quantum Entanglement}, Cambridge Univ. Press, Cambridge, 2006.


\bibitem{Nek}
N. M. Bon, V.Nekrashevych and T. Zheng; Liouville property for groups and conformal dimension (2023), arXiv:2305.14545.

\bibitem{CKW}
Z.-Q Chen, T. Kumagai and J. Wang; Stability of heat kernel estimates for symmetric non-local Dirichlet forms, \textit{Mem. Amer. Math. Soc.} \textbf{271} (2021), v+89.

\bibitem{C}
M. Christ; A $T(b)$ theorem with remarks on analytic capacity and the Cauchy integral, \textit{Colloq. Math.} \textbf{60/61} (1990) 601\texttt{-}628.

\bibitem{CR}
M. Christ, M. A. Rieffel; Nilpotent group $C^*$-algebras-algebras as compact quantum metric spaces, \textit{Canad. Math. Bull.} \textbf{60} (2017), 77--94.

\bibitem{CI}
E. Christensen, C. Ivan; Sums of two-dimensional spectral triples, \textit{Math. Scand.} \textbf{100} (2007), 35\texttt{-}60.

\bibitem{CGIS}
F. Cipriani, D. Guido, T. Isola and J-L. Sauvageot; Spectral triples for the Sierpinski Gasket, \textit{J. Funct. Anal.} \textbf{266} (2014), 4809\texttt{-}4869.

\bibitem{Connes}
A. Connes; \emph{Noncommutative geometry}, Academic Press, Inc., San Diego, CA, 1994.

\bibitem{Connes2}
A. Connes; Compact metric spaces, Fredholm modules and hyperfiniteness, \textit{Erg. Th. \& Dyn. Sys.} \textbf{9} (1989), 207\texttt{-}220.

\bibitem{Cor}
M. Coornaert; Mesures de Patterson--Sullivan sure le bord d'un espace hyperbolique au sensde Gromov, \textit{Pac. J. Math.} \textbf{159} (1993), 241\texttt{-}270.

\bibitem{DFP}
A. Delf\`in, C. Farsi and J. Packer; $L^p$-spectral triples and $p$-quantum compact metric spaces, arXiv:2411.13735, 2024. 

\bibitem{Fal}
K. J. Falconer; \textit{Fractal Geometry: Mathematical Foundations and Applications}, John Wiley \& Sons, New Jersey, 2014.

\bibitem{FLLVW}
Z. Fan, M. Lacey, J. Li, M. N. Vempati and B. D. Wick; Besov Space, Schatten Classes, and Commutators of Riesz Transforms Associated with the Neumann Laplacian. \textit{Mich. Math. J. Advance Publication} (2024), 1\texttt{-}26.

\bibitem{FLP}
C. Farsi, F. Latr\'emoli\`ere and J. Packer; Convergence of inductive sequences of spectral triples for the spectral propinquity, \textit{Adv. Math.} \textbf{437} (2024), 109442.

\bibitem{Frei}
U. R. Freiberg; Einstein relation on fractal objects, \textit{Discrete Contin. Dyn. Syst. Ser. B} \textbf{17} (2012), 509\texttt{-}525.


\bibitem{FOT}
M. Fukushima, Y. Oshima and M. Takeda; \textit{Dirichlet Forms and Symmetric Markov Processes}, De Gruyter, Berlin, New York, 2010.

\bibitem{GSV}
A. E. Gatto, C. Segovia and S. V{\'a}gi; On fractional differentiation and integration on spaces of homogeneous type, \textit{Rev. Mat. Iberoamericana} \textbf{12} (1996), 111\texttt{-}145.

\bibitem{Ger1}
D. M. Gerontogiannis; Ahlfors regularity and fractal dimension of Smale spaces, \textit{Erg. Th. \& Dyn. Sys.} \textbf{42} (2022), 2281\texttt{-}2332.  

\bibitem{GGM}
D. M. Gerontogiannis, M. Goffeng and B. Mesland; Heat operators and isometry groups of Cuntz-Krieger algebras, arXiv:2406.07416, 2024.

\bibitem{GM}
D. M. Gerontogiannis, B. Mesland; The logarithmic Dirichlet Laplacian on Ahlfors regular spaces, \textit{Trans. Amer. Math. Soc.} \textbf{378} (2024), 651\texttt{-}678.

\bibitem{Gof}
M. Goffeng; Analytic formulas for the topological degree of non-smooth mappings: the odd-dimensional case, \textit{Adv. Math.} \textbf{231} (2012), 357\texttt{-}377.

\bibitem{GMR}
M. Goffeng, B. Mesland and A. Rennie; Untwisting twisted spectral triples, \textit{Int. J. Math.} \textbf{30} (2019), 1950076.

\bibitem{GK}
I. C. Gohberg, M. G. Krein; Introduction to the Theory of Linear Nonselfadjoint Operators in Hilbert Space, \textit{Amer. Math. Soc.}, Providence, 1969.

\bibitem{GS}
P. G\'orka, A. S\l{}abuszewski; Embeddings of the fractional Sobolev spaces on metric-measure spaces, \textit{Nonlinear Anal.} \textbf{221} (2022), 112867.

\bibitem{GHL}
A. Grigor’yan, J. Hu and K. S Lau; Heat kernels on metric-measure spaces and an application to semi-linear elliptic equations, \textit{Trans. Amer. Math. Soc.} \textbf{355} (2003), 2065\texttt{-}2095.

\bibitem{HSZW}
A. Hawkins, A. Skalski, S. White and J. Zacharias; On spectral triples on crossed products arising from equicontinuous actions, \textit{Math. Scand.} \textbf{113} (2013), 262\texttt{-}291.

\bibitem{Hei}
J. Heinonen; \textit{Lectures on Analysis on Metric Spaces}, Springer, New York, 2001.

\bibitem{HCT}
J. S. Helman, A. Coniglio and C. Tsallis; Fractons and
the fractal structure of proteins, \textit{Phys. Rev. Lett.} (1984) \textbf{53}, 1195\texttt{-}1197.

\bibitem{HR}
M. Hendrick, P. Renard; Fractal dimension, walk dimension and
conductivity exponent of karst networks around Tulum, \textit{Front.
Phys.} \textbf{4} (2016), 00027.

\bibitem{HKT}
M. Hinz, D. J. Kelleher and A. Teplyaev; Metrics and spectral triples for Dirichlet and resistance forms, \textit{J. Noncommut. Geom.} \textbf{9} (2015), 359\texttt{-}390.

\bibitem{Hor}
P. Ho\v{r}ava, Spectral dimension of the universe in quantum gravity at a Lifshitz point, \textit{Phys. Rev. Lett.} \textbf{102} (2009), 161301.

\bibitem{HK}
T. Hyt\"onen, A. Kairema; Systems of dyadic cubes in a doubling metric space, \textit{Colloq. Math.} \textbf{126} (2012), 1\texttt{-}33.

\bibitem{JW}
A. Jonsson, H. Wallin; \textit{Function Spaces on Subsets of $\mathbb R^n$}, Mathematical Reports, Vol. 2, Part 1. Academic Publishers, Harwood, 1984. 

\bibitem{KLPW}
A. Kairema, J. Li, M. C. Pereyra and L. A. Ward; Haar bases on quasi-metric measure spaces, and dyadic structure theorems for function spaces on product spaces of homogeneous type, \textit{J. Funct. Anal.} \textbf{271} (2016), 1793\texttt{-}1843. 

\bibitem{KH}
A. Katok, B. Hasselblatt; \textit{Introduction to the modern theory of dynamical systems}, Cambridge Univ. Press, Cambridge, 1995.

\bibitem{Kig}
J. Kigami; Harmonic analysis for resistance forms, \textit{J. Funct. Anal.} \textbf{204} (2003), 399\texttt{-}444.

\bibitem{Kl}
M. Klisse; Crossed products as compact quantum metric spaces, \textit{Canad. J. of Math.} (to appear), arXiv:2303.17903, 2023. 

\bibitem{LafSal}
V. Lafforgue, M. De la Salle; Noncommutative $L^p$-spaces without the completely bounded approximation property, \textit{Duke Math. J.} \textbf{160}, 71\texttt{-}116.


\bibitem{LLL}
T. Landry, M. Lapidus and F. Latr\'emoli\`ere; Metric approximations of the spectral triple on the Sierpinki gasket and other fractals, \textit{Adv. Math.} \textbf{385} (2021), 107771.

\bibitem{Lat}
F. Latr\'emoli\`ere; Curved noncommutative tori as Leibniz compact quantum metric spaces, \textit{J. Math. Phys.} \textbf{56} (2015), 123503.

\bibitem{LS}
D. Lind, K. Schmidt; Homoclinic points of algebraic $\mathbb Z^d$-actions, \textit{J. Amer. Math. Soc.}, \textbf{12} (1999), 953\texttt{-}980.

\bibitem{MT}
J. M. Mackay, J. T. Tyson; \textit{Conformal dimension. Theory and application}, Amer. Math. Soc., Providence, RI, 2010.

\bibitem{Nah}
A. Nahmod; Generalized uncertainty principles on spaces of homogeneous type, \textit{J. Funct. Anal.} \textbf{119} (1994), 171\texttt{-}209.


\bibitem{NPV}
E. Di Nezza, G. Palatucci and E. Valdinoci; Hitchhiker’s guide to the fractional Sobolev spaces, \textit{Bull. Sci. Math.} \textbf{136} (2012), 521\texttt{-}573.

\bibitem{OR}
N. Ozawa, M. A. Rieffel; Hyperbolic group $C^*$-algebras and free product $C^*$-algebras as compact quantum metric spaces, \textit{Canad. J. of Math.} \textbf{57} (2005), 1056\texttt{-}1079.

\bibitem{Pisier}
G. Pisier; \textit{Similarity problems and completely bounded maps}, vol. \textbf{1618} of Lecture Notes in Mathematics. Springer--Verlag, Berlin, 2001. 

\bibitem{Put}
I. F. Putnam; $C^*$-algebras from Smale spaces, \textit{Canad. J. Math.} \textbf{48} (1996), 175\texttt{-}195.

\bibitem{RT}
R. Rammal, G. Toulouse; Random walks on fractal structures and
percolation clusters, \textit{J. Phys. Lett.} \textbf{44} (1983), 13.

\bibitem{Renault_Book}
J. Renault; \textit{A groupoid approach to $C^*$-algebras}, Springer-Verlag, New York, 1980.

\bibitem{Rie}
M. A. Rieffel; Metrics on states from actions of compact groups, \textit{Doc. Math.} \textbf{3} (1998), 215\texttt{-}229.

\bibitem{Rie2}
M. A. Rieffel; Metrics on state spaces, \textit{Doc. Math.} \textbf{4} (1999), 559\texttt{-}600.

\bibitem{Rie3}
M. A. Rieffel; Compact quantum metric spaces, \textit{Operator algebras, quantization, and noncommutative geometry, Contemp. Math.} \textbf{365}, Amer. Math. Soc., Providence, RI, 2004.

\bibitem{RS}
R. Rochberg, S. Semmes; Nearly weakly orthonormal sequences, singular value estimates, and Calder\'on--Zygmund operators, \textit{J. Funct. Anal.} \textbf{86} (1989), 237\texttt{-}306.

\bibitem{Ruelle_algebras}
D. Ruelle; Noncommutative algebras for hyperbolic diffeomorphisms, \textit{Invent. Math.} \textbf{93} (1988), 1\texttt{-}13.

\bibitem{Rue}
D. Ruelle; \textit{Thermodynamic Formalism}, Cambridge Univ. Press, Cambridge, 2004.

\bibitem{Ryan}
R. A. Ryan; \textit{Introduction to Tensor Products of Banach Spaces}, Springer--Verlag, London, Berlin, Heidelberg, 2002.

\bibitem{Sa}
K. Sakai; Shadowing properties of $L$-hyperbolic homeomorphisms, \textit{Topology Appl.} \textbf{112} (2001), 229--243.

\bibitem{Sch}
K. Schmidt; \textit{Dynamical Systems of Algebraic Origin}, Birkh\"auser/Springer Basel AG, Basel, 1995.


\bibitem{Thomsen}
K. Thomsen; $C^*$-algebras of homoclinic and heteroclinic structure in expansive dynamics, \textit{Mem. Amer. Math. Soc.} \textbf{206} (2010), x+122.

\bibitem{T}
H. Triebel; \textit{Fractals and Spectra. Monographs in Mathematics}, Birkh\"auser, Basel, Boston, 1997.
Berlin.

\bibitem{Vil}
C. Villani; \textit{Optimal Transport. Old and New,} Springer--Verlag, Berlin, Heidelberg, 2009. 

\bibitem{Wie}
S. Wieler; Smale spaces via inverse limits, \textit{Erg. Th. \& Dyn. Sys.} \textbf{34} (2014), 2066\texttt{-}2092.
2014.

\bibitem{DWilliams}
D. P. Williams; \textit{Crossed products of $C^*$-algebras}, Amer. Math. Soc., Providence, RI, 2007.


\end{thebibliography}
\end{document}